\documentclass[12pt,reqno]{amsart}
\usepackage{amsmath,amsfonts,amsthm,amssymb}
\usepackage[usenames,dvipsnames]{xcolor}
\usepackage[T1]{fontenc}
\usepackage{enumitem}
\usepackage{hyperref}
\hypersetup{
     colorlinks   = true,
     citecolor    = blue,
     linkcolor    = blue
}
\usepackage[width=6.5in, left=0.95in, right=1.05in, height=8.8in, top=1.1in,bottom=1.1in]{geometry}
\usepackage{color,soul}
\usepackage{verbatim}
\usepackage{pdfsync}

\newcommand{\ot}{[0,t]}
\newcommand{\ott}{[0,T]}

\newcommand{\1}{{\bf 1}}

\newcommand{\R}{\mathbb R}

\newcommand{\be}{\mathbf{E}}

\newcommand{\bp}{\mathbf{P}}

\newcommand{\ca}{\mathcal A}

\newcommand{\cd}{\mathcal D}

\newcommand{\cf}{\mathcal F}

\newcommand{\ci}{\mathcal I}

\newcommand{\cn}{\mathcal N}

\newcommand{\cv}{\mathcal V}

\newcommand{\cz}{\mathcal Z}

\newcommand{\XX}{\mathfrak X}
\newcommand{\HH}{\mathfrak H}

\newcommand{\al}{\alpha}

\newcommand{\ep}{\varepsilon}

\newcommand{\ga}{\gamma}
\newcommand{\gga}{\Gamma}
\newcommand{\ka}{\kappa}
\newcommand{\la}{\lambda}
\newcommand{\laa}{\Lambda}
\newcommand{\om}{\omega}
\newcommand{\oom}{\Omega}

\newcommand{\si}{\sigma}

\newcommand{\vp}{\varphi}

\newcommand{\lp}{\left(}
\newcommand{\rp}{\right)}
\newcommand{\lc}{\left[}
\newcommand{\rc}{\right]}
\newcommand{\lcl}{\left\{}
\newcommand{\rcl}{\right\}}
\newcommand{\lln}{\left|}
\newcommand{\rrn}{\right|}

\allowdisplaybreaks[4]
\newtheorem{theorem}{Theorem}[section]

\newtheorem{corollary}[theorem]{Corollary}

\newtheorem{definition}[theorem]{Definition}

\newtheorem{lemma}[theorem]{Lemma}

\newtheorem{problem}[theorem]{Problem}
\newtheorem{proposition}[theorem]{Proposition}
\theoremstyle{remark}
\newtheorem{remark}[theorem]{Remark}

\let\Section=\section
\def\section{\setcounter{equation}{0}\Section}

\def\RR{\mathbb{R}}

\def\EE{\mathbb{E}}

\def\cH{{\cal H}}

\def\la{{\lambda}}
\def\si{{\sigma}}

\def \eref#1{\hbox{(\ref{#1})}}

\begin{document}

\title[Multiplicative stochastic heat equation]
{Stochastic heat equation\\
with rough dependence in space}

\author[Y. Hu, J. Huang, K. L\^e, D. Nualart, S. Tindel]{Yaozhong Hu
 \and Jingyu Huang \and Khoa L\^e \and David Nualart \and Samy Tindel}

\address{Yaozhong Hu, Jingyu Huang, Khoa L\^e and David Nualart: Department of Mathematics, University of Kansas, 405 Snow Hall, Lawrence, Kansas, USA.}
\email{yhu@ku.edu, huangjy@ku.edu, khoale@ku.edu,  nualart@ku.edu}

\address{Samy Tindel: Institut {\'E}lie Cartan,
Universit\'e de Lorraine, B.P. 239,
54506 Vand{\oe}u\-vre-l{\`e}s-Nancy, France.}
\email{samy.tindel@univ-lorraine.fr}
\thanks{Y. Hu is partially supported by a grant from the Simons Foundation \#209206}
\thanks{D. Nualart was supported by the NSF grant  DMS1208625 and the ARO grant FED0070445}
\thanks{S. Tindel is member of the BIGS (Biology, Genetics and Statistics) team at INRIA}


\begin{abstract}
This paper studies  the nonlinear one-dimensional stochastic heat equation
 driven by a Gaussian noise which  is white in time and which
  has the covariance of a fractional Brownian motion with Hurst parameter
 $H \in \left( \frac 14, \frac 12 \right)$ in the space variable.  The existence and uniqueness of the solution $u$ are
 proved assuming the nonlinear coefficient $\sigma(u)$  is differentiable with a Lipschitz derivative and $\sigma(0)=0$.
 In the  case of a multiplicative noise, that is, $\sigma(u)=u$,
 we derive the Wiener chaos expansion of the solution and a Feynman-Kac formula
 for the moments of the solution. These results allow us to
 establish sharp  lower and upper  asymptotic bounds for $\be[u^n(t,x)]$.
\end{abstract}

\maketitle

\tableofcontents

\def\hbp{{\cH(\beta,p)}}
\def\norm{\mathcal{N}}
\def\sp{{\XX^{p,\theta}_\varepsilon}}
\section{Introduction}

In this paper we are interested in the one-dimensional stochastic  partial differential equation
\begin{equation}\label{spde with sigma}
\frac{\partial u}{\partial t}=\frac{\kappa}{2}\frac{\partial^2 u}{\partial x^2}+\sigma(u)\dot{W}\,, \quad t\ge 0, \quad x\in\mathbb{R}\,,
\end{equation}
  where $W$ is a centered Gaussian process with covariance given
by
\begin{equation}  \label{cov}
\be \lc W(s,x)W(t,y)\rc=
\frac{1}{2}\lp |x|^{2H}+|y|^{2H}-|x-y|^{2H}\rp \, (s\wedge t),
\end{equation}
with $\frac 14<H<\frac{1}{2}$.  That is, $W$ is a standard Brownian motion in time and a
fractional Brownian motion with Hurst parameter $H$  in  the space variable and $\dot{W} =\frac {\partial ^2W}{\partial t \partial x}$.
For this stochastic heat equation with a rough noise in space, understood
in the It\^o sense, our aim is twofold: on   one hand, for a  differentiable coefficient
 $\si$ with a Lipschitz derivative and satisfying $\sigma(0)=0$, we will obtain the existence and uniqueness
 of the solution. On the other hand, we shall  further investigate the special relevant case
  $\si(u)=u$. We now detail those two main points.

\noindent
\textbf{(1)}
Since the pioneering work by Peszat-Zabczyk \cite{PZ} and Dalang (see \cite{Dal}), there has been a lot of interest in stochastic partial differential equations driven by a Brownian motion in time with spatial  homogeneous covariance.  After more than a decade of investigations, the standard assumptions on $W$ under which existence and uniqueness hold take the following form

\noindent
\textit{(i)} $\be [ \dot{W}(s,x)\dot{W}(t,y)]=\Lambda(x-y) \, \delta_0(s-t)$, where $\Lambda$ is a positive distribution of positive type.

\noindent
\textit{(ii)} The Fourier transform of the spatial covariance $\Lambda$ is a tempered measure $\mu$ that satisfies the integrability condition $\int_{\mathbb{R}} \frac {\mu(d\xi)}{ 1+|\xi|^2} <\infty$.

\noindent
In case of the covariance \eqref{cov} under consideration, one can easily compute the measure $\mu$, whose explicit expression is  $\mu(d\xi)= c_{1,H}|\xi|^{1-2H}d\xi$, where $c_{1,H}$ is a constant depending on $H$ (see expression \eqref{eq:expr-c1H} below). In addition, it is readily checked that $\mu$ fulfills the condition $\int_{\mathbb{R}} \frac {\mu(d\xi)}{ 1+|\xi|^2} <\infty$ for all $H\in(0,1)$. However, the corresponding covariance $\Lambda$ is a distribution which fails to be positive when $H<\frac 12$,  and the covariance of two stochastic integrals with respect to
$\dot{W}$ is expressed in terms of fractional derivatives. For this reason, the standard methodology used in the classical references \cite{Dal,DQ,PZ} to handle homogeneous spatial covariances does not apply to our case of interest.

In a recent paper, Balan, Jolis and Quer-Sardanyons \cite{BJQ} proved the existence of a unique mild
 solution for equation (\ref{spde with sigma}) in the case $\sigma(u)=au+b$, using techniques of
 Fourier analysis.  The method used in \cite{BJQ} cannot be extended to general nonlinear coefficients.
 Indeed, the isometry property of stochastic integrals with respect to $W$ involves the semi norm
 \[
 \cn_{\frac 12-H,2}u(t,x)  = \lp\int_\RR  \be | u(t,x+h) -u(t,x)|^2 |h|^{2H-2} dh\rp^{\frac 12},
  \]
  where $\cn_{\beta,p}$ is defined in \eref{e4}.
 Then, if $u$ and $v$ are two solutions, $\cn_{\frac 12-H,2}(\sigma(u)-\sigma(v))$ cannot be bounded in terms of $\mathcal{N}_{\frac12-H,2}( u -v)$,
 due to the presence of a double increment of the form $\sigma(u(s, z+h)) -\sigma(v(s,z+h)) - \sigma(u(s,z))+ \sigma (v(s,z))$. To overcome this difficulty we shall use a truncation argument to show the uniqueness of mild solutions, inspired by the work of Gy\"ongy and Nualart in \cite{Gyo2} on  the stochastic Burgers equation on the whole real line driven by a space-time white noise. The main ingredient is a uniform estimate of the $L^p(\Omega)$-norm of a stochastic convolution (see Lemma \ref{lem2}). Due to this argument, the uniqueness is obtained in the space $\cz_T^p$ (see \eref{eq:dcp-norm-ZTp} for the definition of the norm in $\cz_T^p$), which requires an integrability condition in the space variable.

 The existence of a solution is much  more involved. The methodology, inspired by the work of Gy\"ongy in \cite{Gyo} on semi-linear stochastic partial differential equations, consists in taking approximations obtained by regularizing the noise and using a  compactness argument on  a suitable space of trajectories, together with the  strong uniqueness result.

Once existence and uniqueness are obtained, we establish the H\"older continuity of the solution $u$ in both space and time variables.
 We also derive   upper bounds for the moments of the solution using a sharp Burkholder's inequality, as well as the matching lower bounds for the second moment by means of a Sobolev embedding argument.
 Summarizing, we get a complete basic picture of the solution to equation \eqref{spde with sigma}  in the case $\frac 14 < H< \frac 12$. The critical parameter $H=\frac 14$ is worthwhile noting, since it is also the threshold under which rough differential equations driven by a fractional Brownian motion are ill-defined.

\noindent
\textbf{(2)}
The particular case $\sigma(u)=u$ in equation \eqref{spde with sigma} deserves a special  attention. Indeed, this linear equation turns out to be a continuous version of the parabolic Anderson model, and is related to challenging systems in random environment like KPZ equation \cite{Ha,BeC} or polymers~\cite{AKQ,BTV}. The localization and intermittency properties of the linear version of \eqref{spde with sigma} have thus been thoroughly studied for equations driven by a space-time white noise (see \cite{Kh} for a nice survey), while a recent trend consists in extending this kind of result to equations driven by very general Gaussian noises \cite{Ch14,HHNT,HN,HNS}.

Nevertheless, the rough noise $W$ with covariance \eqref{cov} presented here is not covered by the aforementioned references, and we wish to fill this gap. We will thus particularize our setting to $\si(u)=u$, and first go back to the existence and uniqueness problem. Indeed, in this linear case, one can implement a rather simple procedure involving Fourier transform, as well as a chaos expansion technique, in order to achieve existence and uniqueness of the solution to~\eqref{spde with sigma}. Since this point of view is interesting in its own right and short enough, we develop it at Section \ref{sec:anderson-exist-uniq}. Moreover in this case we can consider more general initial conditions.

We then move to a Feynman-Kac type representation for the solution: following the approach introduced in \cite{HN,HHNT},  we obtain an explicit formula for the kernels of the Wiener chaos expansion and we show its convergence. In fact, we  cannot  expect  a Feynman-Kac formula for the solution, because the  covariance   is rougher than the space-time white noise case, and this type of formula requires smoother covariance structures  (see, for instance,  \cite{HNS}). However,  by means of Fourier analysis techniques as in \cite{HN,HHNT}, we are able to obtain a Feynman-Kac formula for the moments that involves a fractional derivative of the Brownian local time.

Finally, the previous considerations allow us to handle, in the last section of the paper,  the intermittency properties of the solution.   More precisely,  we show  sharp lower bounds for the
moments of the solution of the form $\be [u(t,x)^n]\ge\exp(C n^{1+\frac 1H} t)$, for all $t\ge 0$, $x\in \R $ and $n\ge 2$. These bounds entail the intermittency phenomenon and match the corresponding estimates for the case $H>\frac 12$ obtained in \cite{HHNT}.

The paper is organized as follows. Section \ref{sec:preliminaries} contains some preliminaries on stochastic integration with respect to the noise $W$ and elements of Malliavin calculus.  Section \ref{sec:momentest} deals with basic moment estimates and H\"older continuity properties of stochastic convolutions. We establish the uniqueness of the solution in Section \ref{sec:existence}. To do  this, first we derive moment estimates for the supremum norm in space and time for the stochastic convolution. In order to show the existence, we need to introduce several spaces of functions in Subsection \ref{sub:space_time_function_spaces} and derive compactness criteria.  Section \ref{sec:anderson-exist-uniq} deals with the parabolic Anderson model, that is, the case $\sigma(u)=u$. In Section \ref{sec:Anderson.momentbounds}, we derive Feynman-Kac type formulas for the moments of the solution which allow us to derive sharp lower and upper moment estimates and intermittency properties.

\section{Preliminaries}\label{sec:preliminaries}
In this section we introduce some of the functional spaces we will deal with in the remainder of the paper, as well as some general Malliavin calculus tools.

\subsection{Noise structure and stochastic integration}

Our noise $W$ can be seen as a Brownian motion with values in an infinite dimensional Hilbert space. One might thus think that the stochastic integration theory with respect to $W$ can be handled by classical theories  (see e.g \cite{BP,Dal,DPZ}). However, the spatial covariance function of $W$, which is formally equal to $H(2H-1) |x-y|^{2H-2}$,  is not   locally integrable when  $H<1/2$  (in other words, the Fourier transform  of $|\xi|^{1-2H}$ is not a function), and $W$ thus lies outside the scope of application of these classical references. Due to this fact, we provide some details about the construction of a stochastic integral with respect to our noise.

Let us start by introducing our basic notation on Fourier transforms
of functions. The space of   Schwartz functions is
denoted by $\mathcal{S}$. Its dual, the space of tempered distributions, is $\mathcal{S}'$.  The Fourier
transform of a function $u \in \mathcal{S}$ is defined with the normalization
\[ \mathcal{F}u ( \xi)  = \int_{\mathbb{R}} e^{- i
   \xi  x } u ( x) d x, \]
so that the inverse Fourier transform is given by $\mathcal{F}^{- 1} u ( \xi)
= ( 2 \pi)^{- 1} \mathcal{F}u ( - \xi)$.

 Let $ \mathcal{D}((0,\infty)\times \R)$ denote the space  of real-valued infinitely differentiable functions with compact support on $(0, \infty) \times \R$.
Taking into account the spectral representation of the covariance function of the fractional Brownian motion in the case $H<\frac 12$
proved in \cite[Theorem 3.1]{PT}, we represent  our noise $W$   by a zero-mean Gaussian family $\{W(\vp) ,\, \vp\in
\mathcal{D}((0,\infty)\times \R)\}$ defined on a complete probability space
$(\Omega,\cf,\bp)$, whose covariance structure
is given by
\begin{equation}\label{eq:cov1}
\be\lc W(\vp) \, W(\psi) \rc
=  c_{1,H}\int_{\R_{+}\times\R}
\cf\varphi(s,\xi) \, \overline{\cf\psi(s,\xi)} \, |\xi|^{1-2H} \, ds  d\xi,
\end{equation}
where the Fourier transforms $\cf\varphi,\cf\psi$ are understood as Fourier transforms in space only and
\begin{equation}\label{eq:expr-c1H}
c_{1,H}= \frac 1 {2\pi} \Gamma(2H+1)\sin(\pi H)  \,.
\end{equation}

The inner product appearing in (\ref{eq:cov1}) can be expressed in terms of fractional derivatives. Let $\beta$ be in $(0,1)$. The Marchaud fractional derivative $D_-^{\beta}$ of order $\beta$ with respect to the space variable is defined, for a function $\varphi : \RR_+\times \RR \rightarrow \RR$, as follows
\begin{equation}\label{eq:def-frac-deriv}
D^{\beta}_- \varphi(s,x) = \lim_{\varepsilon \to 0}D_{-,\varepsilon}^{\beta}\varphi(s,x)\,,
\end{equation}
where
\begin{equation*}
D_{-,\varepsilon}^{\beta}\varphi(s,x) = \frac{\beta}{\gga(1-\beta)}
\int_{\varepsilon}^{\infty} \frac{\varphi(s,x) - \varphi(s, x+y)}{y^{1+\beta}} \, dy\,.
\end{equation*}
We also define the fractional integral of order $\beta$ of a function $\psi: \RR_+\times \RR \rightarrow \RR$ by
\begin{equation*}
I_-^{\beta}\psi(s,x)=\frac{1}{\Gamma(\beta)}\int_{x}^\infty \psi(s,u)(x-u) ^{\beta-1}du\,.
\end{equation*}
 Note again that here the fractional differentiation and integration are only with respect to space variables.
Observe  that  if  $\varphi =I ^{\beta} _{-}  \psi $  for some $\psi \in L^2( \RR_+\times \RR)$, then by  Theorem 6.1 in~\cite{SKM}
 we have
\[
D^{\beta}_- \varphi = D^{\beta} _- (I^{\beta}_- \psi)= \psi
\]
and, hence,
\[
\int_{\RR_+ \times \RR} \left[D^{\beta}_- \varphi(s,x)\right]^2  dsdx= \int_{\RR_+ \times \RR}  \psi^2(s,x) dsdx <\infty.
\]

The previous notions can be related to our noise in the following way: it is known  (cf. \cite{PT} for further details) that
\begin{equation}  \label{eq1}
\be\lc W(\vp) \, W(\psi) \rc
=   c_ {2,H} \int_{\RR_+\times \RR}D_-^{\frac{1}{2}-H}\varphi(s,x)D_-^{\frac{1}{2}-H}\psi(s,x)dsdx,
\end{equation}
where
\begin{equation}\label{eq:def-c2H}
c_{2,H}     = \left[ \Gamma\left(H+\frac 12\right)  \right]^2  \left( \int_0^{\infty} \left( (1+s)^{H-\frac{1}{2}}-s^{H-\frac{1}{2}}\right)^2 ds +\frac{1}{2H}\right)^{-1}\,.
\end{equation}
for any $\varphi, \psi \in \mathcal{D}((0,\infty)\times \R)$.

Based on the previous observation and relation (\ref{eq1}), we introduce a new set of function spaces. Indeed, let  $\HH$ be the class of functions  $\varphi : \RR_+ \times \RR \rightarrow \RR $ such that  there exists $ \psi \in L^2(\RR_+\times \RR)  $ satisfying $\varphi(s,x)=I_-^{\frac12- H}\psi(s,x) $. The relation between $\HH$ and our noise $W$  is given in the following proposition.

\begin{proposition} \label{prop: H}
The class of functions $\HH $ is a Hilbert space equipped with the inner product
\begin{equation}\label{eq: H inner prod}
\langle \varphi,\psi\rangle_{ \HH}:=\ c_{2 ,  H }\int_{\RR_+\times \RR}D_-^{\frac 12-H }\varphi(s,x)D_-^{\frac 12-H}\psi(s,x)dsdx\,,
\end{equation}
  and  $ \mathcal{D}((0,\infty)\times \R)$ is dense in  $\HH$. Moreover if $\HH_0$ denotes the  class of functions   $\varphi  \in L^2( \RR_+\times \RR)$ such that $ \int_{\RR_+\times \RR} |\mathcal{F}\varphi(s,\xi)|^2|\xi|^{1-2H}d\xi ds < \infty $, then $ \HH_0$ is not complete  and the inclusion
  $\HH_0 \subset \HH$ is strict. Also for any $\varphi,\psi \in \HH_0$,
  \begin{equation}\label{eq: H_0 element H prod}
  \langle\varphi, \psi \rangle_{ \HH}=c_{1, H}\int_{\RR_+\times \RR}\mathcal{F}\varphi(s,\xi)\overline{\mathcal{F}\psi(s,\xi)}|\xi|^{1-2H }d\xi ds\,.
  \end{equation}
\end{proposition}
We refer to \cite{PT} for the proof of this  proposition. Note that in \cite{PT}, the functions considered there are from $\RR$ to $\RR$, but by scrutinizing the proofs we see that the results of this paper can be easily  extended to our case, i.e. for functions from $\RR_+\times \RR$ to $\RR$. We omit the details.

Let us now identify our space $\HH$ with another classical space in harmonic analysis. Indeed, according to Proposition 1.37 in \cite{BCD}, for any $\beta\in (0,1)$  the homogeneous Sobolev space $\dot{H}^\beta$ can be defined as the completion of the space of infinitely differentiable functions with compact support with respect to the norm
\begin{equation}  \label{E1}
\|f\|^2_{\dot{H} ^{\beta}}= \int_{\RR}|D_-^{\beta}f (x)|^2 dx= c^2_{3,\beta} \int_\RR\int_\RR|f(x+y)-f(x)|^2|y|^{-1-2 \beta}dxdy\,,
	\end{equation}
	where  $c_{3,\beta}^2= (1/2-\beta)\beta \, c_{2, \frac12-\beta}^{-1}$ and $c_{2, \frac12-\beta}$ is defined by \eqref{eq:def-c2H}.
	As a consequence, our Hilbert space $\HH$ can be identified with the homogenous Sobolev space of order $\beta= \frac 12-H$ of functions with values in $L^2(\R_+)$. Namely $\HH= \dot{H} ^{\frac 12- H} (L^2(\R_+))$, and for any $f  \in \HH$ the quantity $\|f\|_{\HH}$ can be represented as
	\[
	\| f \|_{\HH}^2= c^2_{3, \frac 12-H} \int_{\R_+} \int_\RR\int_\RR |f(s,x+y)-f(s,x)|^2|y|^{2H-2}dxdyds.
	\]
 From Proposition  \ref{prop: H}, we see that the Gaussian family $W$ can be extended as an isonormal  Gaussian process $W=\{W(\phi), \phi\in  \HH\}$  indexed by  the Hilbert space $\HH$.

Let us now turn to the stochastic integration with respect to $W$. Since we are handling a Brownian motion in time, one can start by integrating elementary processes.

\begin{definition}\label{def:elementary-process}
For any $t\ge0$, let $\mathcal{F}_{t}$ be the $\sigma$-algebra generated by $W$ up to time $t$. An elementary process $g$ is a process given by
\begin{equation*}
g(s,x)
=
\sum_{i=1}^{n} \sum_{j=1}^m X_{i,j} \, \1_{(a_{i},b_{i}]}(s) \, \1_{(h_j,l_{j}]}(x),
\end{equation*}
where $n$ and $m$ are finite positive integers, $-\infty<a_{1}<b_{1}<\cdots<a_{n}<b_{n}<\infty$, $h_j<l_{j}\ $    and $X_{i,j}$ are $\cf_{a_{i}}$-measurable random variables for $i=1,\ldots,n$. The integral of  such a process with respect to $W$ is defined as
\begin{eqnarray}
\int_{\mathbb{R}_+}\int_{\mathbb{R}}g(s,x) \, W(ds,dx)
&=&\sum_{i=1}^{n} \sum_{j=1}^m X_{i,j} \, W\lp \1_{(a_{i},b_{i}]} \otimes \1_{(h_j,l_{j}]}\rp   \label{eq:riemann-sums-W}\\
&=&\sum_{i=1}^{n} \sum_{j=1}^m X_{i,j} \,\big[W(b_{i},l_{j})-W(a_{i},l_{j}) -W(b_i,h_j)+ W(a_i, h_j)\big]\,. \nonumber
\end{eqnarray}

\end{definition}

We can now extend the notion of integral with respect to $W$ to a broad class of adapted processes.

\begin{proposition}\label{prop:intg-wrt-W}
Let $\laa_{{H}}$ be the space of predictable processes $g$ defined on $\R_{+}\times\R$ such that
almost surely $g\in \HH$ and $\be[\|g\|_{\HH}^{2}]<\infty$. Then, we have:

\noindent
\emph{(i)}
The space of elementary processes defined in Definition \ref{def:elementary-process} is dense in $\laa_{H}$.

\noindent
\emph{(ii)}
For $g\in\laa_{H}$, the stochastic integral $\int_{\mathbb{R}_+}\int_{\mathbb{R}}g(s,x) \, W(ds,dx)$ is defined as the $L^{2}(\oom)$-limit of Riemann sums along elementary processes approximating $g$, and we have
\begin{equation}\label{int isometry}
\be\lc \lp \int_{\mathbb{R}_+}\int_{\mathbb{R}}g(s,x) \, W(ds,dx) \rp^{2} \rc
=
\be \left[ \|g\|_{\HH}^{2}\right].
\end{equation}
\end{proposition}

\begin{proof}
Let us prove item (i). To this aim, consider $g \in \Lambda _{H}$ and set $\varphi(t,x)=D_-^{\frac 12 -H}g(t,x)$.
According to the definition of $\laa_{H}$, we have  $\be[\int_{\mathbb{R}_+}\int_{\mathbb{R}}|\varphi(s,x)|^2dxds] < \infty$. Then we will show that $g(t,x)$ can be approximated by elementary processes in $L^2(\Omega;\HH)$ in three steps.

\noindent
{\it Step 1}. Recall that we have set $\dot{H}^{\frac 12- H}$ for the class of functions $f$,  such that there exists $h \in L^2(\RR)$ satisfying $f=I_-^{1/2-H}h$. We show that the process $g$ can be approximated in $L^2(\Omega;\HH)$ by functions of the form
\begin{equation} \label{zz}
\psi_m(s,x;\omega)=\sum_{i=1}^m{\bf 1}_{(a_i,b_i]}(s)\phi_i(x;\omega)\,,
\end{equation}
where for each $i$, $\phi_i(x;\omega)$ is an $\mathcal{F}_{a_i}$-measurable $L^2(\Omega;\dot{H}^{\frac 12- H})$-valued random field. To see this, we just set
\begin{equation*}
\psi_m(s,x;\omega)=\sum_{k=1}^{m2^m} {\bf 1}_{((k-1)2^{-m}, k2^{-m}]}(s)2^m\int_{(k-1)2^{-m}}^{k2^{-m}}g(r,x;\omega)dr\,,
\end{equation*}
and we easily get  that $ D^{\frac 12 -H}_-\psi_m(s,x;\omega)\to D^{\frac 12 -H}_-g(s,x;\omega)$  in $L^2(\Omega\times\R_+ \times \R)$ as $m$ tends to infinity. In this way we get the desired approximation.

\noindent
{\it Step 2}.  We show that each $\psi_m(s,x;\omega)$ of the form  (\ref{zz}) can be approximated in $L^2(\Omega; \HH)$ by a linear combination of elements of the form $X{\bf 1}_{(a,b]}(s) h(x)$. Indeed, for each $\phi_i(x;\omega)$, we notice that since
\begin{equation*}
\be \int_{\RR}|D_-^{\frac{1}{2}-H}\phi_i(x;\omega)|^2 dx < \infty\,,
\end{equation*}
the random function $D_-^{\frac{1}{2}-H}\phi_i(x;\omega)$ can be approximated in $L^2(\Omega; L^2(\RR))$ by functions of the form $\sum_{j=1}^N X_j h_j(x)$, where each $X_j$ is an $\mathcal{F}_{a_i}$-measurable random variable and each $h_j$ is an element in $L^2(\RR)$. Thus, it is easily seen that $\phi_i(x;\omega)$ can be approximated by a sequence of functions  of  the form
\begin{equation*}
\sum_{j=1}^N X_j I_-^{\frac{1}{2}-H}h_j(x)\,.
\end{equation*}
So we conclude that $\psi_m(s,x;\omega)$ can be approximated in $L^2(\Omega; \mathfrak{H})$ by
\begin{equation*}
\sum_{i=1}^m {\bf 1}_{(a_i,b_i]}(s)\sum_{j=1}^N X_{i,j} I_-^{\frac{1}{2}-H}h_{i,j}(x)\,,
\end{equation*}
 where for each $(i,j)$, $X_{i,j}$ are $\mathcal{F}_{a_i}$-measurable random variables and $h_{i,j} \in L^2(\R)$.

\noindent
{\it Step 3}.  Owing to Theorem 3.3 in
\cite{PT} we know that
\begin{equation*}
\text{Span}\lcl
D_-^{\frac{1}{2}-H}\1_{(h,l]},\,  h<l
\rcl
\end{equation*}
is dense in $\Lambda_0:= \{D_-^{\frac{1}{2}-H}f: f \in \dot{H}^{\beta}\}$, in $L^2(\RR)$ norm. This observation and the results in Step 2 immediately show that $\psi_m(s,x;\omega)$ can be approximated by elementary processes in $L^2(\Omega; \HH)$.  This completes the proof.
 \end{proof}

With this stochastic integral defined, we are ready to state the definition of the solution to equation \eref{spde with sigma}.
\begin{definition}\label{def-sol-sigma}
Let $u=\{u(t,x), 0 \leq t \leq T, x \in \mathbb{R}\}$ be a real-valued predictable stochastic process  such that for all $t\in[0,T]$ and $x\in\R$ the process $\{p_{t-s}(x-y)\sigma(u(s,y)) \1_{[0,t]}(s), 0 \leq s \leq t, y \in \mathbb{R}\}$ is an element of $\laa_{H}$, where $p_t(x)$ is the heat kernel on the real line related to $\frac{\ka}{2}\Delta$. We say that $u$ is a mild solution of \eref{spde with sigma} if for all $t \in [0,T]$ and $x\in \mathbb{R}$ we have
\begin{equation}\label{eq:mild-formulation sigma}
u(t,x)=p_t u_0(x) + \int_0^t \int_{\mathbb{R}}p_{t-s}(x-y)\sigma(u(s,y)) W(ds,dy) \quad a.s.,
\end{equation}
where the stochastic integral is understood in the sense of Proposition \ref{prop:intg-wrt-W}.
\end{definition}

\subsection{Elements of Malliavin calculus}

We recall   that the Gaussian family $W$ can be extended to $\HH$ and this produces an isonormal Gaussian process, where $\HH$ is the Hilbert space  introduced in Proposition \ref{prop: H}. We refer to~\cite{Nua}
for a detailed account of the Malliavin calculus with respect to a
Gaussian process. On our Gaussian space, the  smooth and cylindrical
random variables $F$ are of the form
\begin{equation*}
F=f(W(\phi_1),\dots,W(\phi_n))\,,
\end{equation*}
with $\phi_i \in \HH$, $f \in C^{\infty}_p (\R^n)$ (namely $f$ and all
its partial derivatives have polynomial growth). For this kind of random variable, the derivative operator $D$ in the sense of Malliavin calculus is the
$\HH$-valued random variable defined by
\begin{equation*}
DF=\sum_{j=1}^n\frac{\partial f}{\partial
x_j}(W(\phi_1),\dots,W(\phi_n))\phi_j\,.
\end{equation*}
The operator $D$ is closable from $L^2(\Omega)$ into $L^2(\Omega;
\HH)$  and we define the Sobolev space $\mathbb{D}^{1,2}$ as
the closure of the space of smooth and cylindrical random variables
under the norm
\[
\|DF\|_{1,2}=\sqrt{\be [F^2]+\be [\|DF\|^2_{\HH}  ]}\,.
\]
We denote by $\delta$ the adjoint of the derivative operator (or divergence) given
by the duality formula
\begin{equation}\label{dual}
\be  \lc \delta (u)F \rc =\be  \lc \langle DF,u
\rangle_{\HH}\rc ,
\end{equation}
for any $F \in \mathbb{D}^{1,2}$ and any element $u \in L^2(\Omega;
\HH)$ in the domain of $\delta$.

For any integer $n\ge 0$ we denote by $\mathbf{H}_n$ the $n$th Wiener chaos of $W$. We recall that $\mathbf{H}_0$ is simply  $\R$ and for $n\ge 1$, $\mathbf {H}_n$ is the closed linear subspace of $L^2(\Omega)$ generated by the random variables $\{ H_n(W(\phi)),\phi \in \HH, \|\phi\|_{\HH}=1 \}$, where $H_n$ is the $n$th Hermite polynomial.
For any $n\ge 1$, we denote by $\HH^{\otimes n}$ (resp. $\HH^{\odot n}$) the $n$th tensor product (resp. the $n$th  symmetric tensor product) of $\HH$. Then, the mapping $I_n(\phi^{\otimes n})= H_n(W(\phi))$ can be extended to a linear isometry between    $\HH^{\odot n}$ (equipped with the modified norm $\sqrt{n!}\| \cdot\|_{\HH^{\otimes n}}$) and $\mathbf{H}_n$.

Consider now a random variable $F\in L^2(\Omega)$ which is measurable with respect to the $\sigma$-field  $\cf$ generated by $W$. This random variable can be expressed as
\begin{equation}\label{eq:chaos-dcp}
F= \be \lc F\rc + \sum_{n=1} ^\infty I_n(f_n),
\end{equation}
where the series converges in $L^2(\Omega)$, and the elements $f_n \in \HH ^{\odot n}$, $n\ge 1$, are determined by $F$.  This identity is called the Wiener-chaos expansion of $F$.

The Skorohod integral (or divergence) of a random field $u$ can be
computed by  using the Wiener chaos expansion. More precisely,
suppose that $u=\{u(t,x) , (t,x) \in \R_+ \times\R\}$ is a random
field such that for each $(t,x)$, $u(t,x)$ is an
$\cf_t$-measurable and square integrable random  variable.
Then, for each $(t,x)$ we have a Wiener chaos expansion of the form
\begin{equation}  \label{exp1}
u(t,x) = \be \lc u(t,x) \rc + \sum_{n=1}^\infty I_n (f_n(\cdot,t,x)).
\end{equation}
Suppose   that $\be [\|u\|_{ \HH}^{2}]$ is finite.
Then, we can interpret $u$ as a square  integrable
random function with values in $\HH$ and the kernels $f_n$
in the expansion (\ref{exp1}) are functions in $\HH
^{\otimes (n+1)}$ which are symmetric in the first $n$ variables. In
this situation, $u$ belongs to the domain of the divergence operator (that
is, $u$ is Skorohod integrable with respect to $W$) if and only if
the following series converges in $L^2(\Omega)$
\begin{equation}\label{eq:delta-u-chaos}
\delta(u)= \int_0 ^\infty \int_{\R^d}  u(t,x) \, \delta W(t,x)
= W(\be [u]) + \sum_{n=1}^\infty I_{n+1} (\widetilde{f}_n(\cdot,t,x)),
\end{equation}
where $\widetilde{f}_n$ denotes the symmetrization of $f_n$ in all its $n+1$ variables. We note that whenever $u\in \laa_{H}$ the integral $\delta(u)$ coincides with the It\^o integral.

\medskip
Along the paper we denote by $C$ a generic constant that may vary from line to line.

\section{Moment estimates and H\"older continuity of stochastic convolutions}
\label{sec:momentest}

This section is devoted to a thorough study of the stochastic convolution related to our noise $\dot{W}$, including moment bounds and H\"older continuity estimates.

\subsection{Moment bound of the solution} First we introduce some notation, which makes  some of our  formulae easier to read, and which will prevail until the end of the article. Let
 $(B,\|\cdot\|)$ be a Banach space equipped with the norm $\|\cdot\|$, and let $\beta\in(0,1)$ be a fixed number. For every  function $f:   \RR\to B$, we introduce the function $ \cn_\beta^{B} f :    \RR \to[0,\infty]$ defined by
 	\begin{equation}  \label{e1}
 	 	\cn_\beta^{B}  f(x)=\left( \int_ {\RR } \| f( x+ h)- f( x)\|^2|h|^{-1-2 \beta}dh \right)^{\frac12}\,.
 	\end{equation}
When $B=\R$, we abbreviate the notation $\cn _\beta^{\R} f$ into $\cn _\beta f$.
 With this notation, the norm of the homogeneous  Sobolev space $\dot{H} ^\beta$  can be written as $c_{3,\beta} \|  \cn _\beta f\|_{L^2 (\RR)}$. The following technical lemma will be used along the paper.

\begin{lemma} \label{lem1}
For any   $\beta \in (0,1 )$,
\[
 \int_{\RR}  [ \cn_\beta p_s(x)]^2 dx \le C_\beta (\kappa s)^{-\frac 12-\beta}\,.
 \]
\end{lemma}

\begin{proof}
Recalling that  $\cf p_{s}(\xi)=e^{-\frac \kappa 2  s\xi^2}$ and invoking Plancherel's identity we can write		 \begin{align*}
		 \int_{\RR}  [ \cn_\beta p_s(x)]^2  dx&= \int_ {\RR } \int_{\RR}  | p_s( x+ h)-  p_s( x)|^2|h|^{-1-2 \beta}dh  dx \\
		  &=    \int_{\RR}\int_{\RR}e^{-\kappa s|\xi|^2}|e^{- i \xi z} -1|^2 |z|^{-1-2\beta} dz d\xi\\
		   &= C_{1,\beta}  \int_{\RR} e^{-\kappa s\xi^2} |\xi|^{  2\beta} d\xi\,,
		\end{align*}
where the second relation is obtained by a scaling $v\equiv \xi \, z$ in the integral in $z$
and $C_{1,\beta}=  \frac 1{2\pi}\int_{\RR}  |e^{iv} -1|^2 |v|^{-1-2\beta} dv$, which is easily seen to be a convergent integral. Setting now $\eta=(\kappa s)^{1/2}\xi$ in the integral in $\xi$, we get
\begin{equation*}
 \int_{\RR}  [ \cn_\beta p_s(x)]^2  dx
\le
C_\beta  (\kappa s)^{-\frac 12-\beta} \, ,
\end{equation*}
where $C_\beta = C_{1,\beta}\int_{\RR} e^{-\eta^2} |\eta| ^{2\beta} d\eta$.
\end{proof}

The transformation $\cn_{\beta}^B$ can also be defined for functions $f$ defined on $\RR_+\times \RR$ acting on the spacial variable,  and in this case, $\cn_{\beta}^Bf : \RR_+\times \RR \to [0,\infty]$.
Now fix $p\ge 2$, and suppose that $f=\{ f(t,x), t\ge 0, x\in \RR\}$ is a random field such that $\be |f(t,x)|^p <\infty$ for all $(t,x)$. Then we can consider $f$ as an $L^p(\Omega)$-valued function and we will  denote   by $\cn_{\beta,p} f $ the transformation introduced in (\ref{e1}) for $B=L^p(\Omega)$, that is,
 	\begin{equation}   \label{e4}
 	 	\cn_{\beta,p} f (t,x)=\left( \int_ {\RR } \| f(t, x+ h)- f(t, x)\| _{L^p(\Omega)}^2|h|^{-1-2 \beta}dh \right)^{\frac12}\,.
 	\end{equation}

With the above notation in mind, the following proposition is essential in our approach.  	
\begin{proposition}		\label{prop.burkholder}
Let $W$ be the Gaussian noise defined by the covariance \eqref{eq:cov1}, and consider  a predictable random field $f\in \Lambda_H$.  Then, for any $p\ge 2$ we have
 \begin{equation}\label{eq:ineq-burk-2}
\left \|\int_0^t\int_\RR f(s,y)W(ds,dy) \right\|_{L^p(\Omega)}
 \le
\sqrt{4p}  c_{3, \frac 12-H} \lp \int_{0}^{t} \int_{\RR}  [\cn_{\frac 12-H, p} f(s,y)]^2 dyds
 \rp^{\frac 12},
\end{equation}
where $c_{3, \beta}$ is defined by relation \eqref{E1}.
	\end{proposition}
	
	\begin{proof}
		Applying Burkholder's inequality, we have
		\begin{equation}\label{eq:burk-simple}
			\left\|\int_0^t\int_\RR f(s,y)W(ds,dy)\right\|_{L^p(\Omega)}\le\sqrt{4p}\left\|\int_0^t\left\|f(s,\cdot)\right\|_{\dot{H}^{\frac 12- H}}^2ds \right\|_{L^{\frac p2}(\Omega)}^{\frac12}\,.
		\end{equation}
Moreover, using \eref{E1} we can write
\begin{equation} \label{e2}
\|f(s,\cdot)\|^2_{\dot{H}^{\frac 12-H}  }
=c_{3,\frac 12-H} ^2
   \int_{\R^{2}} \left| f(s,y+h) -f(s,y)  \right|^{2}
|h|^{2H-2}  \, dh dy .
\end{equation}
We now invoke Minkowski's  inequality, under the form
$$
\left\|\int_{S} U(\xi) \mu(d\xi)\right\|_{L^{q}(\oom)}
\le \int_{S} \|U(\xi)\|_{L^{q}(\oom)} \mu(d\xi),
$$
for a measure $\mu$ on the state space $S$. Together with (\ref{e2}),  this yields
\begin{align*}
\left\|\int_0^t\left\|f(s,\cdot)\right\|_{\dot{H}^{\frac 12- H}}^2ds \right\|_{L^{\frac p2}(\Omega)}^{\frac12}
&\le
 c_{3,\frac 12-H} \int_{\R^{2}} \left\| \lp f(s,y+h) -f(s,y) \rp^{2} \right\|_{L^{\frac{p}{2}}(\Omega)}
|h|^{2H-2}  \, dh dy \\
&=
 c_{3,\frac 12-H}  \int_{\R^{2}} \left\|  f(s,y+h) -f(s,y)  \right\|_{L^{p}(\Omega)}^{2}
|h|^{2H-2}  \, dh dy,
\end{align*}
from which identity \eqref{eq:ineq-burk-2} is easily deduced.
	\end{proof}

From now on, we fix a finite time horizon $T$. We introduce the following functions space which plays an important role through out the paper. 
\begin{definition} Let $\mathfrak{X}^\beta_T(B)$ be the space of all continuous functions $f:[0,T]\times \RR\to B$ such that
 \[
 \|f \|_{\mathfrak{X}^\beta_ T(B)}:=\sup_{t\in [0,T], x\in\RR} \|f(t,x)\|+\sup_{t\in [0,T], x\in\RR} \cn_{\beta}^B f(t,x) <\infty,
 \]
 where we recall that $\cn_{ \beta}^B$ is defined by \eqref{e1}.
\end{definition}
	We equip $\mathfrak{X}^\beta_ T(B)$ with the norm $\|\cdot\|_{\mathfrak{X}^\beta_ T(B)}$ defined above. Then $\mathfrak{X}^\beta_{T}(B)$ is a normed vector space. In fact, the following proposition states that $\mathfrak{X}^\beta_ T(B)$ is complete.
		\begin{proposition}\label{prop.XBanach}
		$\mathfrak{X}^\beta_T(B)$ is a   Banach space.
	\end{proposition}
	\begin{proof}
		Let $\{f_n\}$ be a Cauchy sequence in $\mathfrak{X}^\beta_T(B)$. Since the space $C_{b}([0,T]\times\RR;B)$ of bounded continuous functions from $[0,T]\times \RR$ to $B$ is complete, there exists a bounded continuous function $f:[0,T]\times\RR\to B$ such that
		\begin{equation*}
			\lim_{n\to\infty}\sup_{t\in[0,T], \,x\in \RR}\|f_n(t,x)-f(t,x)\|=0\,.
		\end{equation*}
		For any $\varepsilon>0$  there exists $n_0>0$ such that
		\begin{equation*}
			\sup_{x\in \RR}\cn_{\beta}^B (f_n-f_m)(t,x)< \varepsilon
		\end{equation*}
		for all $m,n\ge n_0$. It follows from Fatou's lemma that
		\begin{equation*}
			\cn_\beta^B(f_n-f)(t,x)\le\liminf_{m\to \infty} \cn_\beta^B (f_n-f_m)(t,x)\le\varepsilon
		\end{equation*}
		for every $t\in[0,T],\,x\in \RR$ and $n\ge n_0$. This implies that $\lim_{n\to\infty}\sup_{t\le T,x\in\RR} \cn_\beta^B(f_n-f)(t,x)=0$ which means $f_n$ converges to $f$ in $\XX^\beta_T(B)$.
	\end{proof}

When $B=L^p(\Omega)$ with $p\in[1,\infty)$, we use the notation $\mathfrak{X}^{\beta,p}_T=\mathfrak{X}^{\beta}_T(L^p(\Omega)) $. A function $f$ in $\XX^{\beta,p}_T$ can be considered as a stochastic process indexed by $(t,x)$ in $[0,T]\times\RR$ such that
	\begin{equation*}
		\sup_{t\in [0,T],\, x\in \RR}\|f(t,x)\|_{L^p(\Omega)}+\sup_{t\in [0,T],x\in \RR}\left(\int_\RR\|f(t,x+y)-f(t,x)\|_{L^p(\Omega)} ^2|y|^{-2 \beta-1}dy \right)^{\frac12}<\infty\,.
	\end{equation*}
Next, for  $\theta>0$, $\varepsilon>0$ and $\beta\in (0,1)$, we consider the following norm on $\mathfrak{X}^{\beta,p}_{T}$
	\begin{equation}\label{norm.xTe}
		\|u\|_{\mathfrak{X}^{\beta,p}_{T, \theta , \varepsilon} }
		:=\sup_{t\in[0,T], x \in \RR}e^{-\theta t}\|u(t,x)\|_{L^p(\Omega)}
		+ \varepsilon \sup_{t\in[0,T], x\in \RR} e^{-\theta t}  \cn_{ \beta,p}u(t,x)\,,
	\end{equation}
where we recall that $\cn_{ \beta,p}$ is defined by \eqref{e4}.   

\begin{remark}\label{rmk:norm-X-beta-p}
  (i) In the case $\varepsilon=1$, we simply write $\| \cdot \|_{\mathfrak{X}^{\beta,p}_{T,\theta}}$.

  \noindent
 (ii) The second term in the norm in \eqref{norm.xTe} is not invariant by scaling while the first term is. Indeed, denote $f_ \lambda(t,x)=f(t,\lambda x)$, then
		\begin{multline*}
			\sup_{x\in\RR}\left(\int_\RR\|f_ \lambda(t,x+h)-f_ \lambda(t,x)\|^2_{L^p(\Omega)}|h|^{-1-2\beta}dh\right)^{\frac12} \\
			=\lambda^{\beta} \sup_{x\in\RR}\left(\int_\RR\|f(t,x+h)-f(t,x)\|^2_{L^p(\Omega)}|h|^{-1-2\beta}dh\right)^{\frac12}\,.
		\end{multline*}
		This is the very reason why various orders of $(t-s)$ appear in the proof of Proposition \ref{prop.young} below. We bypass this technical difficulty by the introduction of an additional scaling factor $\varepsilon$ in \eqref{norm.xTe}.

\noindent
		(iii) Another way to see the role of $\varepsilon$ is via dimensional analysis. Suppose that the amplitude of $f$ has unit $L$, the spatial variable $x$ has unit $S$, while the randomness $\omega$ is dimensionless. Then the first term in \eqref{norm.xTe} has unit $L$ while the second term has unit $L/S^{\beta}$. Hence, in order for the two terms to have the same dimension, we multiply the second term with a constant $\varepsilon$ having unit of $S^\beta$.
		
\noindent
(iv) Because $T$ is finite, the norm $\|\cdot\|_{\mathfrak{X}^{\beta,p}_{T, \theta , \varepsilon}}$ defined as above is equivalent to the norm $\|\cdot\|_{\mathfrak{X}^{\beta,p}_{T}}$.

%		When $\beta=\frac 12-H$ we simply write  $\| \cdot \|_{\mathfrak{X}^{p}_{\theta}, \varepsilon}$ %instead of
	%	$\|\cdot \|_{\mathfrak{X}^{\beta, p}_{\theta}, \varepsilon}$.

  \end{remark}

The next proposition gives a convenient bound on the stochastic convolution in term of the spaces $\mathfrak{X}^{\beta,p}_{T}$.
	\begin{proposition}\label{prop.young}
	Consider  a predictable random field $f\in \Lambda_H$ and define process $\{\Phi(t,x), \, t\ge 0, x\in\R\}$   by
	\begin{equation}\label{eq:def-A-t-x}
\Phi(t,x)=\int_0^t\int_\RR p_{t-s}(x-y)f(s,y)W(ds,dy)\,.
\end{equation}
Then, for any $  \beta<H$ and $p\ge 2$,  the following inequality holds:
\begin{eqnarray}  \nonumber
\left\|\Phi\right\|_{\mathfrak{X}^{\beta,p}_{T,\theta , \varepsilon}}& \le&  C_0\sqrt{p}\|f\|_{\mathfrak{X}^{\frac 12-H, p}_{T,\theta, \varepsilon}} \\ \label{k1}
&& \times  \left( \kappa^{\frac H2 -\frac12}\theta^ {-\frac H2} +   \kappa ^{-\frac 14- \frac \beta 2}
\theta^{  \frac \beta 2- \frac 14}+\varepsilon^{-1}\kappa^{-\frac 14}\theta^{-\frac 14} +\varepsilon  \kappa^{\frac H2- \frac \beta 2 -\frac12}\theta^{\frac \beta 2-\frac H2}\right)\,,
		\end{eqnarray}
		where $C_0$ is a  constant depending only on $H$ and $\beta$.
\end{proposition}

\begin{remark}
According to relation \eqref{k1}, the stochastic convolution induces some stability properties in the spaces $\mathfrak{X}^{\beta,p}_{\theta , \varepsilon}$ whenever $\frac12-H<\beta<H$. This imposes the restriction $H>\frac14$ already at this stage.
\end{remark}

\begin{proof}[Proof of Proposition \ref{prop.young}]
According to our definition \eqref{norm.xTe}, we have $\|\Phi\|_{\mathfrak{X}^{\beta,p}_{T,\theta, \varepsilon}}=\ca_{1}+\ep\ca_{2}$, with
\begin{equation*}
\ca_{1}=\sup_{t\in[0,T], x \in \RR}e^{-\theta t}\|\Phi(t,x)\|_{L^p(\Omega)},
\quad\text{and}\quad
\ca_{2}= \sup_{t\in[0,T], x\in \RR} e^{-\theta t} \cn_{\beta,p}\Phi(t,x).
\end{equation*}
We now estimate those terms separately. Along the proof $C$ will denote a generic constant depending only on $H$ and $\beta$.

\noindent
\textit{Step 1: Upper bound for $\ca_{1}$.}
The term $\Phi(t,x)$ is of the form
\begin{equation*}
\int_{0}^{t} \int_{\R} g_{t,x}(s,y)W(ds,dy),
\quad\text{with}\quad
g_{t,x}(s,y) = p_{t-s}(x-y)f(s,y).
\end{equation*}
Applying inequality \eqref{eq:ineq-burk-2}, we thus have
\begin{equation*}
\|\Phi(t,x)\|_{L^{p}(\Omega)}
\le C \sqrt{p} \,
\lp \int_{0}^{t} \int_{\R^{2}} \left\| g_{t,x}(s,y+h) -g_{t,x}(s,y)  \right\|_{L^{p}(\Omega)}^{2} |h|^{2H-2}
\, dh \, dy \, ds \rp^{\frac 12}.
\end{equation*}
A simple decomposition of the increment $g_{t,x}(s,y+h) -g_{t,x}(s,y)$ then yields
\begin{equation*}
\|\Phi(t,x)\|_{L^{p}(\Omega)}
\le C \sqrt{ p}  \left[
\lp \int_{0}^{t} J_{1}(s)  ds\rp^{\frac 12}  + \lp \int_{0}^{t} J_{2}(s) \, ds \rp^{\frac 12}  \right]\, ,
\end{equation*}
where
		\begin{equation*}
			J_{1}(s)=\int_\RR\int_\RR |p_{t-s}(x-y-z)-p_{t-s}(x-y)|^2\|f(s,y+z)\|_{L^p(\Omega)} ^2|z|^{2H-2}dydz
		\end{equation*}
		and
		\begin{equation*}
			J_{2}(s)=\int_\RR\int_\RR p^2_{t-s}(x-y)\|f(s,y+z)-f(s,y)\|_{L^p(\Omega)}^2 |z|^{2H-2}dydz\,.
		\end{equation*}
To estimate $J_{1}(s)$, we write
\[
J_1(s) \le \sup_{x \in \RR} \| f(s,x)\|^2_{L^p(\Omega)}   \int_{\RR}  [\cn_{\frac 12-H } p_{t-s}(y)]^2 dy.
\]
Applying Lemma \ref{lem1} with $\beta=\frac 12-H$, we obtain
\begin{equation*}
J_{1}(s)
\le
C  \sup_{x \in \RR} \| f(s,x)\|^2_{L^p(\Omega)} [\kappa (t-s)]^{H-1} \, .
\end{equation*}

Let us now turn to   estimate   $J_{2}(s)$. Recalling our notation \eqref{e4}, we have
\begin{eqnarray}\label{est.f2}
J_{2}(s)&=&
\int_{\R} p_{t-s}^{2}(x-y) [\cn_{\frac 12-H,p}f(s,y)]^2 \, dy
\le
\sup_{x \in \RR} [\cn_{\frac 12-H,p}f(s,x)]^2 \int_{\R} p_{t-s}^{2}(x-y)  \, dy \notag \\
&\le&
[2\pi \kappa (t-s)]^{-\frac12} \, \sup_{x \in \RR} [\cn_{\frac 12-H,p}f(s,x)]^2 .
\end{eqnarray}
Hence, putting together our bounds on $J_{1}$ and $J_{2}$,  we get
\begin{multline*}
e^{-\theta t}\sup_{x\in\RR}\|\Phi(t,x)\|_{L^p(\Omega)}
\le
C \sqrt {p} \sup_{\substack{0\le s\le T\\  x \in \RR}} e^{-\theta s} \|f(s,x)\|_{L^p(\Omega)} \left (\int_0^t e^{-2\theta (t-s)} [\kappa (t-s)]^{H-1} ds \right)^{\frac12} \\
+ C \sqrt{p} \, \varepsilon \sup_{\substack{0\le s\le T\\  x \in \RR}} e^{-\theta s}
     \sup_{x \in \RR}  \cn_{\frac 12-H,p}f(s,x)  \,
     \frac{\left (\int_0^t e^{-2 \theta (t-s)}
     [\kappa (t-s)]^{-\frac 12} ds \right)^{\frac 12}}{\varepsilon} ,
\end{multline*}
and some elementary computations for the integrals above yield
\begin{equation*}
\ca_{1} = \sup_{t\in[0,T] ,  x\in\RR} e^{-\theta t}\ \|\Phi(t,x)\|_{L^p(\Omega)}
\le
C \sqrt{p} \|f\|_{\mathfrak{X}^{\frac 12-H,p}_{T,\theta, \varepsilon}} (\kappa^{\frac H2-\frac 12}\theta^{ -\frac H2}+ \varepsilon ^{-1}\kappa^{-\frac14}\theta^{-\frac14})\,.
\end{equation*}

\noindent
\textit{Step 2: Upper bound for $\ca_{2}$.}
According to the definition of $\ca_{2}$, we have to bound $\cn_{\beta,p}\Phi(t,x)$, where we recall that
\begin{equation}\label{eq:recall-N-beta-p-Phi}
\cn_{\beta,p}\Phi(t,x) = \left(\int_{\R} \|\Phi(t,x+h) - \Phi(t,x) \|_{L^{p}(\oom)}^{2} |h|^{-1-2\beta} dh\right)^{\frac12} .
\end{equation}
Furthermore, arguing as in Step 1 above, it is easily seen that
\begin{equation}\label{est.Gxh}
\|\Phi(t,x+h) - \Phi(t,x) \|_{L^{p}(\oom)}
\le C\sqrt{p}\left[
\lp \int_0^t  J_{1}^{\prime}(s,h) \, ds \rp^{1/2}+ \lp \int_0^t(J_{2}^{\prime}(s,h)) \, ds \rp^{1/2}\right],
\end{equation}
where
\begin{align*}
			J_{1}^{\prime}(s,h)=\int_\RR\int_\RR |p_{t-s}(x+h-y-z)-p_{t-s}(x-y-z)-&p_{t-s}(x+h-y)+p_{t-s}(x-y)|^2
			\\&\times \|f(s,y+z)\|_{L^p(\Omega)} ^2|z|^{2H-2}dydz\,,
		\end{align*}
		and
		\begin{align*}
			J_{2}^{\prime}(s,h)=\int_\RR\int_\RR |p_{t-s}(x+h-y)-p_{t-s}(x-y)|^2\|f(s,y+z)-f(s,y)\|_{L^p(\Omega)} ^2|z|^{2H-2}dydz\,.
		\end{align*}
Plugging \eqref{est.Gxh} into \eqref{eq:recall-N-beta-p-Phi}, we end up with
\begin{equation*}
\cn_{\beta,p}\Phi(t,x)
\le C\sqrt{p}\left[
\int_{0}^{t}\int_{\RR} J_{1}^{\prime}(s,h) |h|^{-1-2 \beta} \, dh \, ds
+
\int_{0}^{t}\int_{\RR} J_{2}^{\prime}(s,h) |h|^{-1-2 \beta} \, dh \, ds\right].
\end{equation*}
In addition, arguing again as in the proof of Lemma \ref{lem1}, we can show that
		
\[
\int_{\RR} J_{1}^{\prime}(s,h) |h|^{-1-2 \beta} dh \le C [\kappa (t-s)]^{ H-\beta-1} \sup_{x \in \RR} \|f(s,x)\|^2_{L^p(\Omega)}.
\]
On the other hand, applying Lemma  \ref{lem1} leads to
\[
\int_{\RR}  J_{2}^{\prime}(s,h) |h|^{-1-2 \beta} dh  \le  C [\kappa(t-s)]^{-\frac 12-\beta} \sup_{x \in \RR} [\cn_{\frac 12-H,p}f(s,x)]^2 \,.
\]
		Combining these estimates for $J_1'$, $J_2'$ and resorting to \eqref{est.Gxh}, similarly as the estimate for $e^{-\theta t}\|\Phi(t,x)\|_{L^p(\Omega)}$, we obtain
\begin{equation*}
\ca_{2}
\le
C \sqrt{p}  \left(\|f\|_{\mathfrak{X}^{\frac 12-H ,p}_{T,\theta, \varepsilon}} \kappa^{\frac H2- \frac \beta 2 -\frac 12} \theta^{ \frac \beta 2-\frac H2}
+  \varepsilon ^{-1}  \|f\|_{\mathfrak{X}^{\frac 12-H, p}_{T,\theta, \varepsilon}}  \kappa^{-\frac 14-\frac \beta 2}\theta^{\frac \beta 2 -\frac 14} \right)\,.
\end{equation*}
Putting together Step 1 and Step 2, our claim \eref{k1} is now easily checked.
	\end{proof}
	
We conclude this section by a simple remark which is labeled for further use. In the particular case $\beta =\frac12-H$, and using the simplified notation
  $\| \cdot \|_{\mathfrak{X}^{\frac12-H, p}_{T,\theta, \varepsilon}}=
   \| \cdot \|_{\mathfrak{X}^{  p}_{T,\theta, \varepsilon}}$, the estimate  \eref{k1} can be written as
   \begin{equation}\label{est.normXX}
\left\|\Phi\right\|_{\mathfrak{X}^{ p}_{T,\theta , \varepsilon}} \le C_0\sqrt{p}\|f\|_{\mathfrak{X}^{  p}_{T,\theta, \varepsilon}} \left( \kappa^{\frac H2-\frac12}\theta^{-\frac H2} + \varepsilon^{-1}\kappa^{-\frac 14}\theta^{-\frac 14} +\varepsilon \kappa^{H-\frac34}\theta^{\frac 14-H} \right)\,.
		\end{equation}

\subsection{H\"older continuity estimates}

A natural question arising from the definition \eqref{eq:def-A-t-x} of the process $\Phi$ is the derivation of H\"older type exponents in both time and space. Some estimates in this direction are provided in the next proposition. We set  $\mathfrak{X}^{p}_{T}=\mathfrak{X}^{\frac 12-H, p}_{T}$, and the norm $\| \cdot \|_{ \mathfrak{X}^{p}_{T,\theta_0}}$ is given by  \eref{norm.xTe} with $\ep=1$ and $\beta =\frac 12-H$.

\begin{proposition}\label{prop:holder.est}
Recall that the noise $W$ is given by the covariance \eqref{eq:cov1}. Consider $p\ge2$ and a predictable random field $f\in\mathfrak{X}^{p}_{T}$, where $T$ is a fixed finite time horizon. Let $\theta_0$ be any positive number.  We define the random field $\Phi$ as in \eqref{eq:def-A-t-x}. Then for every $x,h\in\RR$, $t_1,t_2\in[0,T]$ and every $\ga\in[0,H]$ we have
		\begin{equation}\label{eq:joint Holder}
			\|\Phi([t_1,t_2],x+h)-\Phi([t_1,t_2],x)\|_{L^p(\Omega)}\leq C \sqrt{p} e^{\theta_0 T} \|f\|_{\mathfrak{X}^{p}_{T,\theta_0}} |t_2-t_1|^{\frac{ H- \ga}{2}}|h|^{ \ga } \,.
		\end{equation}
		In the above, the   constant $C$ depends on $T$ and does not depend on $p$, and we are using the notation
	\begin{equation*}
	\Phi([t_1,t_2],x)=\Phi(t_2,x)-\Phi(t_1,x)\,.
	\end{equation*}
In particular, if we let $t_1 = 0$, we get the H\"older estimate of the space variable.  For the H\"older estimate of the time variable, we have
\begin{equation}\label{eq:time Holder}
\|\Phi(t_2,x)-\Phi(t_1,x)\|_{L^p(\Omega)} \leq C \sqrt{p} e^{\theta_0 T} \|f\|_{\mathfrak{X}^{p}_{T,\theta_0}} |t_2-t_1|^{\frac{H}{2}}\,.
\end{equation}
	\end{proposition}
	\begin{proof}
		First we prove \eref{eq:joint Holder}. Without loss of generality, we assume $t_1<t_2$ and denote $\Delta t=t_2-t_1$.
		We also set:
\begin{equation}\label{eq:def-V1-V2}
V_1(f)=\sup_{t\in[0,T]}\sup_{x\in\RR}\|f(t,x)\|_{L^p(\Omega)}\,,
\quad\quad
	V_2(f)=\sup_{t\in[0,T]}\sup_{x\in\RR}  \cn_{H,p}f(t,x),
\end{equation}
		and $V(f)=V_1(f)+V_2(f)$. Observe that according to \eqref{norm.xTe}, we have $V(f)\le \exp(\theta_{0} T) \|f\|_{\mathfrak{X}^{p}_{T,\theta_0}}$.
		
		As in the proof of Proposition \ref{prop.young},
		we first write $\Phi([t_1,t_2],x+h)-\Phi([t_1,t_2],x)=\ca_1+\ca_2$, where
		\begin{equation*}
		 	\ca_1=\int_0^{t_1}\int_\RR [p_{[t_1-s,t_2-s]}(x+h-y)-p_{[t_1-s,t_2-s]}(x-y)]f(s,y)W(ds,dy)\,,
		\end{equation*}
		and
		\begin{equation*}
			\ca_2=\int_{t_1}^{t_2}\int_\RR [p_{t_2-s}(x+h-y)-p_{t_2-s}(x-y)]f(s,y)W(ds,dy)\,.
		\end{equation*}
We now treat those two terms separately. To alleviate notation we will include the $\sqrt{p}$ into the constant $C$ below. 

\noindent
\textit{Step 1: Upper bound for $\ca_{1}$.}		The computations are carried out analogously to the proof of Proposition~\ref{prop.young}, and we have
		\begin{align*}
			\|\ca_1\|_{L^p(\Omega)}^2\leq C \int_0^{t_1}(A_{11}(s)+A_{12}(s))ds,
		\end{align*}
		where $A_{11}$ and $A_{12}$ are analogous to $J_{1},J_{2}$ in the proof of Proposition~\ref{prop.young}, and are respectively defined by
		\begin{multline*}
			A_{11}(s)=\int_\RR\int_\RR |p_{[t_1-s,t_2-s]}(x+h-y-z)-p_{[t_1-s,t_2-s]}(x-y-z)
			\\-p_{[t_1-s,t_2-s]}(x+h-y)+p_{[t_1-s,t_2-s]}(x-y)|^2
			\|f(s,y+z)\|_{L^p(\Omega)} ^2|z|^{2H-2}dydz\,,
		\end{multline*}
		and
		\begin{multline*}
			A_{12}(s)=\int_\RR\int_\RR |p_{[t_1-s,t_2-s]}(x+h-y)-p_{[t_1-s,t_2-s]}(x-y)|^2  \\
		 \times	\|f(s,y+z)-f(s,y)\|_{L^p(\Omega)} ^2|z|^{2H-2}dydz\,.
		\end{multline*}
Let us now bound $A_{11}$. Invoking Plancherel's identity with respect to $y$ and the explicit formula for $\cf p_{t}$, we have
		\begin{align*}
			A_{11}(s)&\leq C V_1^2(f)
			\quad\int_\RR\int_\RR\big|p_{[t_1-s,t_2-s]}(h+y-z)-p_{[t_1-s,t_2-s]}(y-z) \\
			&\hspace{2in}-p_{[t_1-s,t_2-s]}(h+y)+p_{[t_1-s,t_2-s]}(y)\big|^2|z|^{2H-2}dydz
			\\&\leq C V_1^2(f)\int_\RR\int_\RR  |e^{-\frac{t_2-s}{2} \kappa |\xi|^2}-e^{-\frac{t_1-s}{2}\kappa|\xi|^2}|^2 |e^{-i \xi z}-1|^2|e^{i \xi h}-1|^2 |z|^{2H-2}d\xi dz
			\\&\leq C V_1^2(f)\int_\RR   |e^{-\frac{t_2-s}{2} \kappa |\xi|^2}-e^{-\frac{t_1-s}{2}\kappa|\xi|^2}|^2 |e^{i \xi h}-1|^2 |\xi|^{1-2H}d\xi\,.
		\end{align*}
Moreover, owing to the inequality
\begin{equation}\label{eq:elem-dif-exp}
\int_0^{t_1} |e^{-\frac{t_2-s}{2} \kappa |\xi|^2}-e^{-\frac{t_1-s}{2}\kappa|\xi|^2}|^2 ds
\leq  \frac{|e^{-\frac{\Delta t \kappa}{2} |\xi|^2} -1|^2}{\kappa |\xi|^2} \,,
\end{equation}
	we obtain
		\begin{align}\label{eq:bnd-A11s-ds}
			\int_0^{t_1}A_{11}(s)ds
			&\leq C \kappa^{-1}V_1^2(f)\int_\RR |e^{-\frac{\Delta t \kappa }{2}|\xi|^2}-1|^2 |e^{i\xi h}-1|^2 |\xi|^{-1-2H} d\xi			    \notag  \\
			&\leq C \kappa^{-1} V_1^2(f) \, I \,,
		\end{align}
		where
\begin{equation}   \label{E2}
I:=
\int_\RR |1-e^{-\frac{\Delta t \kappa}{2}|\xi|^2}|^2\sin^2(\xi h/2) |\xi|^{-1-2H}d\xi.
\end{equation}
Our next step is to bound $I$ in two elementary and different ways.

\noindent
\emph{(i)}
The change of variable $h\xi:=\xi$ yields 		
		\begin{equation*}
			I= |h|^{2H}\int_\RR \lp 1-e^{-\frac{\kappa  \Delta t}{2|h|^2} |\xi|^2}\rp^2
			\sin^2( \xi/2)|\xi|^{-1-2H}d \xi\,,
		\end{equation*}
		and we then bound $1-e^{-\frac{\kappa  \Delta t}{2h^2}}$ by 1 to obtain $I\leq C |h|^{2H}$.

\noindent
\emph{(ii)}		
On the other hand, the change of variable $( \kappa  \Delta t )^{1/2}\xi:=\xi$ in \eref{E2} leads to
		\begin{equation*}
			I=(\kappa \Delta t)^{H}\int_\RR \lp 1-e^{-\xi^2/2}\rp^2
			\sin^2\Big(\frac{h \xi }{2(\kappa \Delta t)^{1/2}} \Big) \, |\xi|^{-1-2H}d \xi\,,
		\end{equation*}
		and we bound the trigonometric function $\sin^2$ by 1 to obtain $I\leq C (\kappa \Delta t)^{H}$.

		Interpolating the two estimates we have obtained for $I$, with a coefficient $\delta=\frac{\ga}{2H}\in[0,1]$, we see that
		\begin{equation}  \label{E3}
I \le C |h|^{2H\delta} \lp \ka \Delta t \rp^{H(1-\delta)}
		\leq C (\kappa \Delta t)^{\frac{2H- \ga}{2}}|h|^{\ga}\,.
		\end{equation}
		Plugging this identity back into \eqref{eq:bnd-A11s-ds}, we have shown
		\begin{equation*}
			\int_0^{t_1}A_{11}(s)ds\leq C \kappa^{-1}(\kappa \Delta t)^{\frac{2H- \ga}2}|h|^{\ga}V_1^2(f)\,,
		\end{equation*}
			forall  $\ga\in[0,2H]$.
		Let us now turn to the estimate for $A_{12}$. Similarly to what has been done for $A_{11}$ we get
		\begin{align*}
			\int_0^{t_1}A_{12}(s)ds&\leq C V_2^2(f)\int_0^{t_1}\int_\RR|p_{[t_1-s,t_2-s]}(h+y)-p_{[t_1-s,t_2-s]}(y)|^2dyds
			\\&\leq C V_2^2(f)\int_\RR\int_0^{t_1}| e^{-\frac{t_2-s}{2}\kappa |\xi|^2} - e^{-\frac{t_1-s}{2} \kappa |\xi|^2}  |^2ds|e^{i \xi h}-1|^2 d \xi\,.
		\end{align*}
Thanks to \eqref{eq:elem-dif-exp}, we thus end up with
\begin{equation*}
\int_0^{t_1}A_{12}(s)ds
\le
C	 \kappa^{-1}V_2^2(f)\int_\RR|1-e^{-\frac{\Delta t \kappa}{2}|\xi|^2}|^2\sin^2(h \xi/2) |\xi|^{-2} d \xi\,.
\end{equation*}
In addition, the integral on the right hand side can be estimated as $I$ above, and we get
		\begin{equation*}
			\int_0^{t_1} A_{12}(s) ds \leq C \, V_2^2(f) ( \kappa \Delta t)^{\frac{1-\ga'}{2}} |h|^{\ga'}\,	
			\end{equation*}
			for all $\gamma' \in [0,1]$.
		Since $1>2H$, we may choose $\ga'=\ga$ to obtain
		\begin{equation*}
			\int_0^{t_1}A_{12}(s)ds\leq C \kappa^{-1} (\kappa \Delta t)^{\frac{2H- \ga}{2}}|h|^{\ga} V_2^2(f)\,,
		\end{equation*}
		for all  $\ga\in[0,2H]$\,.
		Hence, the bounds on $A_{11}$ and $A_{12}$ yield
		\begin{align*}
			\|\ca_1\|_{L^p(\Omega)}^2\leq C V^2(f)(\Delta t)^{\frac{2H- \ga}{2}}h^{\ga}\,, 		\end{align*}
			for all $ \ga\in[0,2H]$\,.

\noindent
\textit{Step 2: Upper bound for $\ca_{2}$.}			
		The term $\|\ca_{2}\|^2_{L^p(\Omega)}$ can be estimated analogously to $\ca_{1}$. Indeed, the reader can check that, owing to inequality \eqref{eq:ineq-burk-2} and Plancherel's identity, we have
		\begin{align*}
			\|\ca_{2}\|^2_{L^p(\Omega)}\leq C V_{1}^2(f) \int_{0}^{\Delta t}\int_\RR e^{-s \kappa |\xi|^2}\sin^2({h \xi} /{2}) ( |\xi|^{1-2H}+1)d\xi ds \,,
		\end{align*}
		where we recall that $V_{1}$ is defined by \eqref{eq:def-V1-V2}.
		Taking integration in $ds$  first,  we see that
		\begin{align*}
			\|\ca_{2}\|^2_{L^p(\Omega)}\leq C \kappa^{-1} V_{1}^2(f)\int_\RR (1-e^{-\Delta t \kappa |\xi|^2})\sin^2(h \xi/2) ( |\xi|^{-1-2H}+ |\xi|^{- 2} )d\xi\,.
		\end{align*}
		These two integrals can be estimated as the term $I$ in \eref{E3}, and we get
		\begin{equation*}
			\|\ca_{2}\|^2_{L^p(\Omega)}\leq C V_{1}^2(f)(\Delta t)^{\frac{2H - \ga}{2}}|h|^{\ga}\,,		\end{equation*}
			for all  $\ga\in[0,2H]$\,.
		Let us remark that the constants in all previous estimates depend only on $T$, $p$ and $\kappa^{-1}$. In addition, as functions of $(p,\kappa^{-1})$, these constants have at most polynomial growth. Hence, gathering the estimates for $\|\ca_1\|^2_{L^p(\Omega)}$ and $\|\ca_{2}\|^2_{L^p(\Omega)}$ the proof of our claim \eqref{eq:joint Holder} is finished.
		
\noindent
\textit{Step 3: Proof of \eref{eq:time Holder}.}
   Again, we assume that $t_1 < t_2$, and we proceed as in the previous steps and the proof of Proposition \ref{prop.young}.  Indeed, we begin by writing
\begin{equation*}
 \|\Phi(t_2,x)-\Phi(t_1,x)\|_{L^p(\Omega)}
 \le
 B_1+B_2,
\end{equation*}
where
\[
B_{1}=\left\|\int_0^{t_1} \int_{\RR} p_{[t_1-s, t_2-s]}(x-y)f(s,y)W(ds,dy)\right\|_{L^p(\Omega)}
\]
and
\[
B_{2}=
\left \| \int_{t_1}^{t_2} \int_{\RR} p_{t_2-s}(x-y)f(s,y)W(ds,dy) \right \|_{L^p(\Omega)}.
\]
Once again we handle those two terms separately.

For the term $B_1$, we resort to inequality \eqref{eq:ineq-burk-2} in our usual way. We get
\begin{eqnarray*}
B_{1} &\le&
C \left ( \int_0^{t_1} \int_{\RR} \int_{\RR} p^2_{[t_1-s, t_2-s]}(x-y) \|f(s,y)-f(s,y+z)\|_{L^p(\Omega)}^2 |z|^{2H-2} dz dy ds \right ) ^{\frac{1}{2}}\\
&& + C \Big ( \int_0^{t_1} \int_{\RR} \int_{\RR}
|p_{[t_1-s, t_2-s]}(x-y) - p_{[t_1-s, t_2-s]}(x-y-z) |^2 \\
&&\hspace{3in}\times\|f(s,y+z)\|_{L^p(\Omega)}^2 |z|^{2H-2} dz dy ds \Big )^{\frac{1}{2}}.
\end{eqnarray*}
With the definition \eqref{eq:def-V1-V2} in mind, it is now readily checked that
\begin{equation}\label{eq:bnd-B1-1}
B_{1} \le
C \left (B_{11} V_2(f) + B_{12} V_1(f) \right )\,,
\end{equation}
with
\[
 B_{11} = \left (  \int_0^{t_1} \int_{\RR} |p_{[t_1-s, t_2-s]}(x-y)|^2 dy ds \right)^{\frac{1}{2}}
 \]
 and
 \[
 B_{12} = \left ( \int_0^{t_1} \int_{\RR} \int_{\RR} |p_{[t_1-s, t_2-s]}(x-y) - p_{[t_1-s, t_2-s]}(x-y-z) |^2 |z|^{2H-2} dz dy ds \right )^{\frac{1}{2}}.
 \]
We now appeal to Plancherel's identity to get
\begin{equation*}
B_{11}
=
C  \left ( \int_0^{t_1} \int_{\RR} \left |e^{-\frac{t_2-s}{2} \kappa |\xi|^2}-e^{-\frac{t_1-s}{2} \kappa |\xi|^2} \right|^2 d\xi ds \right)^{\frac{1}{2}}
= C (t_2-t_1)^{\frac{1}{4}}\,,
\end{equation*}
and
 \begin{eqnarray*}
B_{12}
&=& C \left ( \int_0^{t_1} \int_{\RR} \int_{\RR} \left | e^{-\frac{t_2-s}{2} \kappa |\xi|^2} -e^{-\frac{t_1-s}{2}\kappa |\xi|^2}\right|^2 |e^{-i\xi z} -1|^2 |z|^{2H-2} dz d\xi ds  \right)^{\frac{1}{2}}\\
&=&C \left ( \int_0^{t_1} \int_{\RR} \left | e^{-\frac{t_2-s}{2}\kappa |\xi|^2} - e^{-\frac{t_1-s}{2}\kappa |\xi|^2}\right|^2 |\xi|^{1-2H} d\xi ds \right)^{\frac 12}\\
&=&C (t_2-t_1)^{\frac H 2}\,.
 \end{eqnarray*}
Reporting these estimates in \eqref{eq:bnd-B1-1} and observing that $H < \frac{1}{2}$, we end up with
 \begin{equation*}
 B_1
\leq
C (t_2-t_1)^{\frac{H}{2}} \lc V_1(f) + V_2(f) \rc
 \leq  C (t_2-t_1)^{\frac{H}{2}} \| f\|_{\mathfrak {X}_{T,\theta_0}^ p} e^{\theta_0 T}.
  \end{equation*}
The patient reader might check that the same kind of upper bound is valid for $B_{2}$, and gathering the estimates for $B_{1}$ and $B_{2}$ yields inequality \eqref{eq:time Holder}.
	\end{proof}

\section{Existence and uniqueness of the solution}
\label{sec:existence}
In this section we will  establish a result regarding the uniqueness of the solution.  Then we will describe the structure of some new spaces of stochastic processes. Those spaces will finally be used to show the existence of the solution.

\subsection{Uniqueness of the solution}

In this subsection we give some results about the uniqueness of the solution assuming that the solution has enough regularity. To this end, we first introduce a norm $\|\cdot\|_{\mathcal Z_T^p}$ for a random field $v(t,x)$ as follows
\begin{equation}\label{eq:dcp-norm-ZTp}
\|v\|_{\mathcal{Z}_T^{p}} = \sup_{t\in[0,T]} \|v(t,\cdot)\|_{L^p(\Omega \times \RR)}  + \sup_{t\in[0,T]}  \cn^*_{\frac 12-H, p} v(t),
\end{equation}
where $p\ge 2$ and
\begin{equation}\label{def: space Z}
 \cn^*_{\frac 12-H, p} v(t)= \left (\int_{\RR}  \|v(t,\cdot)-v(t,\cdot+h)\|^2_{L^p(\Omega\times \RR)}  |h|^{2H-2}  dh \right)^{\frac{1}{2}}.
\end{equation}
Then the space $\mathcal Z_T^p$ will consist all the random fields such that the above quantity is finite. Observe that according to the definition \eqref{e1}, we have $\cn^{*}_{\frac12-H, p} v(t)=\cn^{L^p(\Omega\times \RR)}_{\frac12-H, p} v(t)$.

\begin{remark}
	A similar quantity to $\cn^*_{\frac12-H}$ defined in \eqref{def: space Z} appears in the characterization of Besov spaces in finite differences, see \cite[Theorem 2.36]{BCD}.
\end{remark}

The proof of the uniqueness theorem  requires a localization argument, based on  uniform estimates (in space and time) of stochastic convolutions. This is provided by the following lemma.

\begin{lemma}  \label{lem2}
Suppose that $p> \frac 6{4H-1}$. Let $v$ be a process in the space $\mathcal Z_T^p$.    As in \eqref{eq:def-A-t-x}, define
\begin{equation}\label{eq:not-stoch-convolution}
\Phi(t,x) = \int_0^t \int_{\RR} p_{t-s}(x-y)v(s,y )W(ds,dy)\, .
\end{equation}
Then, there exists a constant $C$ depending on $T$, $p$ and $H$, such that
\begin{equation}  \label{E4}
 \left\| \sup_ {t\in [0,T] ,\, x\in \RR}   \cn_{\frac 12-H} \Phi(t,x)   \right\|_{L^p(\Omega)}
\le C  \| v\|_{ \mathcal Z_T^p},
\end{equation}
where recalling our definition \eqref{e1}, we have $\cn_{1/2-H} \Phi(t,x)=\cn_{1/2-H}^{\R}\Phi(t,x)$.
\end{lemma}

\begin{remark}
Let us stress the following facts:

\noindent
(i)
In relation \eqref{E4}, the operator $\cn_{\frac12-H}$ (defined in \eref
{e4}) acts on the trajectories of the random field $\Phi(t,x)$. As a consequence,  $\cn_{\frac 12-H} \Phi(t,x)$ is a random variable.

\noindent
(ii) With respect to Proposition \ref{prop.young}, inequality \eqref{E4} involves a $\sup$ in the variable $x\in\R$ before taking $L^{p}(\oom)$ norms. We thus get a stronger result with different kind of assumptions (namely $v\in\mathcal Z_T^p$ instead of $v\in\mathfrak{X}_{T}^{\frac12-H,p}$).

\end{remark}

\begin{proof}[Proof of Lemma \ref{lem2}]
We shall apply the factorization method to handle the stochastic convolution (see, for instance,   \cite{DPZ}). Namely, an application of a stochastic version of Fubini's theorem enables to write
\begin{equation*}
\Phi (t,x)=\frac{\sin(\pi \alpha)}{\pi}\int_0^t \int_{\RR} (t-r)^{\alpha-1}p_{t-r}(x-z)Y(r,z)dz dr\,,
\end{equation*}
with
\begin{equation*}
Y(r,z)=\int_0^r \int_{\RR}(r-s)^{-\alpha}p_{r-s}(z-y)v(s,y)W(ds,dy)\,,
\end{equation*}
and where $\alpha\in(0,1)$ is a parameter whose value  will be chosen later.  The proof will be done in two steps.

\noindent
\textit{Step 1:  Uniform estimate of $  \cn_{\frac 12-H} \Phi(t,x)$.}
In order to estimate  $\cn_{\frac 12-H} \Phi(t,x)$, we   bound the difference $\Phi(t,x)-\Phi(t,x+h)$ as follows
\begin{eqnarray*}
&&|\Phi(t,x)-\Phi(t,x+h)|\\
&=&\frac{\sin(\alpha \pi)}{\pi} \left| \int_0^t \int_{\RR}(t-r)^{\alpha-1}\left (p_{t-r}(x-z)-p_{t-r}(x+h-z) \right) Y(r,z)dz dr \right|\\
&\leq&\frac{\sin(\alpha \pi)}{\pi} \int_0^t (t-r)^{\alpha-1}\left \|p_{t-r}(\cdot)-p_{t-r}(\cdot+h) \right\|_{L^{q}(\RR)} \|Y(r,\cdot)\|_{L^p(\RR)}dr\,,
\end{eqnarray*}
where   $q$ satisfies $p^{-1}+q^{-1}=1$.  So using Minkowski's integral inequality, we get
\begin{align}\label{eq:bnd-incr-Phi-1}
&\int_{\RR} |\Phi(t,x)-\Phi(t,x+h)|^2 |h|^{2H-2} dh \notag\\
&\le C \, \int_{\RR} \left (\int_0^t (t-r)^{\alpha-1} \left \|p_{t-r}(x-\cdot)-p_{t-r}(x+h-\cdot) \right\|_{L^q(\RR)} \left\| Y(r,\cdot)\right\|_{L^p(\RR)} dr \right)^2 |h|^{2H-2}dh \notag\\
&\leq C \,  \left (\int_0^t (t-r)^{\alpha-1} \left \|Y(r,\cdot) \right\|_{L^p(\RR)}
\lc K_{t}(r) \rc^{1/2}  \, dr \right)^2\, ,
\end{align}
where we have set
\begin{equation*}
K_{t}(r)
:=
\int_{\RR} \left\|p_{t-r}(x-z)-p_{t-r}(x+h-z) \right\|^2_{L^q(\RR, dz)}|h|^{2H-2} dh\, .
\end{equation*}
Now the kernel $K_{t}$ can be bounded by elementary methods:
with the change of variable $z \rightarrow \sqrt{t-r}z$ and $h \rightarrow \sqrt{t-r}h$, we obtain
\begin{multline*}
K_{t}(r)=\int_{\RR} \left\|p_{t-r}(x-z)-p_{t-r}(x+h-z) \right\|^2_{L^q(\RR, dz)}|h|^{2H-2} dh\\
= C \, (t-r)^{-\frac{3}{2}+\frac{1}{q}+H} \int_{\RR} \left( \int_{\RR} \left |e^{-\frac{z^2}{2\kappa}}-e^{-\frac{(z+h)^2}{2\kappa}} \right|^q dz \right)^{\frac{2}{q}} |h|^{2H-2}dh
=C (t-r)^{-\frac{1}{2}-\frac{1}{p}+H}\,,
\end{multline*}
where we have used the fact that $q^{-1}=1-p^{-1}$, and the constant $C$ in the above equation and below in this proof may depend on $\kappa$.  Going back to \eqref{eq:bnd-incr-Phi-1}, the following holds true:
\begin{multline*}
\int_{\RR} |\Phi(t,x)-\Phi(t,x+h)|^2 |h|^{2H-2} dh
\leq C  \left (\int_0^t (t-r)^{\alpha-1+\frac{1}{2}(H-\frac{1}{p}-\frac{1}{2})} \left \| Y(r, \cdot)\right\|_{L^p(\RR)} dr \right)^2\\
\leq C \left (\int_0^t (t-r)^{q [\alpha-1+\frac{1}{2}(H-\frac{1}{2}-\frac{1}{p})]} dr\right)^{\frac{2}{q}} \left (\int_0^t \left \|Y(r,\cdot) \right\|^p_{L^p(\RR)} dr\right)^{\frac{2}{p}}\, .
\end{multline*}
We can now start to tune our parameters. It is easily checked that
the first integral in the right hand side above is finite (uniformly in $0 < t \leq T$) if and only if
\begin{equation}\label{eq: restriction alpha}
\alpha > \frac{3}{2p}+\frac{1}{4}-\frac{H}{2}\,.
\end{equation}
With this choice of $\alpha$, we get
\begin{equation*}
\int_{\RR} |\Phi(t,x)-\Phi(t,x+h)|^2 |h|^{2H-2} dh\leq C \left ( \int_0^t \left\|Y(r,\cdot) \right\|^p_{L^p(\RR)} dr \right)^{\frac{2}{p}}\,,
\end{equation*}
and since this bound is uniform in $x$, this yields
\begin{equation}\label{eq:bnd-incr-Phi-2}
\sup_{t\in[0,T], x \in \RR} [\cn _{\frac 12-H} \Phi(t,x)]^2 \leq C \left ( \int_0^T \left\|Y(r,\cdot) \right\|^p_{L^p(\RR)} dr \right)^{\frac{2}{p}}\,.
\end{equation}
Then, to prove \eref{E4}  it suffices to show that
\begin{equation}
\label{e5}
\be   \int_{\R} |Y(r,z)|^p \, dz   \le C \| v\|_{ \mathcal Z_T^p}^p.
\end{equation}

\noindent
\textit{Step 2: Proof of  \eqref{e5}.}
Set $g_{r,z}(s,y) = (r-s)^{-\alpha} p_{r-s}(z-y) v(s,y) $, so that
\begin{equation*}
Y(r,z)=\int_{0}^{r}\int_{\R} g_{r,z}(s,y) \, W(ds,dy).
\end{equation*}
Then applying the Burkholder type inequality \eqref{eq:ineq-burk-2}, plus an elementary decomposition of the increments of $g_{r,z}$, we obtain
\begin{equation*}
  \be   \int_{\R} |Y(r,z)|^p \, dz
\le
C\, \lc D_{1}(r) + D_{2}(r) \rc,
\end{equation*}
where
\begin{multline*}
D_{1}(r)=
\int_{\RR} \Big (\int_0^r \int_{\RR^2} (r-s)^{-2\alpha} \left |p_{r-s}(y)-p_{r-s}(y+h)\right|^2 \\
\times \left\|v(s,y+z+h) \right\|_{L^p(\Omega)}^2 |h|^{2H-2} dh dy ds\Big)^{\frac{p}{2}} dz
\end{multline*}
and
\begin{multline*}
D_{2}(r)=
\int_{\RR} \Big (\int_0^r \int_{\RR^2} (r-s)^{-2\alpha} \left |p_{r-s}(y)\right|^2 \\
\times \left\|v(s,y+z+h)-  v(s,y+z)\right\|_{L^p(\Omega)}^2 |h|^{2H-2} dh dy ds\Big)^{\frac{p}{2}} dz.
\end{multline*}

Let us now bound the term $D_{1}$. Invoking Minkowski's integral inequality, it is easily seen that
\begin{equation*}
D_{1}(r)
\le
\Big( \int_0^r \int_{\RR^2} (r-s)^{-2\alpha} \left |p_{r-s}(y)-p_{r-s}(y+h)\right|^2 \left\|v(s,\cdot) \right\|_{L^p(\Omega \times \RR)}^2 |h|^{2H-2} dh dy ds\Big)^{\frac{p}{2}}.
\end{equation*}
Integrating this identity in $h$ and $y$, we end up with
\begin{equation*}
D_{1}(r)
\le
C  \, \Big (\int_0^r  (r-s)^{-2\alpha+H-1}   \|v(s,\cdot) \|_{L^p(\Omega \times \RR)} ^2 \, ds\Big)^{\frac{p}{2}} \,.
\end{equation*}
Similarly we get the following estimate for $D_{2}(r)$
\begin{eqnarray*}
D_{2}(r) &\leq& C \, \bigg (\int_0^r \int_{\RR} (r-s)^{-2\alpha-\frac{1}{2}} \left\|v(s,\cdot+h)-v(s,\cdot) \right\|_{L^p(\Omega\times \RR)}^2 |h|^{2H-2} dh ds\bigg)^{\frac{p}{2}} \\
&=&
C \, \bigg (\int_0^r  (r-s)^{-2\alpha-\frac{1}{2}}  \left[\cn^*_{\frac 12-H,p} v(s)\right]  ^2\, ds\bigg)^{\frac{p}{2}}\,.
\end{eqnarray*}
Combining the estimates for $D_{1}(r)$ and $D_{2}(r)$ we obtain
\begin{multline}\label{eq:bnd-Y-Lp}
  \be     \int_{\R} |Y(r,z)|^p \, dz
\le
C  \, \bigg(\int_0^r  \Big[ (r-s)^{-2\alpha+H-1}  \|v(s,\cdot) \|_{L^p(\Omega \times \RR)} ^2\\
    +
(r-s)^{-2\alpha-\frac{1}{2}}  \left[\cn^*_{\frac 12-H,p} v(s)\right]  ^2
\Big]  \, ds\bigg)^{\frac{p}{2}}   \,.
\end{multline}

Let us go back now to the values of our parameters $\al,p$. One can check that the two singularities in the integrals on the right hand side above are non divergent whenever $\alpha < \frac{H}{2}$. Combining this condition with the restriction \eref{eq: restriction alpha}, we end up with the relation
\begin{equation}\label{eq: restriction-alpha-2}
\frac{3}{2p}+\frac{1}{4} -\frac{H}{2} < \alpha < \frac{H}{2}.
\end{equation}
Those two conditions can be jointly met if and only if $H > \frac{1}{4}$ and $p > \frac 6{4H-1}$.
This completes the proof of the lemma.
\end{proof}

\begin{remark}\label{rmk:stoch-convol-in-Z-T-p}
Notice that the previous lemma implies that for any process $v\in \mathcal Z_T^p$,  the random variable
$ \sup_ {t\in[0,T]} \sup_{x\in \RR}   \cn_{\frac 12-H} \Phi(t,x)$ is finite almost surely, if $\Phi$ is given by
\eref{eq:not-stoch-convolution}.
\end{remark}

We can now turn to the uniqueness result  for equation \eqref{spde with sigma}.

\begin{theorem}\label{thm:uniqueness}
 Assume the following conditions hold  true:
 \begin{enumerate}
 \item For   $p> \frac 6{4H-1}$, the initial condition $u_0$ is in $L^p(\RR)$ and
 \begin{equation}  \label{e7}
\int_{\RR}\|u_0(\cdot)-u_0(\cdot+h)\|^2_{L^p(\RR)}|h|^{2H-2} dh< \infty\,.
 \end{equation}
 \item $\sigma$ is differentiable, its derivative is Lipschitz and $\sigma(0)=0$.
 \item $u$ and $v$ are two solutions of \eref{spde with sigma} and $u,v \in \mathcal Z_T^p$.
 \end{enumerate}
Then for every $t \in [0,T]$ and $x \in \RR$, $u(t,x)=v(t,x), a.s.$
\end{theorem}

\begin{remark}
This is the first occurrence of the hypothesis $\si(0)=0$, and one might wonder about the necessity of this assumption. To this respect, let us mention that if we define $\Phi$ as in \eqref{eq:def-A-t-x} for $f\equiv \1$, then $\Phi$ does not belong to $\mathcal{Z}_T^p$.
% This seems to impose the condition $\si(0)=0$ if we wish to solve our equation in the space $\mathcal{Z}_T^p$, but does not rule out the possibility of solving it in a weighted space (with respect to the space variable).
\end{remark}

\begin{proof}
Assume that $u$ solves \eref{spde with sigma} and $u \in \mathcal Z_T^p$. From the mild formulation of the solution we have
\begin{equation}  \label{E5}
u(t,x)=p_tu_0(x)+\int_0^t \int_{\RR} p_{t-s}(x-y)\sigma(u(s,y))W(ds,dy)\, .
\end{equation}
We claim that
\begin{equation}  \label{e8}
 \sup_ {t\in[0,T]} \sup_{x\in \RR}   \cn_{\frac 12-H} u(t,x)<\infty, \quad {\rm a.s.}
 \end{equation}
  This follows from the   decomposition \eref{E5}. Indeed, on one hand,  \eref{e7} implies that, if $g(t,x)=p_tu_0(x)$, then
\[
 \sup_ {t\in[0,T]} \sup_{x\in \RR}   \cn_{\frac 12-H} g (t,x) <\infty.
 \]
 On the other hand, from the properties of $\sigma$, it follows that if $u\in \mathcal Z_T^p$, then $\sigma(u)$ also belongs to $\mathcal Z_T^p$ (notice that to estimate the first term of \eref{eq:dcp-norm-ZTp}  for $\sigma(u)$, we need to assume $\sigma(0)=0$). Hence, Remark \ref{rmk:stoch-convol-in-Z-T-p} entails
 \[
 \sup_ {t\in[0,T]} \sup_{x\in \RR}   \cn_{\frac 12-H}  \sigma(u)(t,x) <\infty, \quad {\rm a.s.}
 \]
If $v$ is another solution of equation \eref{spde with sigma} belonging also to $\mathcal Z_T^p$, then \eref{e8} also holds for $v$. In this way, we can define the stopping times
\begin{eqnarray*}
T_k&=&\inf \bigg\{0 \leq t \leq T: \sup_{0 \leq s \leq t , x \in \RR} \int_{\RR} |u(s,x)-u(s,x+h)|^2|h|^{2H-2} dh \geq k\\
&& \quad \text{or} \quad \sup_{0 \leq s \leq t , x \in \RR} \int_{\RR} |v(s,x)-v(s,x+h)|^2|h|^{2H-2} dh \geq k \bigg\}\,,
\end{eqnarray*}
and $T_k \uparrow T$, almost surely, as $k$ tends to infinity.
Our strategy will be to control the two following quantities
\[
I_{1}(t,x)=\be\lc {\bf 1}_{\{ t < T_k\}}  | u(t,x)-v(t,x)  |^2 \rc
\]
and
\[
I_{2}(t,x)=
\be \lc \int_{\RR}  {\bf 1}_{\{ t < T_k\}} \left | u(t,x) -  v(t,x) - u (t, x+h) +  v(t, x+h) \right|  ^2
|h|^{2H-2} dh \rc \, .
\]
We also set $\ci_{j}(t)=\sup_{x\in\R}I_{j}(t,x)$ for $j=1,2$.

In order to bound $I_{1}$, let us first use elementary properties of It\^o's integral, which yield
 \begin{align*}
&{\bf 1}_{\{ t < T_k\}} \lp u(t,x)-v(t,x) \rp
=
{\bf 1}_{\{ t < T_k\}} 
\int_0^{t \wedge T_k} \int_{\RR} p_{t-s}(x-y)\left [ \sigma(u(s,y))-\sigma(v(s,y)) \right] W(ds,dy)  \\
=&
{\bf 1}_{\{ t < T_k\}} \int_0^t \int_{\RR} p_{t-s}(x-y){\bf 1}_{\{s < T_k\}} \left [ \sigma(u(s,y))-\sigma(v(s,y)) \right] W(ds,dy) .
\end{align*}
We thus get $I_{1}(t,x) \le C( I_{11}(t,x) + I_{12}(t,x) )$, where
\begin{multline*}
I_{11}(t,x) =
\be \int_0^t \int_{\RR^2} |p_{t-s}(x-y)-p_{t-s}(x-y-h)|^2\\
 \times   {\bf 1}_{\{ s < T_k\}}    \left|   \sigma(u(s,y+h))-   \sigma (v(s,y+h))  \right|^2 |h|^{2H-2}dh dy ds \, ,
\end{multline*}
and
\begin{multline*}
I_{12}(t,x) =
\be \int_0^t \int_{\RR^2} p^2_{t-s}(x-y)  {\bf 1}_{\{ s < T_k\}}    \big|\sigma (u(s,y)  )-\sigma  (v(s,y) ) \\
 - \sigma (u(s,y+h)  )+ \sigma (u(s,y+h)  )\big|^2 |h|^{2H-2} dh dy ds \, .
\end{multline*}
Next we bound the term $I_{11}(t,x)$ as follows
\begin{multline*}
I_{11}(t,x) \le
C \, \be \int_0^t \int_{\RR^2} |p_{t-s}(x-y)-p_{t-s}(x-y-h)|^2 \\
 \times
 {\bf 1}_{\{ s < T_k\}} |u(s,y+h) - v(s,y+h)| ^2
 |h|^{2H-2}dh dy ds
\le
C \, \int_0^t (t-s)^{H-1} \ci_{1}(s)  \, ds,
\end{multline*}
where we recall that $\ci_{1}(t)=\sup_{x\in\R}I_{1}(t,x)$, and the constant $C$ in the above inequality and below in this proof may depend on $\kappa$. 
Let us now invoke the following elementary bound on the rectangular increments of $\si$, valid whenever $\si^{\prime}$ is Lipschitz
\begin{eqnarray*}
|\sigma(a)-\sigma(b)-\sigma(c)+\sigma(d)| \leq C |a-b-c+d| + C |a-b|(|a-c|+|b-d|)\,,
\end{eqnarray*}
With this additional ingredient, and along the same lines as for $I_{11}(t,x)$, we get
\begin{equation*}
I_{12}(t,x)
\le Ck\,
\int_0^t (t-s)^{-\frac{1}{2}} \lc \ci_{1}(s) + \ci_{2}(s)  \rc \, ds\,.
\end{equation*}
Finally, gathering our estimates on $I_{11}$ and $I_{12}$ we end up with
\begin{equation*}
\ci_{1}(t) \le Ck \,
\int_0^t (t-s)^{H-1} \lc \ci_{1}(s) + \ci_{2}(s)  \rc \, ds\,.
\end{equation*}
The term $I_{2}(t,x)$ above is dealt with exactly  the same way, and we leave to the reader the task of showing that
\begin{equation*}
\ci_{2}(t) \le Ck \,
\int_0^t (t-s)^{2H-\frac32} \lc \ci_{1}(s) + \ci_{2}(s)  \rc \, ds\,.
\end{equation*}
As a consequence,
\[
\ci_{1}(t)  + \ci_{2}(t) \le Ck \,
\int_0^t (t-s)^{2H-\frac32} \lc \ci_{1}(s) + \ci_{2}(s)  \rc \, ds\,,
\]
 which implies $\ci_{1}(t)  + \ci_{2}(t)=0$ for all $t \in [0,T]$. In particular,
\[
\be \lc {\bf 1}_{\{t < T_k\}}  |  u(t,x)-   v(t,x) |^2 \rc=0\, ,
\]
which implies $u(t,x)=v(t,x)$ a.s. on $\{t<T_k\}$ for all $k\ge 1$ and $t\in [0,T] $. Therefore, taking into account that $T_k \uparrow \infty$ a.s. as $k$ tends to infinity, we conclude that $u(t,x) =v(t,x)$ a.s. for all $(t,x)\in [0,T] \times \RR$.
This proves the uniqueness.
\end{proof}
\subsection{Space-time function spaces} % (fold)
\label{sub:space_time_function_spaces}
	We introduce here the function spaces which form the underlying spaces of our treatment for the  existence  of the solution. Since these spaces do not belong to standard classes of function spaces, we describe them in detail.

	We denote by $C_{\mathrm {uc}}([0,T]\times\RR)$ the space of    all  real-valued continuous functions on $[0,T]\times \RR$ equipped with the topology of convergence uniformly over compact sets.
Let $(B,\|\cdot\|)$ be a Banach space equipped with the norm $\|\cdot\|$. Let $\beta\in(0,1)$ be a fixed number. For every $\delta\in(0,\infty]$ and every function $f: \RR\to B$, we introduce the function $\cn_{\beta}^{B,(\delta)}  f:  \RR\to[0,\infty]$ defined by
 	\begin{equation}  \label{E6}
 	 	\cn_{\beta}^{B,(\delta)} f(x)=\left( \int_ {|h|\le \delta } \| f(x+ h)- f(x)\|^2|h|^{-1-2 \beta}dh \right)^{\frac12}\,.
 	\end{equation}
 	 Notice that for $\delta =\infty$, the quantity \eqref{E6} coincides with the function $ \cn_{\beta}^{B,(\infty)}  f =\cn_\beta^B f$ introduced in  \eref{e1}. As our usual practice, when $B=\RR$ we omit the dependence of $\RR$ in $\cn_{\beta}^{\RR,(\delta)}$ and simply write $\cn_{\beta}^{(\delta)}$.
	
	   As we will see later along the development of the paper, $ \cn_{\beta}^{B,(\delta)}  f$ plays a role analogous to the modulus of continuity of $f$ near $x$.  It follows from  the triangular inequality, that $\cn$ satisfies
 	\begin{equation}\label{V.subadditive}
 		| \cn_{\beta}^{B,(\delta)} f (x)- \cn_{\beta}^{B,(\delta)} g(x)  |\le \cn_{\beta}^{B,(\delta)} (f -g)(x) 	\end{equation}
 	for all $\delta\in(0,\infty]$, functions $f, g$ and $x$ in $\RR$. Thus, $\cn $ is a seminorm.

 	Suppose, for instance, that a function $f$ has modulus of continuity  $|h|^\beta \omega(h)$ at $x$, for any $|h| \le \delta$. Then $[\cn_{\beta}^{B,(\delta)} f  (x)]^2 $ is majorized by $2\int_0^\delta\omega^2(h)h^{-1}dh$. Thus,   for $\cn_{\beta}^{B,(\delta)} f(x)$ to be finite, it is sufficient that $ \omega^2(h)h^{-1}$ is integrable near 0. On the other hand, if $\cn_{\beta}^{B,(\delta)} f$ is bounded over a domain, the following proposition asserts that $f$ is necessarily H\"older continuous.
 		\begin{proposition}\label{prop.embed}
		Let $I$ be a non-empty open interval of $\RR$ and $\delta\in(0,\infty]$. Let $f$ be a function on $\RR$ such that $\sup_{x\in \bar I}\cn_{\beta}^{B,(\delta)} f(x)$ is finite. Then
		\begin{equation}\label{ineq.embed}
			\sup_{x\in I, |y|\le \frac{\delta}3\wedge \mathrm{dist}(x,\partial I)} \frac{\|f(x+y)-f(x)\|}{|y| ^\beta} \le c(\beta)\sup_{x\in \bar I}\cn_{\beta}^{B,(\delta)} f(x)
		\end{equation}
		for some finite constant $c(\beta)$ which depends only on $\beta$.
	\end{proposition}
	\begin{proof}
		For every $x\in I$ and positive $R$, $R\le \delta$, we denote $f_{x,R}=\frac{1}{2R}\int_{-R}^R f(y+x)dy$. We first estimate $\|f(x)-f_{x,R}\|$ as follows
		\begin{align}
			\|f(x)-f_{x,R}\|&\le\frac1{2R}\int_{-R}^R \|f(x)-f(x+y)\|dy
			\nonumber\\&\le \frac1{2R}\left(\int_{-R}^R \|f(x)-f(x+y)\|^2 |y|^{-1-2 \beta}dy\right)^{\frac 12 }\left(\int_{-R}^R |y|^{1+2 \beta}dy \right)^{\frac 12}
			\nonumber\\&\le \frac{R^{\beta}}  {2\sqrt{(1+\beta)}} \sup_{x\in \bar I}\cn_{\beta}^{B,(\delta)} f(x)\,. \label{tempt.ffr}
		\end{align}
		Let us now fix $x\in I$ and $y\in\RR$ such that $|y|\le \delta/3\wedge \mathrm{dist}(x,\partial I)$. We also choose $R=|y|$. It follows from triangle inequality that
		\begin{equation}\label{tempt.fxyr}
			\|f(x+y)-f(x)\|\le\|f(x+y)-f_{x+y,R}\|+\|f_{x+y,R}-f_{x,R}\|+\|f(x)-f_{x,R}\|\,.
		\end{equation}
		For the second term, we apply Minkowski's inequality to get
\begin{equation*}
\|f_{x+y,R}-f_{x,R}\|
\le\frac1{4R^2} \int_{-R}^R\int_{-R}^R\|f(x+y+z)-f(x+w)\| \, dzdw,
\end{equation*}
and invoking Cauchy-Schwarz' inequality this yields
\begin{multline*}
\|f_{x+y,R}-f_{x,R}\|
\le\frac1{4R^2} \int_{-R}^R
\left(\int_{-R}^R\|f(x+y+z)-f (x+w)\|^2 |y+z-w|^{-2 \beta-1} dz\right)^{\frac12}  \\
\times \left(\int_{-R}^R |y+z-w|^{2 \beta+1}dz \right)^{\frac12}  dw\,.
\end{multline*}
Notice that because of the restrictions on the variables, the domain of integration above satisfies $|y+z-w|\le 3R\le \delta$ and $x+w\in \bar I$. Hence
		\begin{align*}
			\|f_{x+y,R}-f_{x,R}\|\leq C_\beta \sup_{y\in \bar I}\cn_{\beta}^{B,(\delta)} f (y)R^\beta\,.
		\end{align*}
		We can now conclude our proof as follows: the first and third terms on  the right hand side of \eqref{tempt.fxyr} are estimated in \eqref{tempt.ffr}. Combining these estimates within \eqref{tempt.fxyr} yields \eqref{ineq.embed}.
	\end{proof}

\begin{remark} 
In the same way as for the quantities $\cn_{\beta}^{B} f$, the function $\cn_{\beta}^{B,(\delta)} f$ can be defined for functions  defined on $\RR_+ \times \RR$. In this case, we have $\cn_{\beta}^{B,(\delta)}f : \RR_+ \times \RR \to [0,\infty]$.
\end{remark}

	Whenever $\sigma$ is an affine function (i.e. $\sigma(u)=au+b$ for some constants $a,b$), the spaces $\XX^{\beta,p}_T$ are sufficient to show existence and uniqueness for equation \eref{spde with sigma}. On the other hand, the case of general Lipschitz function $\sigma$ leads to the consideration of additional spaces, which we are going to study now.

  For every $h\in\RR$, let $\tau_h$ be the translation map in the spatial variable, that is $\tau_hf(t,x)=f(t,x-h)$.

  \begin{definition}  \label{def1}
  Let $X^\beta_T$ be the space of all real-valued continuous functions $f$ on $[0,T]\times \RR$ such that
  \begin{itemize}
  \item[(i)]   $(t,x)\mapsto \cn_{\beta }^{(1)} f (t,x)$ is finite and continuous on $[0,T]\times \RR$.
  \item[(ii)]  $\lim_{h\downarrow0}\sup_{t\in[0,T],\,x\in[-R,R]} \cn_{\beta }^{(1)} (\tau_hf-f) (t,x)=0$
 for every positive $R$.
 \end{itemize}
\end{definition}
	
We equip $X^\beta_T$ with the following topology. A sequence $\{f_n\}$ in $X^\beta_T$  converges to $f$ in $X^\beta_T$ if for all $R>0$, the sequences $\{f_n\}$ and $ \{\cn_{\beta }^{(1)} (f_n-f)\} $ converge uniformly on $[0,T]\times[-R,R]$ to $f$ and $0$ respectively. We define a metric on $X^\beta_T$ as follows
	\begin{equation}\label{metric.d}
		 d_ \beta(f,g)=\sum_{n=1}^\infty 2^{-n}\frac{\|f-g\|_{n,\beta}}	{1+\|f-g\|_{n,\beta}}\,,
	\end{equation}
	where $\|\cdot\|_{n,\beta}$ is the seminorm
	\begin{equation*}
		\|f\|_{n,\beta}:=\sup_{t\in[0,T],\,x\in[-n,n]}|f(t,x)|+\sup_{t\in[0,T],\,x\in[-n,n]}\cn_{\beta }^{(1)} f(t,x) \,.
	\end{equation*}
	Since functions in $X^\beta_T$ are locally bounded, the topology of $X^\beta_T$ is not altered if in the previous definition  $\cn_{\beta }^{(1)}f$ is replaced by $\cn_{\beta}^{(\delta)} f$ for some \textit{finite}  and positive $\delta$. We emphasize that replacing $\delta$ by $\infty$ would create a strictly smaller space.
	\begin{remark} 
		The space which satisfies only condition (i) in Definition \ref{def1} would be too big and fail to be separable.  Analogous situations occur frequently in analysis. In the study of Morrey spaces, this fact was first observed by Zorko in \cite{Z}. Continuity spatial translations with respect to a norm is therefore sometimes called Zorko condition. 
	\end{remark}

	\begin{proposition}\label{prop.Ncomplete}
		$X^\beta_T$ is a complete metric space.
	\end{proposition}
	\begin{proof}
		Let $\{f_n\}$ be a Cauchy sequence in $X^\beta_T$.  Since the space $C_{\mathrm {uc}}([0,T]\times\RR)$ is complete, there exists continuous function $f:[0,T]\times\RR\to \RR$ such that for all compact intervals $I$,
		\begin{equation*}
			\lim_{n\to\infty}\sup_{t\in[0,T], \, x\in I}|f_n(t,x)-f(t,x)|=0\,.
		\end{equation*}
		Let us fix a compact interval $I=[-N,N]$, and $\varepsilon>0$. There exists $n_0>0$ such that
		\begin{equation*}
			\sup_{  t \in [0,T], \, x\in I }\cn_{\beta }^{(1)}(f_n-f_m)(t,x)< \varepsilon
		\end{equation*}
		for all $m,n\ge n_0$. It follows from Fatou's lemma that
		\begin{equation*}
			\cn_{\beta }^{(1)} (f_n-f)(t,x)\le\liminf_{m\to \infty}  \cn_{\beta }^{(1)}(f_n-f_m)(t,x)\le\varepsilon ,
		\end{equation*}
		for every $t\in[0,T],\,x\in I$ and $n\ge n_0$. This implies that $\cn_{\beta }^{(1)}(f_n-f)$ converges to 0 uniformly on $[0,T]\times I$. In addition, from \eqref{V.subadditive}, it follows that $\cn_{\beta }^{(1)}f_n$ converges to $\cn_{\beta }^{(1)}f$ uniformly on $[0,T]\times I$, thus the continuity of $\cn_{\beta }^{(1)} f_n$ implies that  of $\cn_{\beta }^{(1)}f$.

		It remains to check that $f$ satisfies the condition (ii) of  Definition \ref{def1}. For every $\varepsilon>0$ and $|h|\le 1$, choose $n$ sufficiently large so that $\sup_{t\in[0,T],\,x\in[N-1,N+1]} \cn _{\beta }^{(1)}(f_n-f)(t,x)<\varepsilon$. Applying Minkowski's inequality, for every $(t,x)\in[0,T]\times[-N,N]$, we have
	\begin{align*}
		\cn_{\beta }^{(1)}(\tau_hf-f)(t,x)  &\le \cn_{\beta }^{(1)}(\tau_h f- \tau_hf_n)(t,x)+\cn_{\beta }^{(1)}(\tau_h f_n- f_n)(t,x)  +\cn_{\beta }^{(1)}(f_n-f)(t,x) \\
		&\le  2 \varepsilon+\cn_{\beta }^{(1)}(\tau_h f_n- f_n)(t,x).
	\end{align*}
	Since $f_n$ belongs to $X^\beta_T$, $\lim_{h\to0}\sup_{t\in[0,T],\,x\in[-N,N]} \cn_{\beta }^{(1)}(\tau_h f_n- f_n)(t,x)=0$ which implies $f$ belongs to $X^\beta_T$.
		\end{proof}

	The next results give some characterizations of the space $X^{\beta}_T$.
	\begin{lemma}\label{lem.barcn}
		Let $f:[0,T]\times\RR\to \RR$ be a continuous function such that $t\mapsto \cn_{\beta }^{(1)}f(t,x)$ is continuous for every fixed $x$. Suppose in addition that for every $R>0$,
			$$\lim_{\delta\downarrow0}\sup_{t\in[0,T],\,x\in[-R,R]}\int_{-\delta}^\delta |f(t,x+y)-f(t,x)|^2|y|^{-2 \beta-1}dy=0\,.$$
		Then $ \cn_{\beta }^{(1)}f$ is continuous and $f$ belongs to $X_T^\beta$.
	\end{lemma}
	\begin{proof}
		Fix $R,\ep>0$, and choose $\delta$ such that
		\begin{equation*}
			\sup_{t\in[0,T],\,x\in[-R-1,R+1]}\int_{-\delta}^\delta |f(t,x+y)-f(t,x)|^2|y|^{-2 \beta-1}dy<\varepsilon\,.
		\end{equation*}
		Then for every $t\in[0,T],\,x\in[-R,R]$ and $|h|\le 1$
		\begin{equation*}
			[\cn_{\beta }^{(1)}(\tau_hf-f)(t,x)]^2\le 2\varepsilon+\sup_{t\in[0,T],\,x\in[-R-1,R+1]} 2|\tau_hf(t,x)-f(t,x)| ^2\int_{|y|>\delta}|y|^{-2 \beta-1}d y\,.
		\end{equation*}
		Since $f$ is continuous, $\lim_{h\to0}\sup_{t\in[0,T],x\in[-R-1,R+1]} |\tau_hf(t,x)-f(t,x)|=0$. Together with the previous estimate, this yields $\lim_{h\to0}\sup_{t\in[0,T],\, x\in[-R,R]} \cn_{\beta }^{(1)}(\tau_hf-f)(t,x)=0$ which on one hand, together with \eqref{V.subadditive} implies the continuity of $ \cn_{\beta }^{(1)}f$. On the other hand, it obviously implies $f\in X_T^\beta$.
	\end{proof}

	\begin{proposition}\label{prop.Xbeta}
		Let $\phi\in C^\infty(\RR)$ be supported in $[-1,1]$, such that $\int_{\RR} \phi(x)dx=1$ and $0\le \phi\le 1$. Set  $\phi_n(x)=n \phi(nx)$. Then
		\begin{enumerate} %[label=(\roman*)]
			\item\label{X.mfier} If  $f\in X^\beta_T$, then $f*\phi_n\to f$ in $X^\beta_T$ as $n\to\infty$, where $*$ denotes the convolution with respect to the space variable.
			\item \label{X.dense} $C ^{0,1}([0,T]\times\RR)$ i.e., the space of functions which are continuous in time and continuously differentiable in space, is dense in $X^\beta_T $.
			\item\label{X.sup0} Suppose that $f$ is a continuous function on $[0,T]\times\RR$ such that $t\mapsto \cn_{\beta}^{(1)} f(t,x)$ is finite and continuous in time for every fixed $x\in\RR$. Then $f$ belongs to $X^\beta_T$ if and only if for every $R>0$
			\begin{equation}\label{cond.sup0}
				\lim_{\delta\downarrow0}\sup_{t\in[0,T],\,x\in[-R,R]}\int_{-\delta}^\delta|f(t,x+y)-f(t,x)|^2|y|^{-2 \beta-1}dy=0\,.
			\end{equation}
		\end{enumerate}
	\end{proposition}
	\begin{proof}
		We denote $f_n=f*\phi_n$. To show \eref{X.mfier}, we observe that
		\begin{eqnarray*}
			&&f_n(t,x+y)-f_n(t,x)-f(t,x+y)+f(t,x)\\
			&&\quad =\int_\RR[\tau_hf(t,x+y)-\tau_hf(t,x)-f(t,x+y)+f(t,x)]\phi_n(h)dh
		\end{eqnarray*}
		and hence, for every $x\in[-R,R]$, applying Jensen's inequality, we get
		\begin{align*}
			\int_{-1}^1 &|f_n(t,x+y)-f_n(t,x)-f(t,x+y)+f(t,x)|^2|y|^{-2 \beta-1}dy
			\\&\le\int_\RR \int_{-1}^{1}|\tau_hf(t,x+y)-\tau_hf(t,x)-f(t,x+y)+f(t,x)|^2|y|^{-2 \beta-1}   \phi_n(h)dhdy
			\\&\le \  \sup_{r\in[0,T],\,z\in[-R-1,R+1]} \sup_{|h| \le \frac 1n} [\cn_{\beta }^{(1)}(\tau_hf-f) (r,z)]^2 \,.
		\end{align*}
		By assumption $f$ belongs to $X_T^\beta$. Therefore, owing to condition (ii) in Definition \ref{def1}, this integral converges to 0 when $n\to\infty$. This proves item \eref{X.mfier}.

		To show \eref{X.dense}, we first prove that $X_T^\beta$ contains $C^{0,1}([0,T]\times\RR)$. Indeed, if $g$ is a function in $C^{0,1}([0,T]\times\RR)$, by dominated convergence theorem, it is easy to show that $ \cn_{\beta }^{(1)} g(t,x)$ is finite and continuous in time for every fixed $x$. Moreover, for every $R>0$, we have
		\begin{equation}\label{eq:bnd-V-C01}
			\sup_{t\in[0,T],\,x\in[-R,R]} \int_{-\delta}^{\delta}|g(t,x+y)-g(t,x)|^2|y|^{-2 \beta-1}dy\le \sup_{x\in [-R,R]} \|\partial_x g\|_{\infty}\int_{|y|\le \delta}|y|^{1-2 \beta}dy\,.
		\end{equation}
Since $\lim_{\delta\to 0} \int_{|y|\le \delta}|y|^{1-2 \beta}dy = 0$, Lemma~\ref{lem.barcn} implies that $g$ belongs to $X_T^\beta$. We have thus proved that $C^{0,1}\subset X_T^\beta$. Together with item \eref{X.mfier}, this yields item \eref{X.dense}.

		The sufficiency of \eref{X.sup0} is in fact the content of Lemma \ref{lem.barcn}. We focus on the necessity of~\eqref{cond.sup0}. Assume that $f$ belongs to $X^\beta_T$. Fix $R>0$, $\varepsilon>0$ and choose $g$ in $C^{0,1}$ so that
\begin{equation*}
\sup_{t\in[0,T],\,x\in[-R,R]} \cn_{\beta }^{(1)}(f-g)(t,x)< \varepsilon \, .
\end{equation*}
Then for every $\delta>0$ we have
\begin{multline}\label{tmp.fdelta}
\sup_{t\in[0,T],\,|x|\le R}\int_{-\delta}^\delta|f(t,x+y)-f(t,x)|^2|y|^{-2 \beta-1}dy  \\
\le 2 \varepsilon^2+2\sup_{t\in[0,T],\,|x|\le R}\int_{-\delta}^\delta|g(t,x+y)-g(t,x)|^2|y|^{-2 \beta-1}dy\,.
\end{multline}
		Since $g$ is $C^{0,1}$, the last term converges to 0 when $\delta\downarrow0$ (see relation \eqref{eq:bnd-V-C01}). Due to the fact that $\varepsilon$ can be chosen arbitrarily small, this implies that $f$ satisfies the condition \eqref{cond.sup0}.
	\end{proof}
	\begin{corollary}
			$X_T^\beta$ is a Polish (complete and separable)  space.
	\end{corollary}
	\begin{proof}
		Completeness comes from Proposition \ref{prop.Ncomplete}. For separability, we invoke Proposition ~\ref{prop.Xbeta}\eref{X.dense} and the fact that the functions in $C^{0,1}([0,T]\times \RR)$ can be approximated by polynomials  with rational coefficients, using a truncation argument.
	\end{proof}

	\begin{proposition} The inclusion $X_T^\beta\subset X_T^\alpha$ holds continuously for $\beta> \alpha$.
	\end{proposition}
	\begin{proof}
		Suppose $f$ belongs to $X_T^\beta$. Fix $n\ge1$. By Proposition \ref{prop.embed}, we see that
		\begin{equation*}
			\sup_{t\in [0,T],\, |x|\le n}|f(t,x+y)-f(t,x)| \leq C \sup_{t\in [0,T],\,|x|\le n+1}\cn_{\beta }^{(3)}f(t,x) |y|^\beta
		\end{equation*}
		for every $|y|\le 1$. Hence for every $t\le T$, $|x|\le n$ and $\al<\beta$ we have:
		\begin{align*}
			\int_{|y|\le 1}|f(t,x+y)-f(t,x)|^2|y|^{-2 \alpha-1}dy\leq C \sup_{t\in [0,T ,\,|x|\le n+1}\cn_{\beta }^{(3)}f(t,x)\, ,
		\end{align*}
which is a  finite quantity. The continuity of $(t,x)\mapsto \int_{|y|\le 1}|f(t,x+y)-f(t,x)|^2|y|^{-2 \alpha-1}dy$ follows at once from dominated convergence theorem.
	\end{proof}
	\noindent 
	We state an analogous result for $\XX^{\beta,p}_T$ without proof.
	\begin{proposition}\label{prop:XXcompare}
		The inclusion $\XX^{\beta,p}_T\subset \XX^{\alpha,q}_T$ holds continuously for $\beta>\alpha$ and $p\ge q$.
	\end{proposition}
	Next, we derive a compactness criterion for $X_T^\beta$. We first recall some well-known definitions and facts. An \textit{$\varepsilon$-cover} of a metric space is a cover of the space consisting of sets of diameter at most $\varepsilon$. A metric space is called \textit{totally bounded} if it admits a finite $\varepsilon$-cover for every $\varepsilon>0$. It is well known that a metric space is compact if and only if it is complete and totally bounded. The following lemma is the key ingredient for many compactness results.
	\begin{lemma}\label{lem.tbded}
		Let $X$ be a metric space. Assume that, for every $\varepsilon>0$, there exists a $\delta>0$, a metric space $W$, and a mapping $\Phi:X\to W$ such that $\Phi(X)$ is totally bounded, and for all $x,y\in X$ with $d(\Phi(x),\Phi(y))<\delta$, we have  $d(x,y)<\varepsilon$. Then $X$ is totally bounded.
	\end{lemma}
	The proof of this lemma is elementary, we refer readers to {Lemma 1} in \cite{HH}  for details. The following result  provides sufficient conditions for relative compactness in $X_T^{\beta}$.
	\begin{proposition}\label{prop.compX}
		A set $\mathfrak{F}$ in $X_T^{\beta}$ is relatively compact if
		\begin{enumerate} [label={\rm [A\arabic*]}]
			\item\label{cond.a1} $\displaystyle\sup_{f\in\mathfrak{F}}|f(0,0)|$ is finite.
			\item \label{cond.a2} For every fixed $x\in\RR$, $\{f(\cdot,x):f\in\mathfrak F\}$ is equicontinuous in time.
			\item\label{cond.a3} For every $R>0$, $\displaystyle\lim_{\delta\downarrow 0}\sup_{f\in\mathfrak F}\sup_{t\in[0,T]\,x\in [-R,R]}\int_{-\delta}^{\delta}  \frac{| f(t,x+y)- f(t,x)|^2}{|y|^{1+2 \beta}} \, dy =0$.
		\end{enumerate}
	\end{proposition}

\begin{proof}
Suppose that $\mathfrak{F}$ satisfies the three conditions. We first observe that  condition \ref{cond.a3} together with \eqref{ineq.embed} implies the following equicontinuity property. For every $R>0$ and $\varepsilon>0$, there exists $\eta>0$ such that
\[
\sup_{t\in[0,T]} |f(t,x)-f(t,y)|<\varepsilon
\]
 whenever 	
$f \in \mathfrak{F}$ and $x, y \in [-R,R]$  satisfy  $|x-y|<\eta$. Together with \ref{cond.a2}, this implies equicontinuity for $\mathfrak F$ in $(t,x) \in [0,T]\times [-R,R]$. Indeed, take $N$ to be a sufficiently large integer, and   set $x_i = -R + \frac{j}{N}R$, $j= 0,1,\dots, 2N$. According to \ref{cond.a2}, $\{ f(\cdot, x_i): f \in \mathfrak {F}\}$ is equicontinuous in time, uniformly for $j=0,1,\dots, 2N$. By writing
\begin{equation*}
 |f(t,x)-f(s,x)| \leq |f(t,x)-f(t,x_i)| + |f(t,x_i)-f(s,x_i)|+ |f(s,x_i)-f(s,x)|\,,
\end{equation*}
where $x_i$ is chosen in such a way that $|x-x_i|< \eta$, this shows the uniformity in $x$.
		
		Fix now $R>0$ and $\varepsilon>0$. From \ref{cond.a3}, we can choose a positive number $\delta_1=\delta_1(\varepsilon)$, such that $\delta_1<1$ and
		\begin{equation*}
			2 \sup_{f\in\mathfrak F}\sup_{t\in[0,T],\,x\in [-R,R]}
			\int_{-\delta_1}^{\delta_1}
			\frac{| f(t,x+y)- f(t,x)|^2}{|y|^{1+2 \beta}} \, dy < \varepsilon^2\,.
		\end{equation*}
We now choose $\delta_2\le \varepsilon$ satisfying
		\begin{equation*}
		 	2 (3 \delta_2)^2 \int_{|y|>\delta_1} \frac{dy}{|y|^{1+2 \beta}}< \varepsilon^2 \,.
		\end{equation*}
 By the equicontinuity, we can also choose a positive number $\eta=\eta(\varepsilon)$, $\eta<1$, such that
\begin{equation}\label{eq:incr-f-less-delta2}
\|f(t,x)-f(s,y)\|< \delta_2,
\end{equation}
whenever $f\in\mathfrak {F}$ and $(t,x),(s,y)\in [0,T]\times[-R-2,R+2]$ satisfy  $|t-s|+|x-y|< \eta$. Since $[0,T]\times[-R-2,R+2]$ is compact, we can find a finite set of points $\{(t_a,x_i):1\le a,i\le n\}$ in $[0,T]\times[-R-2,R+2]$ such that for every $(t,x)\in [0,T]\times[-R-1,R+1]$, there is some $(t_a,x_j)$ so that $|t-t_a|+ |x-x_j|<\eta$ and $[x_j-1,x_j+1]\subset [-R-2,R+2]$.

		Define $\Phi:\mathfrak {F}\to \RR^{n^2}$ by
		\begin{equation*}
			\Phi(f)=(f(t_a,x_i): 1\le a,i\le n)\,.
		\end{equation*}
		Condition \ref{cond.a1} and equicontinuity  imply that the image $\Phi(\mathfrak {F})$ is bounded and thus totally bounded in $\RR^{n^2}$.
Furthermore, consider $f,g\in\mathfrak {F}$ with $\|\Phi(f)-\Phi(g)\|_\infty< \delta_2$. Resorting to the fact that for any $(t,x)\in [0,T]\times[-R-1,R+1]$ there are some $a,j$ so that $|t-t_a|+ |x-x_j|<\eta$, we can write
		\begin{equation*}
			|f(t,x)-g(t,x)|\le |f(t,x)-f(t_a,x_j)|+|f(t_a,x_j)-g(t_a,x_j)|+|g(t_a,x_j)-g(t,x)|\le 3 \delta_2\,,%\le3 \varepsilon\,,
		\end{equation*}
where we bounded the first and third term on the right hand side thanks to \eqref{eq:incr-f-less-delta2}, and the second one according to the fact that $\|\Phi(f)-\Phi(g)\|_\infty< \delta_2$. We end up with
\begin{equation*}
\sup_{t\in[0,T],\,x\in [-R-1,R+1]}|f(t,x)-g(t,x)|\le3 \delta_2  \le 3\varepsilon \,.
\end{equation*}
In addition, for every $(t,x)\in[0,T]\times[-R,R]$ we have
\begin{eqnarray*}
[\cn_{\beta }(f-g) (t,x)]^2 &\le& 2\sup_{h \in \{f,g\}} \int_{|y|\le \delta_1}| h(t,x+y)- h(t,x)|^2\frac{dy}{|y|^{1+2 \beta}}
\\&&+2\sup_{r\in[0,T], \,z\in[-R-1,R+1]}|f(r,z)-g(r,z)|^2\int_{|y|>  \delta_1}\frac{dy}{|y|^{1+2 \beta}}
\le 2\varepsilon^2\,.
\end{eqnarray*}
		Therefore, by the definition of the metric    on $X^\beta_T$ (see   \eref{metric.d})  and Lemma~\ref{lem.tbded}, the set $\mathfrak {F}$ is totally bounded in $X_T^{\beta}$.
	\end{proof}

	A useful consequence of the previous proposition is the following corollary.
	\begin{corollary} Suppose $\alpha> \beta$.
		Let $\mathfrak {F}$  be a subset of $X^\alpha_T$ such that $\mathfrak {F}$ is equicontinuous in time for every fixed $x$,  $\displaystyle\sup_{f\in\mathfrak{F}}|f(0,0)|<\infty$ and $\sup_{f\in\mathfrak {F}}\sup_{t\in [0,T],\,|x|\le R} \cn_{\alpha }^{(1)}f (t,x)<\infty$   for every positive $R$. Then $\mathfrak {F}$ is relatively compact in $X^\beta_T$.
	\end{corollary}
	\begin{proof}
		It suffices to check that $\mathfrak {F}$ satisfies condition \ref{cond.a3}. Applying \eqref{ineq.embed}, for $\delta$ small enough, the assumption on $\mathfrak {F}$ implies
		\begin{equation*}
			\sup_{f\in\cf}\sup_{t\in [0,T],\, |x|\le R} |f(t,x+y)-f(t,x)|\leq C |y|^{\alpha} \, ,
		\end{equation*}
		for all $|y|\le \delta$. Hence,
		\begin{align*}
			\sup_{f\in\cf}\sup_{t\in [0,T],\, |x|\le R}\int_{|y|\le \delta} |f(t,x+y)-f(t,x)|^2|y|^{-2 \beta-1}dy \leq C \int_{|y|\le \delta} |y|^{2(\alpha- \beta)-1}dy \, ,
		\end{align*}
		which clearly implies \ref{cond.a3} since $\alpha>\beta$.
	\end{proof}

The following result  provides  sufficient conditions for relative compactness in  $\XX^\beta_T(B)$.
Its proof is   completely analogous to that of Proposition \ref{prop.compX}  and is omitted for the sake of conciseness.
	\begin{proposition}\label{prop.compbarXX}
		Suppose that a set $\mathfrak {F}$ in $ \XX^{\beta}_T(B)$ satisfies the following properties.
		\begin{enumerate}%[label=(\roman*)]
			\item\label{cond.xx1} For every $t\in[0,T]$ and $x\in\RR$, $\mathfrak {F}(t,x): =\{f(t,x):f\in\mathfrak {F}\} $ is relatively compact in the Banach space $B$.
			\item \label{cond.xx2} For every fixed $x\in\RR$, $\{f(\cdot,x):f\in\mathfrak {F}\}$ is equicontinuous in time.
			\item\label{cond.xx3} For every $R>0$, we have
\begin{equation*}
\lim_{\delta\downarrow 0}\sup_{f\in\mathfrak {F}}\sup_{t\in[0,T],\,x\in [-R,R]}\int_{-\delta}^{\delta}
 \frac{\| f(t,x+y)- f(t,x)\|^2}{|y|^{1+2 \beta}} \, dy =0 \, .
\end{equation*}
		\end{enumerate}
		Then $\mathfrak {F}$ is relatively compact in $\XX^\beta_T(B)$.
	\end{proposition}

In order to handle the nonlinearity in equation \eqref{spde with sigma}, the following composition rule is crucial.
	\begin{proposition}[Left composition]\label{prop.leftcomp}
		Let $\sigma$ be a Lipschitz function on $\RR$ and let  $f$ be a function in $X^\beta_T$. Suppose that for every fixed $x$, the map $t\mapsto  \cn_{\beta }^{(1)}  \sigma( f) (t,x)$ is continuous. Then $\sigma(f)$ belongs to $X^\beta_T$. Furthermore, if $f_n$ is a sequence converging to $f$ in $X^\beta_T$, then for every positive $R$ and for any $\delta>0$, we have
		\begin{equation*}
			\lim_{n\rightarrow \infty }\sup_{t\in [0,T],\, |x|\le R} \cn_{\beta}^{(\delta)}(\sigma( f_n)- \sigma( f))(t,x)=0\,.
		\end{equation*}
	\end{proposition}
	\begin{proof}
		We first show that $\sigma(f)$ belongs to $X^{\beta}_T$. For any $\delta>0$ we have
		\begin{equation*}
			\int_{|y|\le \delta}|\sigma(f(t,x+y))-\sigma(f(t,x))|^2 |y|^{-2 \beta-1}dy \le \|\sigma\|_{\rm Lip}^2 [\cn_{\beta}^{(\delta)} f(t,x)]^2
		\end{equation*}
		which together with the criterion \eref{X.sup0} in Proposition~\ref{prop.Xbeta} implies that $\sigma(f)$ belongs to $X^\beta_T$.

		For the second assertion, for every positive $R$ and any $\varepsilon>0$,  we can  choose $\delta_0>0$ and $n_0>0$,  so that, for any $n\ge n_0$,
\begin{equation}\label{eq:bnd-V-sigma-f}
\sup_{t\in [0,T],\, |x|\le R}  \cn_{\beta}^{(\delta_0)} (\sigma(f_n)-\sigma(f))(t,x)
\le \ep \, .
\end{equation}
Indeed, it is easily seen that
		\begin{align*}
			  \cn_{\beta}^{(\delta_0)}(\sigma(f_n)-\sigma(f))(t,x) & \leq  \cn_{\beta}^{(\delta_0)} \sigma(f_n) (t,x) +   \cn_{\beta}^{(\delta)}\sigma(f) (t,x) \\
			& \leq  \|\sigma\|_{\rm Lip}  \left(     \cn_{\beta}^{(\delta_0)}f_n(t,x) +   \cn_{\beta}^{(\delta_0)}f(t,x) \right) \\
			&\leq \|\sigma\|_{\rm Lip}  \left(   \cn_{\beta}^{(\delta_0)}(f_n- f)(t,x) + 2  \cn_{\beta}^{(\delta_0)}f(t,x) \right)\,,
		\end{align*}
and the last term is readily bounded by $\ep$ if $\delta_0$ is chosen small enough. Now with \eqref{eq:bnd-V-sigma-f} in hand we obtain, for any $\delta>0$,
		\begin{multline*}
			\sup_{t\in [0,T],\, |x|\le R} \cn_{\beta}^{(\delta)} (\sigma(f_n)-\sigma(f))(t,x)  \\
			\leq C \, \varepsilon+ C \, \|\sigma\|_{{\rm Lip}} \sup_{t\in [0,T],\, |x|\le R+1} |f_n(t,x)-f(t,x)| \left(\int_{|y|>\delta_0}|y|^{-2 \beta-1}dy\right)^{\frac 12}\,.
		\end{multline*}
		We conclude the proof by taking the limit as $n$ tends to infinity.
	\end{proof}

		The next lemma gives a criterion for a process in $\mathfrak{X}_T^{\alpha, p}$ to have its paths almost surely lie in the space $X_T^{\beta}$ for a certain value of $\beta$.
	
	\begin{lemma}\label{lem.lawXX}
		Let $f$ be a stochastic process in $\mathfrak{X}^{\alpha,p}_T$ with $p \alpha>1$. Assume that for any $R>0$,
		\begin{equation}  \label{E7}
		\sup_{s,t\in [0,T]} \sup_{|x|\le R} \| f(t,x) -f(s,x)\|_{L^p(\Omega)} \le C_R |t-s|^\lambda,
		\end{equation}
		where $\lambda>p^{-1}$.
	  Then $f$ has a version $\tilde{f}$ such that with probability one, $\tilde{f}$ belongs to $X^\beta_T$ for every $\beta<\alpha-\frac1 p$.
	\end{lemma}
	\begin{proof}
		Since $f$ belongs to $\mathfrak{X}^{\alpha,p}_T$, inequality \eqref{ineq.embed} implies
		\begin{equation*}
			\sup_{t\in [0,T]}\sup_{x,y\in\RR}\frac{\|f(t,x+y)-f(t,x)\|_{L^p (\Omega)} }{|y|^{\alpha}} \leq C \sup_{t\in [0,T] ,\, x\in\RR} \int_\RR \|f(t,x+y)-f(t,x)\|^2_{L^p(\Omega)}|y|^{-2 \alpha-1}dy\,.
		\end{equation*}
		Then by Kolmogorov continuity criterion, $f$ has a version $\tilde{f}$ such that with probability one, $\tilde{f}$ satisfies
		\begin{equation*}
		 	\sup_{s,t\in [0,T] ,\, |x|\le R} |\tilde{f}(t,x+y)-\tilde{f}(s,x)|\leq C (|y|^{\beta'}+ |t-s|^{\lambda'})
		\end{equation*}
		for every $R$ and $|y|\le 1$, where $\beta'$ and $\lambda'$  are fixed and  such that $\beta<\beta'<\alpha- 1/p$ and $\lambda <\lambda' <\lambda - 1/p$. This implies that a.s. $ \cn_{\beta} ^{(1)} f(t,x)$ is finite and  $\cn_{\beta}^{(\delta)}$ satisfies   condition~\eqref{cond.sup0}. The continuity of $\cn_{\beta }^{(1)}f$ follows from dominated convergence theorem. These facts imply that  $\tilde{f}$ belongs to $X^\beta_T$ almost surely.
	\end{proof}

\subsection{Probability measures on $X^{\beta}_T$} % (fold)
\label{sub:probability_measures}
	To show the existence of solution to equation \eref{spde with sigma} we need some tightness arguments for some probability measures defined on $X_T^{\beta}$.   We have the following result towards this aim.

	%$X^\beta([0,T])$.
	\begin{theorem}\label{thm.tight}
		Let $\{\bp_n, \, n\ge 1\}$ be a sequence of probability measures on $X_T^\beta$. This sequence is tight if the following three conditions hold:
		\begin{enumerate} %[label=\textbf{[P\arabic*]}]
			\item\label{P.bded} For each positive $\eta$, there exist $a$ and $n_0$ such that for all $n\ge n_{0}$
			\begin{equation}
				\bp_n(f \in X_T^{\beta}:|f(0,0)|\ge a)\le \eta \,.
			\end{equation}
			\item\label{P.time} For every $x\in\RR$, and every positive $\varepsilon$ and $\eta$, there exist  $\delta$ satisfying $0<\delta<1$, and $n_0$ such that for all $n\ge n_{0}$
			\begin{equation}
				\bp_n\left(f \in X_T^{\beta} : \sup_{s,t\le T,|t-s|<\delta} |f(t,x)-f(s,x)| \ge \varepsilon \right)\le \eta	 \,.
			\end{equation}
			\item\label{P.space} For every $R>0$, for each positive $\varepsilon$ and $\eta$, there exist $\delta\in(0,1)$ and $n_0$ such that for all $n\ge n_{0}$
			\begin{equation}
				\bp_n\left(f \in X_T^{\beta} :\sup_{t\in [0,T],\, |x|\le R}\int_{-\delta}^\delta|f(t,x+y)-f(t,x)|^2|y|^{-2 \beta-1}dy\ge \varepsilon \right)\le \eta  \,.
			\end{equation}
		\end{enumerate}
	\end{theorem}
	\begin{proof}
Without loss of generality we assume $n_0=1$.  For a given   $\eta>0$, we choose $a$ so that $\bp_{n}(B^c)\le \eta$ for all $n \geq 1$, where
     \begin{equation*}
     B=\left \{ f \in X_T^{\beta}: |f(0,0)|< a \right\}\,.
     \end{equation*}
 According to condition \eqref{P.space}, for any integer $k,N$, we also choose and fix $\delta_{k,N}$ such that $\bp_{n}(A_{k,N}^c)\le \eta 2^{-k-N}$ for all $n \geq 1$, where
\begin{equation*}
A_{k,N}=\left \{  f \in X^{\beta}_T: \sup_{t\in [0,T],\, |x|\leq N}\int_{-\delta_{k,N}}^{\delta_{k,N}} |f(t,x+y)-f(t,x)|^2|y|^{-2 \beta-1}dy\le \frac{1}{k^2} \right\}\,.
\end{equation*}
Then for each $\tilde{x} \in  [-N, N] \cap \frac{\delta_{k,N}}{3}\mathbb{Z}$, where $\mathbb{Z}$ is the set of integers (note that the number of such $\tilde{x}$ has order $\frac{N}{\delta_{k,N}}$), we choose $\delta^{\prime}_{k,N}(\tilde{x})$ according to condition \eqref{P.time} such that $\bp_{n}(B^c_{k,N} (\tilde{x}))\le \delta_{k,N} \eta 2^{-k-N}$, where
\begin{equation*}
B_{k,N}(\tilde{x})=\left \{f \in X_T^{\beta}: \sup _{t,s, \leq T, |t-s|\leq \delta^{\prime}_{k,N}(\tilde{x})} |f(t,\tilde{x})-f(s,\tilde{x})|\leq \frac{1}{k^2} \right\}\, .
\end{equation*}
Consider now $B_{k,N}=\cap_{\tilde{x} \in [-N,N] \cap \frac{\delta_{k,N}}{3}\mathbb{Z}} B_{k,N}(\tilde{x})$. It is easy to see that
\begin{eqnarray*}
\bp_n(B_{k,N}^c)\leq \sum_{\tilde{x} \in [-N,N]\cap \frac{\delta_{k,N}}{3}\mathbb{Z}}\bp_n (B^c_{k,N}(\tilde{x})) \leq C \frac{N}{\delta_{k,N}}\eta \delta_{k,N} 2^{-k-N}= C \eta 2^{-k-N}N \,.
\end{eqnarray*}
We thus set $A = \cap_{k,N} (A_{k,N}\cap B_{k,N}) \cap B$. Then according to Proposition \ref{prop.compX} we see that the closure of $A$ is compact in $X_T^{\beta}$, and $\bp_n(A) \geq 1-C\eta$. This shows the tightness of $\bp_n$.
	\end{proof}

The following proposition states that under some moment conditions, a sequence of processes $u_n$ can be regarded as a tight sequence of probability measures on the space $X_T^{\beta}$. 	
	
	\begin{proposition}\label{prop.utight}
		Assume that $\alpha, \lambda\in(0,1)$ and $p\ge 1$ satisfy $p\alpha >1$, $p \lambda>1$ and $\beta<\alpha-1/p$. Let $\{u_n , \, n \ge 1\}$ be a sequence of stochastic processes such that
		\begin{enumerate} % [label=(\roman*)]
				\item $\displaystyle \lim_{\delta\to \infty}\limsup_n \bp(|u_n(0,0)|> \delta )=0$\,,
				\item For every $R>0$,   $\displaystyle\sup_{n} \sup_{s,t\in [0,T], |x| \le R}\|u_n(t,x)-u_n(s,x)\|_{L^p(\Omega)}\leq C_R |t-s|^{\lambda} $\,,
				\item $\displaystyle\sup_{n}\|u_n\|_{\mathfrak{X}^{\alpha,p}_T}$ is finite\,.
		\end{enumerate}	
		From Lemma~\ref{lem.lawXX}, the law of $u_n$ can be considered as a probability measure on $X^\beta_T$. In addition, as probability measures on $X^\beta_T$, the sequence $\{u_n , \, n \ge 1\}$ is tight.
	\end{proposition}
	\begin{proof}
	This proposition can be easily proved using the same ideas as in the proof of Lem\-ma~\ref{lem.lawXX} and Theorem \ref{thm.tight}, we omit the details.
	\end{proof}

\subsection{Existence of the solution}

The main result of this subsection is the     existence of a solution for equation  \eref{spde with sigma}.   The methodology, inspired by the work of Gy\"ongy \cite{Gyo} on semilinear stochastic partial differential equations,  consists in proving tightness of a sequence of solutions obtained by regularizing the noise, and then using the uniqueness result.  The space $\cz_T^p$, where we proved our uniqueness result, consists of $L^p(\RR)$-valued processes, and  it is not clear how to characterize compactness of probability laws  on the space of trajectories of these processes. For this reason, we prove  the existence of a solution with paths in  the space $X^{\frac 12-H}_T$ introduced in Definition \ref{def1}, equipped with the metric \eref{metric.d}.

\begin{theorem}\label{thm:exist with sigma}
Assume that for equation \eref{spde with sigma} the following conditions hold:
 \begin{enumerate}
 \item For some $\beta_0> \frac{1}{2}-H$ and some   $p>\max (\frac 6 {4H-1}, \frac{1}{\beta_0+H-{1}/{2}})$, the initial condition $u_0$ is in $L^p(\RR)\cap L^\infty(\RR)$ and  
 \begin{equation}
   \sup_{x\in\RR}\cn_{\beta_0}u_0(x)+ \left(\int_{\RR}\|u_0(\cdot)-u_0(\cdot+h)\|^2_{L^p(\RR)}|h|^{2H-2}dh\right)^{\frac12} < \infty\,.
 \end{equation}
 \item $\sigma$ is differentiable and the derivative of $\sigma$ is Lipschitz and $\sigma(0)=0$.\end{enumerate}
 Then there exists a solution $u$ to \eref{spde with sigma} in $\mathcal{Z}_T^p\cap \XX^{\frac12-H,p}_T$. In addition, the solution has sample paths in the space $X^{\frac{1}{2}-H}_T$.
 \end{theorem}

\begin{proof}

 As mentioned above, we follow the methodology developed in \cite{Gyo} and we consider a regularization of the noise in space. Indeed, for $\ep>0$ and  $\vp\in  \HH $,  we define
 	\begin{equation}\label{eq:cov-W-epsilon}
	W_{\varepsilon}(\varphi)
	= \int_0^t \int_{\mathbb{R}} [\rho_{\ep}*\varphi](s,x)W(ds,dy)
	=\int_0^t \int_{\mathbb{R}}\int_{\mathbb{R}}\varphi(s,x)\rho_{\varepsilon}(x-y)W(ds,dy)dx\,,
	\end{equation}
	where $\rho_t (x)=(2\pi t)^{-\frac{1}{2}} e^{-{x^2}/{2t}}$. Notice that relation \eqref{eq:cov-W-epsilon} can be also read (either in Fourier or direct coordinates) as:
	\begin{eqnarray}\label{eq:ident-cov-W-ep}
	\be\left[W_{\varepsilon}(\varphi) W_{\varepsilon}(\psi) \right]
	&=&
	c_{1,H} \int_0^t \int_{\mathbb{R}}
	\mathcal{F}\varphi(s,\xi)\, \overline{\mathcal{F}\psi(s,\xi)} \, e^{-\varepsilon |\xi|^2} |\xi|^{1-2H} d\xi ds   \\
	&=&
	c_{1,H} \int_0^t \int_{\mathbb{R}}\int_{\mathbb{R}}\varphi(s,x)f_{\varepsilon}(x-y)\psi(s,y) \, dx   dy   ds,  \notag
	\end{eqnarray}
	where $f_{\ep}$ is given by $f_{\varepsilon}(x)=
	\mathcal{F}^{-1}(e^{-\varepsilon |\xi|^2} |\xi|^{1-2H})$. In other
	words, our noise is still a white noise in time but its space
	covariance is now given by $f_{\ep}$. Note that $f_{\varepsilon}$
	is a  {real} positive definite function, but is not necessarily
	positive. As assessed by \eqref{eq:ident-cov-W-ep}, we however have
	\begin{equation}\label{ineq.WeW}
		\be\lc |W_\varepsilon(\varphi)|^2\rc \le \be\lc |W(\varphi)|^2 \rc \, ,
	\end{equation}
	for all $\varphi$ in $\HH$.

	For every fixed $\varepsilon>0$, the noise $W_{\ep}$ induces an approximation to equation \eref{eq:mild-formulation sigma}, namely
	\begin{equation}\label{appr eq sigma}
		u_{\ep}(t,x)=p_t u_0(x) + \int_0^t \int_{\mathbb{R}}p_{t-s}(x-y)\sigma(u_{\ep}(s,y)) \, W_{\ep}(ds,dy) ,
	\end{equation}
	where the integral is understood in the It\^o sense. Since $|\xi|^{1-2H}e^{-\varepsilon |\xi|^2}$ is in $L^1(\RR)$, $|f_{\varepsilon}|$ is bounded. Thus, using Picard iteration, it is easy to see that \eqref{appr eq sigma} has a unique random field solution, and by estimating the $p$th moment of $|u_{\varepsilon}(t,x)-u_{\varepsilon}(t,x')|$, we see that each solution $u_{\varepsilon}(t,x)$ is H\"older continuous in space with order $\beta$ for all $\beta \in (0,1)$. Therefore we conclude that  $u_\varepsilon$ is in $\XX^{\beta,p}_T$ for all $\beta\in(0,1)$.
	  We remark that $\|u_\varepsilon\|_{\XX^{\beta,p}_T}$ may not be bounded uniformly in $\varepsilon$  as seen from this procedure.
{However, using \eqref{ineq.WeW},  \eqref{est.normXX} for suitable parameters yielding a contraction, plus the norm equivalence stated in Remark \ref{rmk:norm-X-beta-p} item (iv), we obtain the following uniform bound:}
	\begin{align*}
		\sup_{\varepsilon>0} \|u_{\varepsilon}\|_{\XX^{\beta,p}_T}<\infty \, ,
	\end{align*}
	for all $\beta\le \beta_0$ and $\beta<H$. In particular, because $\frac12-H<\beta_0-\frac1p$, we can choose $\beta $ such that $\frac12-H<\beta - \frac1p$.
In addition, we can show that $u_\varepsilon$ satisfies Condition (2) in Proposition~\ref{prop.utight}. With these properties, we can check that the three conditions in Proposition~\ref{prop.utight} are satisfied. Hence the  laws of the processes $u_\varepsilon$, considered as probability measures on the space $X^{\frac 12-H}_T$,    are tight and hence weakly relatively compact.

	We now base our final considerations on the forthcoming Lemmas \ref{lemG} - \ref{int conv lemma}.  Fix a sequence $\varepsilon_n $ converging to zero and set
	$u_n= u_{\varepsilon_n}$.
We shall hinge on Lemma \ref{Gyongy lemma} in order to prove that the sequence $u_n$ actually converges in probability. To apply this lemma, we consider now two sequences $  u_{m(n)}$ and $ u_{l(n)}$, where $\{m(n), n\ge 1\} $ and $\{ l(n), n\ge1 \}$ are strictly increasing sequences of positive integers.  For each $n \ge 1$, the triplet $(u_{m(n)},u_{l(n)},W)$ defines probability measure on the space
$$
\mathcal{B}:=X^{\frac12-H}_T\times X^{\frac12-H}_T\times C_{\rm uc}([0,T]\times \RR).
$$
Since the family $\{u_ \varepsilon, \, \ep >0\}$ is weakly relatively compact, there exists a subsequence  of the form $\{(u_{m(n_k)},u_{l(n_k)},W), k \ge 1 \}$ which converges in distribution as $k $ tends to infinity. Thus, by Skorokhod embedding theorem, there is a probability space $(\Omega',\mathcal{F}',\bp')$ and a sequence of random elements $z_k=(u'_{m(n_k)},u'_{l(n_k)},W')$ with values on $\mathcal{B}$ such that $z_k$ has the same distribution as $(u_{m(n_k)},u_{l(n_k)},W)$ and $z_k$ converges almost surely (in the topology of $\mathcal{B}$) to $(u',v',W')$. By Lemma \ref{int conv lemma} we see that both $u'$ and $v'$ are solutions to equation \eref{eq:mild-formulation sigma}, with $W$ replaced by $W'$. Then by Lemma~\ref{lemma: u in Z}  and the uniqueness result Theorem \ref{thm:uniqueness} we thus get that $u'=v'$ in $X_T^{\frac 12 -H}$. We can now apply Lemma \ref{Gyongy lemma} in order to assert that $u_{n}$ converges to some random field $u$ in $X_T^{\frac 12-H}$, in probability. Moreover, taking a subsequence if necessary, we see that $u_{n}$ converges to $u$ in $X_T^{\frac 12-H}$ a.s. Hence, thanks to another application of Lemma~\ref{int conv lemma} we see that $u$ satisfies equation~\eref{eq:mild-formulation sigma}. This proves the existence of the solution.
\end{proof}

We now state the lemmae on which the proof of Theorem \ref{thm:exist with sigma} relies.
The first lemma is a version of Gronwall's lemma, borrowed from \cite[Lemma 15]{Dal}, and the correction~\cite{Dal1} to this paper.

\begin{lemma}\label{lemG}
Let $g\in L^{1}(\ott;\R_{+})$ and consider a sequence of functions $\{f_{n}; \, n\ge 0\}$ with $f_{n}:\ott\to\R_{+}$, such that $f_{0}$ is bounded and for all $n\ge 1$
\begin{equation}\label{eq:gronwall-ineq}
f_{n}(t) \le c_{1} + c_{2} \int_{0}^{t} g(t-s) \, f_{n-1}( s) \, ds,
\end{equation}
for two positive constants $c_{1},c_{2}$.
Then $\sup_{n\ge 1} f_{n}$ is bounded. If we assume moreover that
$c_{1}=0$ in inequality \eqref{eq:gronwall-ineq}, we obtain that $\sum_{n\ge 0} f_{n}^{1/p}$
converges uniformly in $\ott$, for all $1\le p<\infty$.
\end{lemma}

The second lemma is a general result on convergence of random variables borrowed from \cite{GK,Gyo}.
	
		\begin{lemma}\label{Gyongy lemma}
	Let $\mathbb{E}$ be a Polish space equipped with the Borel $\sigma$-algebra. A sequence of $\mathbb{E}$-valued random elements $z_n$ converges in probability if and only if for every pair of subsequences $z_{l(n)}$, $z_{m(n)}$ there exists a subsequence $w_k:=(z_{l(n_k)},z_{m(n_k)})$ converging weakly to a random element $w$ supported on the diagonal $\{(x,y) \in \mathbb{E}\times \mathbb{E}: x=y\}$.
	\end{lemma}

The next result asserts that the approximate solution to the stochastic heat equation is uniformly bounded in the space $\mathcal{Z}_T^{p}$ defined by \eqref{def: space Z}.
\begin{lemma}\label{lemma: u in Z}
The approximate solutions $u_{\varepsilon}$  satisfy the condition
\begin{equation}
\sup_{\varepsilon>0}\|u_{\varepsilon}\|_{\mathcal{Z}_T^{p}} < \infty\,.
\end{equation}
Furthermore, if $u_{\varepsilon} \to u$ in $X_T^{\frac{1}{2}-H}$ a.s., as $\varepsilon$ tends to zero,  then $u$ is also in $\mathcal{Z}_T^{p}$.
\end{lemma}

\begin{proof}
  We will use Picard iteration to show that for each $\varepsilon$, $u_{\varepsilon} \in \mathcal{Z}_T^{p}$. Then we will use Gronwall's lemma to show that the processes  $u_{\varepsilon}$ are uniformly (in $\ep$) bounded in $\mathcal{Z}_T^{p}$. To this end, we first define
\begin{equation*}
u^{0}_{\varepsilon}(t,x)=p_tu_0(x)\,,
\end{equation*}
and recursively
\begin{equation*}
u^{n+1}_{\varepsilon}(t,x)=p_tu_0(x)+\int_0^t \int_{\RR} p_{t-s}(x-y)\sigma(u_{\varepsilon}^{n}(s,y))W_{\varepsilon}(ds,dy)\,.
\end{equation*}
We wish to bound $\|u_{\varepsilon}^{n}\|_{\mathcal{Z}_T^{p}}$ uniformly in $n$. First recall that
\[
\|u_{\varepsilon}^{n}\|_{\mathcal{Z}_T^{p}} = \sup_{t \in [0,T]} \|u_{\varepsilon}^{n}(t,\cdot)\|_{L^p(\Omega \times \RR)}  + \sup_{t \in [0,T]}  \cn^*_{\frac 12-H, p} u_{\varepsilon}^{n}(t),
\]
where  $\cn^*_{\frac 12-H, p}$ is defined in  \eref{def: space Z}.
  Let us now bound the terms $ \|u_{\varepsilon}^{n}(t,\cdot)\|_{L^p(\Omega \times \RR)} $ and $\cn^*_{\frac 12-H, p} u_{\varepsilon}^{n}(t)$.

\noindent
\textit{Step 1.}  We shall bound $ \|u_{\varepsilon}^{n}(t,\cdot)\|_{L^p(\Omega \times \RR)} $  uniformly in $n$ by considering the differences of Picard's iterations.
Indeed, by Burkholder's inequality we have
\begin{align*}
& \be  |u^{n+1}_{\varepsilon}(t,x)-u_{\varepsilon}^{n}(t,x)|^p  \\
& =
\be  \left |\int_0^t \int_{\RR}p_{t-s}(x-y)[\sigma(u_{\varepsilon}^{n}(s,y))-\sigma(u_{\varepsilon}^{n-1}(s,y))]W_{\varepsilon}(ds,dy)\right |^p   \\
&\leq C_p \, \be  \Big|\int_0^t \int_{\RR} p_{t-s}(x-y)p_{t-s}(x-z)[\sigma(u_{\varepsilon}^{n}(s,y))-\sigma(u_{\varepsilon}^{n-1}(s,y))] \\
&\hspace{2in} \times[\sigma(u_{\varepsilon}^{n}(s,z))-\sigma(u_{\varepsilon}^{n-1}(s,z))] f_{\varepsilon}(y-z)dy dz ds \Big|^{\frac{p}{2}}  \, .
\end{align*}
Thus, since $\|f_{\ep}\|_{\infty}\le C_{\ep}$ and owing to the fact that $\si$ is a Lipschitz function, we have
\begin{eqnarray*}
&&\be   |u^{n+1}_{\varepsilon}(t,x)-u_{\varepsilon}^{n}(t,x)|^p  \\
&\leq&C_{\varepsilon} \, \be  \left | \int_0^t \left ( \int_{\RR} p_{t-s}(y)|u_{\varepsilon}^{n}(s,x+y)-u_{\varepsilon}^{n-1}(s,x+y)| dy \right)^2 ds\right|^{\frac{p}{2}} \,,
\end{eqnarray*}
where $C_\ep$ denotes a generic constant depending on $\ep$ and $p$.
We now integrate with respect to the space variable and invoke Minkowski's inequality. In this way we obtain
\begin{eqnarray*}
&&\be \left \|u_{\varepsilon}^{n+1}(t,\cdot)-u_{\varepsilon}^{n}(t,\cdot) \right\|^p_{L^p(\RR)}\\
&\leq& C_{\varepsilon} \be \left \|\int_0^t \left ( \int_{\RR} p_{t-s}(y) |u_{\varepsilon}^{n}(s,y+\cdot)-u_{\varepsilon}^{n-1}(s,y+\cdot)|dy \right)^2 ds \right\|^{\frac{p}{2}}_{L^{\frac{p}{2}}(\RR)}\\
&\leq&C_{\varepsilon} \be \left ( \int_0^t \left (\int_{\RR} p_{t-s}(y)\left \| u_{\varepsilon}^{n}(s,\cdot)-u_{\varepsilon}^{n-1}(s,\cdot) \right\|_{L^p(\RR)} dy \right)^2 ds \right)^{\frac{p}{2}}\\
&\leq& C_{\varepsilon} \left (\int_0^t \left\|u_{\varepsilon}^{n}(s,\cdot)-u_{\varepsilon}^{n-1}(s,\cdot) \right\|^2_{L^p(\Omega\times \RR)} ds \right)^{\frac{p}{2}}\,.
\end{eqnarray*}
This relation easily entails
\begin{equation*}
 \left\|u_{\varepsilon}^{n+1}(t,\cdot)-u_{\varepsilon}^{n}(t,\cdot) \right\|^2_{L^p(\Omega\times \RR)}\leq C_{\varepsilon}\int_0^t \left\|u_{\varepsilon}^{n}(s,\cdot)-u_{\varepsilon}^{n-1}(s,\cdot) \right\|^2_{L^p(\Omega\times \RR)} ds\,,
\end{equation*}
and a direct application of Gronwall's lemma as stated in Lemma \ref{lemG} yields  that the quantity $\sup_n\sup_{t \in [0,T]}  \|u_{\varepsilon}^{n}(t,\cdot)\|_{L^p(\Omega \times \RR)} $ is finite for each fixed $\varepsilon>0$. This implies that $ \sup_{t \in [0,T]}  \|u_{\varepsilon} (t,\cdot)\|_{L^p(\Omega \times \RR)} <\infty $ for each fixed $\varepsilon>0$.

\noindent
\textit{Step 2}.
Next we estimate $\cn^*_{\frac 12-H, p} u_{\varepsilon} (t)$, and observe that we are able to handle this term   directly (namely without invoking Picard's iterations). We can write
\begin{align*}
&\int_{\RR}\be |u_{\varepsilon}(t,x)-u_{\varepsilon}(t,x+h)|^p dx
\leq C \int_{\RR} |p_tu_0(x)-p_tu_0(x+h)|^p dx\\
&\hspace{1.25in}+ C_{\varepsilon}\int_{\RR}\be \left | \int_0^t \left ( \int_{\RR} |p_{t-s}(y)-p_{t-s}(y+h)| |u_{\varepsilon}(s,y+x)| dy \right)^2 ds \right|^{\frac{p}{2}}dx\\
\leq& C \int_{\RR}|p_tu_0(x)-p_tu_0(x+h)|^p dx \\
&\hspace{1.25in}+ C_{\varepsilon}  \left ( \int_0^t \left ( \int_{\RR} |p_{t-s}(y)-p_{t-s}(y+h)| dy \right) ^2 \|u_{\varepsilon}(s,\cdot)\|^2_{L^p(\Omega\times \RR)} ds \right )^{\frac{p}{2}}\,.
\end{align*}
We thus end up with
\begin{eqnarray*}
\cn^*_{\frac 12-H, p} u_{\varepsilon} (t)&=& \int_{\RR} \frac{\|u_{\varepsilon}(t,\cdot)-u_{\varepsilon}(t,\cdot+h)\|^2_{L^p(\Omega\times \RR)}}{|h|^{2-2H}} dh\\
&\leq& C \int_{\RR} \frac{\|p_tu_0(\cdot)-p_tu_0(\cdot+h)\|^2_{L^p(\RR)}}{|h|^{2-2H}} dh +
  C_{\varepsilon} \, \sup_{s\in [0,T]} \|u_{\varepsilon}(s,\cdot)\|^2_{L^p(\Omega\times \RR)}  \\
 &&\times\int_0^t \int_{\RR} \frac{\left ( \int_{\RR} |p_{t-s}(y)-p_{t-s}(y+h)|dy \right) ^2}{|h|^{2-2H}}dh ds
,
\end{eqnarray*}
and the right-hand side in the above inequality is easily seen to be finite. Putting together the last two steps, we can conclude that for each fixed $\varepsilon$, $u_{\varepsilon} \in \mathcal{Z}_T^{p}$.

\noindent
\textit{Step 3: Uniform bounds in $\ep$.}
To prove the norms of $u_{\varepsilon}$ in $\mathcal{Z}_T^{p}$ are uniformly bounded in $\varepsilon$, we note that $u_{\varepsilon}$ satisfies the equation
\begin{equation*}
u_{\varepsilon}(t,x)=p_tu_0(x)+\int_0^t \int_{\RR} [\left(p_{t-s}(x-\cdot)\sigma(u_{\varepsilon}(s,\cdot))\right)*\rho_{\varepsilon}](y)W(ds,dy)\,.
\end{equation*}
Hence we have
\begin{align}\label{eq:fourier-domination}
 &\be |u_{\varepsilon}(t,x)|^p
  \leq C |p_tu_0(x)|^p + C \be \left ( \int_0^t \int_{\RR} \left |\mathcal{F} \left ( p_{t-s}(x-\cdot)\sigma(u_{\varepsilon}(s,\cdot))\right)(\xi) \right|^2 e^{-\varepsilon |\xi|^2}|\xi|^{1-2H} d\xi ds \right)^{\frac{p}{2}} \notag\\
&\hspace{.3in}\leq C |p_t u_0(x)|^p
+ C \be \left ( \int_0^t \int_{\RR} \left |\mathcal{F} \left ( p_{t-s}(x-\cdot)\sigma(u_{\varepsilon}(s,\cdot))\right)(\xi) \right|^2
|\xi|^{1-2H} d\xi ds \right)^{\frac{p}{2}}.
\end{align}
Going back from Fourier to direct coordinates, one can check that
\begin{equation*}
\be |u_{\varepsilon}(t,x)|^p
\le C\, |p_tu_0(x)|^p
+ \cd_{1}(t) + \cd_{2}(t) \, ,
\end{equation*}
with
\[
\cd_{1}(t)=
\Big (\int_0^t \int_{\RR^2} \left |p_{t-s}(y)-p_{t-s}(y+h)\right|^2
 \left\|u_{\ep}(s,y+x+h) \right\|_{L^p(\Omega)}^2 |h|^{2H-2} dh dy ds\Big)^{\frac{p}{2}}
 \]
 and
 \[
\cd_{2}(t)=
\Big (\int_0^t \int_{\RR^2}  \left |p_{t-s}(y)\right|^2
\left\|u_{\ep}(s,y+x+h)-  u_{\ep}(s,y+x)\right\|_{L^p(\Omega)}^2 |h|^{2H-2} dh dy ds\Big)^{\frac{p}{2}}.
\]
These terms are treated exactly as the terms $D_{1},D_{2}$ in the proof of Lemma \ref{lem2}, except for the fact that $\al=0$ in the current situation. We obtain
\begin{eqnarray}
\nonumber
&&\| u_{\ep}(t,\cdot) \|_{L^p(\Omega \times \RR)} ^{2} \le
C \, \|u_0\|^{2}_{L^p(\RR)}
+ C \,  \int_0^t (t-s)^{H-1} \|u_{\varepsilon}(s,\cdot)\|^2_{L^p(\Omega\times \RR)}  ds  \\  \label{y1}
& & \qquad+ C \,  \int_0^t (t-s)^{-\frac{1}{2}} \int_{\RR}\|u_{\varepsilon}(s,\cdot)-u_{\varepsilon}(s,\cdot+h)\|^2_{L^p(\Omega\times \RR)} |h|^{2H-2} dh ds  \, .
\end{eqnarray}
Similarly we get (see also the bounds for the terms $\ci_{1},\ci_{2}$ in the proof of Theorem \ref{thm:uniqueness})
\begin{eqnarray}  \nonumber
&&[ \cn^*_{\frac 12-H, p} u_{\varepsilon} (t)]^{2}  \le
C \int_{\RR} \|u_0(\cdot)-u_0(\cdot+h)\|^2_{L^p(\RR)} |h|^{2H-2} dh  \\  \nonumber
&&\qquad+ C \int_0^t (t-s)^{2H-\frac{3}{2}}\|u_{\varepsilon}(s,\cdot)\|^2_{L^p(\Omega\times \RR)} ds \\  \label{y2}
&&\qquad +C \int_0^t \int_{\RR} (t-s)^{H-1} \|u_{\varepsilon}(s,\cdot)-u_{\varepsilon}(s,\cdot + l)\|^2_{L^p(\Omega \times \RR)} |l|^{2H-2} dl ds \,.
\end{eqnarray}
Set
\[
\Psi(t)=\| u_{\ep}(t,\cdot) \|_{L^p(\Omega \times \RR)} ^{2}
+[ \cn^*_{\frac 12-H, p} u_{\varepsilon} (t)]^{2}.
\]
Thus combining the  estimates  \eref{y1} and \eref{y2} yields
\[
\Psi(t)
\le C \, \|u_0\|^{2}_{L^p(\RR)} +
C \int_{\RR} \|u_0(\cdot)-u_0(\cdot+h)\|^2_{L^p(\RR)} |h|^{2H-2} dh
+ C \int_0^t (t-s)^{2H-\frac{3}{2}} \Psi(s) ds \,.
\]
Since we have shown that for each fixed $\varepsilon$, $\|u_{\varepsilon}\|_{\mathcal{Z}_T^{p}} < \infty$, we can apply the Gronwall type Lemma \ref{lemG} to the above inequality to show that
\begin{equation*}
\sup_{\varepsilon > 0}\|u_{\varepsilon}\|_{\mathcal{Z}_T^{p}}< \infty\,.
\end{equation*}

\noindent
\textit{Step 4: $u$ is an element of $\mathcal{Z}_T^{p}$.}
Recall once again that we have decomposed $\|u\|_{\mathcal{Z}_T^{p}}$ according to relation \eqref{eq:dcp-norm-ZTp}. We now bound $\| u (t,\cdot) \|_{L^p(\Omega \times \RR)} ^{2}$ and $\cn^*_{\frac 12-H, p} u (t)$ in this decomposition.

Since $u_{\varepsilon}$ converges to $ u$ in $X_T^{\frac{1}{2}-H}$ a.s., we have $u_{\varepsilon}(t,x) \to u(t,x)$ a.s. for each $(t,x)\in\R_{+}\times\R$.  Thus by Fatou's lemma,
\begin{equation*}
\| u (t,\cdot) \|_{L^p(\Omega \times \RR)}
= \left ( \be \int_{\RR} \lim_{\varepsilon \to 0}|u_{\varepsilon}(t,x)|^p dx\right)^{\frac{1}{p}}\\
\leq \varliminf_{\varepsilon \to 0} \left ( \be \int_{\RR} |u_{\varepsilon}(t,x)|^p dx\right)^{\frac{1}{p}} \leq C\,.
\end{equation*}
Therefore we conclude that $\sup_{t \in [0,T]} \|u(t,\cdot)\|_{L^p(\Omega \times \RR)}$ is finite. On the other hand, for each $x$ and $h$ we have $|u_{\varepsilon}(t,x+h)-u_{\varepsilon}(t,x)|^2 \to |u(t,x+h)-u(t,x)|^2$ a.s., so by Fatou's lemma again we obtain
\begin{eqnarray*}
\int_{|h|\leq 1} \frac{\|u(t,\cdot+h)-u(t,\cdot)\|^2_{L^p(\Omega\times \RR)}}{|h|^{2-2H}}dh
&\leq& \int_{|h|\leq 1} \frac{ \varliminf_{\varepsilon \to 0}\|u_\ep(t,\cdot+h)-u_\ep(t,\cdot)\|^2_{L^p(\Omega\times \RR)}}{|h|^{2-2H}}dh\\
&\leq& \varliminf_{\varepsilon \to 0}\int_{|h|\leq 1} \frac{\|u_\ep(t,\cdot+h)-u_\ep(t,\cdot)\|^2_{L^p(\Omega\times \RR)}}{|h|^{2-2H}}dh\,.
\end{eqnarray*}
The desired bound on $\cn^*_{\frac 12-H, p} u (t)$ is obtained from the inequality above, by handling the integral on the domains $|h| \le 1$ and  $|h|>1$. In the latter case, we simply bound $\|u(t,\cdot+h)-u(t,\cdot)\|^2_{L^p(\Omega\times \RR)}$ by $2 \|u(t,\cdot)\|_{L^p(\Omega\times \RR)}$. By doing so, we conclude that
\begin{equation*}
\sup_{t \in [0,T]} \cn^*_{\frac 12-H, p} u (t)
=
\sup_{t \in [0,T]}\int_{\RR}\frac{\|u(t,\cdot+h)-u(t,\cdot)\|^2_{L^p(\Omega \times \RR)}}{|h|^{2-2H}}dh < \infty\,.
\end{equation*}
Together with the previous estimate on $\| u (t,\cdot) \|_{L^p(\Omega \times \RR)}$, we conclude that $u \in \mathcal{Z}_T^{p}$.
\end{proof}	
	
We now state a convergence result for stochastic integrals, with respect to the approximating noise $W_{\ep}$.	

\begin{lemma}\label{int conv lemma}
	Let $u_n(t,x)$ be a solution to the equation
	\begin{equation*}
	u_n(t,x)=p_tu_0(x)+\int_0^t \int_{\RR} p_{t-s}(x-y)\sigma(u_n(s,y))W_n(ds,dy)\,,
	\end{equation*}
	where  we have set $W_n = W_{\ep_{n}}$ (recall that $W_{\ep}$ is defined by \eref{eq:cov-W-epsilon}) for a sequence $\{\ep_{n}, \, n\ge 1\}$ satisfying $\lim_{n\to\infty}\ep_{n}=0$. We assume the following conditions:
	\begin{enumerate} [label=(\roman*)]
	 	\item with probability one, $u_n$ converges to $u$ in $X^{\frac12-H}_T$,
	 	\item $\sup_n\|u_n\|_{\XX^{\beta,p}_T}<\infty$, {with $ \frac 12 -H <\beta<H$ and $p>\frac2H$.}
	 \end{enumerate}
Then the process $u$ belongs to $\XX^{\frac12-H,2}_T$. Furthermore, for any fixed $t\le T$ and $x\in\RR$, the random variable  $\Phi^{n}(t,x)=\int_0^t \int_{\RR} p_{t-s}(x-y)\sigma(u_{n}(s,y))W_{n}(ds,dy)$ converges a.s. to
$\Phi(t,x)=\int_0^t \int_{\RR} p_{t-s}(x-y)\sigma(u(s,y))W(ds,dy)$, as $n \to \infty$.
	\end{lemma}

	\begin{proof}
We focus on the convergence part and decompose the difference $\Phi(t,x) - \Phi^{n}(t,x)$ into $(\Phi(t,x) - \Phi^{n,1}(t,x)) + (\Phi^{n,1}(t,x) - \Phi^{n}(t,x))$, where
\begin{equation*}
\Phi^{n,1}(t,x) = \int_0^t \int_{\RR} p_{t-s}(x-y)\sigma(u(s,y))W^{n}(ds,dy) \, .
\end{equation*}
Now we note that $\Phi(t,x) - \Phi^{n,1}(t,x)$ can be expressed as
\begin{equation*}
\int_0^t \int_{\RR} p_{t-s}(x-y)\sigma(u(s,y))W(ds,dy)
-
\int_0^t \int_{\RR}\left[ \left(p_{t-s}(x-\cdot)\sigma(u(s,\cdot))\right)*\rho_{\varepsilon_n}\right](y)W(ds,dy) ,
\end{equation*}
and thus
\begin{multline*}
\be \lln \Phi(t,x) - \Phi^{n,1}(t,x)  \rrn^{2}  \\
=
C \be \int_0^t \int_{\RR} \left |e^{-\frac{\varepsilon_n |\xi|^2}{2}} -1\right|^2 \left | \mathcal{F}\left (p_{t-s}(x-\cdot) \sigma\left (u(s,\cdot) \right) \right) (\xi)\right|^2 |\xi|^{1-2H} d\xi ds\,.
\end{multline*}
The latter quantity obviously converges to $0$ as $\varepsilon_n$ goes to $0$ because of the finiteness of
	\begin{eqnarray*}
	\be \int_0^t \int_{\RR} \left | \mathcal{F}\left (p_{t-s}(x-\cdot) \sigma\left (u(s,\cdot) \right) \right) (\xi)\right|^2 |\xi|^{1-2H} d\xi ds ,
	\end{eqnarray*}
which can be seen by an application of Fatou's lemma (as in Step 4 of the proof of Lem\-ma~\ref{lemma: u in Z}).
	
	It remains to show that $\lim_{n\to\infty}\be |\Phi^{n,1}(t,x) - \Phi^{n}(t,x)|^{2}=0$. However, similarly to \eqref{eq:fourier-domination}, we have
\begin{equation*}
\be\left[|\Phi^{n,1}(t,x) - \Phi^{n}(t,x)|^{2}\right]
\le
\be  \left|\int_0^t\int_\RR p_{t-s}(x-y)f_n(s,y)W(ds,dy) \right|^2 ,
\end{equation*}
	where  we have set $f_n=\sigma( u_n)- \sigma ( u)$. Furthermore, appealing to Proposition \ref{prop.leftcomp}, we see that $f_n$ converges to 0 in $X^{\frac12-H}_T$. We will verify that $f_n$ satisfies the conditions \ref{cond.c1}-\ref{cond.c3} of Lemma \ref{lem.fnW} below. Indeed, \ref{cond.c1} is verified by assumption (i). \ref{cond.c2} is verified by assumption (ii) and the estimate \eqref{eq:time Holder}. \ref{cond.c3} is readily assumption (ii). Then an application of Lemma~\ref{lem.fnW} completes the proof.
	\end{proof}

	\begin{lemma}\label{lem.fnW} 
		Suppose  that $\{f_n,  \, n\ge 1\}$ is a sequence of stochastic processes  belonging to $\XX^{\beta,p}_T$ with $ \frac 12 -H <\beta<H$ and $p>\frac2H$. Assume that  the following conditions hold:
		\begin{enumerate}  [label= {(C\arabic*)}]
		 	\item\label{cond.c1} With probability one, $f_n$ converges uniformly to $0$ over compact sets of $[0,T]\times \RR$.
		 	\item\label{cond.c2} For every $R>0$, $\sup_n\sup_{s,t\in [0,T], |x| \le R} \be |f_n(t,x)-f_n(s,x)|^p\leq C |t-s|^{p\frac H2}$.
		 	\item\label{cond.c3} $\displaystyle\sup_n\|f_n\|_{\XX^{\beta,p}_T}\le M$, where $M$ is a finite number.
		\end{enumerate}  Then for every $t\le T$ and $x\in\RR$ the random variable $Y_{n}(t,x)$ defined by:
		\begin{equation*}
		Y_{n}(t,x) = \int_0^t\int_\RR p_{t-s}(x-y)f_n(s,y)W(ds,dy)
		\end{equation*}
		converges to 0 in $L^2(\Omega)$.
	\end{lemma}
	\begin{proof} We first observe that Proposition \ref{prop:XXcompare} asserts that $f_n$ belongs to $\XX^{\frac12-H,2}_T$. Next, we show that $\{f_n, \, n\ge 1\}$ is relatively compact and converges to 0 in $\XX^{\frac12-H,2}_T$. For this purpose, we verify the three conditions \eref{cond.xx1}-\eref{cond.xx3} of Proposition~\ref{prop.compbarXX}. Condition \eref{cond.xx2} in Proposition \ref{prop.compbarXX} is evident from \ref{cond.c2}. Condition \eref{cond.xx3} in Proposition \ref{prop.compbarXX}  follows from the  following inequality, where $\delta \le 1$
		\[
		\int_{|y| \le \delta} \frac {\|f(t,x+y) -f(t,x)\|_{ L^2(\Omega)}^2}{|y|^{2-2H}} dy
		\le \sup_{|y| \le 1} \frac {\|f(t,x+y) -f(t,x)\|_{ L^2(\Omega)}^2}{|y|^{2\beta} }
		\int_{|y| \le \delta}  |y|^{ 2\beta +2H-2} dy.
		\]
		In fact,  the first factor on the right side of the above inequality is uniformly bounded in $(t,x) \in [0,T] \times \RR$ because of
		inequality \eqref{ineq.embed} and the fact that $f_n$ is bounded in $\XX^{\beta,2}_T$ by condition \ref{cond.c3}. Taking into account that
		 $\beta>1/2-H$, the second factor converges to zero as $\delta$ tends to zero. To verify condition \eref{cond.xx1} in Proposition \ref{prop.compbarXX}, we fix $t,x$ and note that \ref{cond.c1} implies $f_n(t,x)$ converges almost surely to $0$. On the other hand, $\be |f_n(t,x)|^p$ is uniformly bounded, where $p>2$. These two facts imply $\{f_n(t,x)\}$ converges to $0$ in $L^2(\Omega)$, thus condition \eref{cond.xx1} in Proposition \ref{prop.compbarXX} is verified.
		Furthermore, condition \ref{cond.c1} ensures that $0$ is the only possible limit point of $\{f_n\}$ in $\XX^{1/2-H,2}_T$. We conclude that $f_n$ converges to $0$ in $\XX^{1/2-H,2}_T$.

Let us now prove that $Y_{n}(t,x)$ converges to 0 in $L^2(\Omega)$. Applying \eqref{eq:ineq-burk-2} we   get $\be |Y_{n}(t,x)|^{2} \le C\, (J_{1}(t) + J_{2}(t))$ with
\[
J_1(t)=\int_0^t \int_\RR\int_\RR |p_{t-s}(x-y-z)-p_{t-s}(x-y)|^2 \, \be f_n^2(s,y+z) |z|^{2H-2}dydz ds
\]
and
\[
J_2(t) = \int_0^t\int_\RR\int_\RR |p_{t-s}(x-y)|^2 \, \be |f_n(s,y+z)-f_n(s,y)|^2  |z|^{2H-2}dydzds\,.
\]
Now for every fixed $\varepsilon>0$ we choose  $R>0$ sufficiently large such that 
		\begin{equation*}
			\int_0^t\int_{|y|>R}[|p_{t-s}(y)|^2+ [\cn_{\frac 12-H} p_{t-s}(x-y)]^2 ]dyds<\varepsilon \,.
		\end{equation*}
With this choice of $R$ we choose $n$ so that
		\begin{equation*}
			\sup_{s\in [0,T],\, |y|\le R}\be  f_n^2(s,y)
			+\sup_{s\in [0,T] , \, |y|\le R}\int_\RR
			\, \be |f_n(s,y+z)-f_n(s,y)|^2 |y|^{2H-2}dy  <\varepsilon \,.
		\end{equation*}
		By making a shift in $y$, we end up with
		\begin{align*}
			J_1(t)&= \int_0^t\int_\RR\int_\RR|p_{t-s}(x-y)-p_{t-s}(x-y+z)|^2 \be f_n^2(s,y)|z|^{2H-2}dydzds
			\\&\le  \int_0^t  \sup_{|y|\le R}  \be f_n^2(s,y)  \int_{|y-x|\le R} [\cn_{\frac 12-H} p_{t-s}(x-y)]^2 dy ds \\
&\hspace{2in}	+ \sup_{r\in [0,T] ,\, w\in\RR}    \be f_n^2(r,w)      \int_0^t \int_{|y-x|>R}[\cn_{\frac 12-H} p_{t-s}(x-y)]^2dy ds
			\\&\leq C \varepsilon+CM \int_0^t\int_{|y|>R} [\cn_{\frac 12-H} p_{t-s}(x-y)]^2dy ds\,.
		\end{align*}
		Similarly,
		\begin{align*}
			J_2(t)\leq C\varepsilon+C M \int_0^t\int_{|y|>R}p_{t-s}^2(y)dy ds\,.
		\end{align*}
		Then $\be  |Y_n(t,x)|^2 \leq C \varepsilon$ for $n$ sufficiently large. This implies the result. 
	\end{proof}

Finally, the techniques we have designed  to get existence and uniqueness for equation \eqref{spde with sigma} also allow us to obtain the following moment bound for the solution.
\begin{theorem}
{There are some changes in the formulae for this theorem.}
Assuming the conditions in Theorem \ref{thm:exist with sigma}, then a solution of \eqref{spde with sigma} satisfies following moment bounds
\begin{equation}\label{eq:upp-bnd-Lp-u-sigma-general}
 \sup_{x\in \RR} \|u(t,x)\|_{L^p(\Omega)}  \leq 2 \|u_0\|_{\varepsilon_0}\exp\{\theta_{0} p^{\frac{1}{H}} t\} ,
 \end{equation}
 and
 \begin{equation*}
  \sup_{x\in \RR} \cn_{1/2-H,p} u(t,x)
  \leq2 \|u_0\|_{\varepsilon_0}\varepsilon_0^{-1}  
   \exp\{\theta_{0} p^{\frac{1}{H}} t\}\,,
\end{equation*}
where we recall that $\cn_{1/2-H,p}$ is defined by \eqref{e4}, and where for any $\varepsilon>0$ we have:
\begin{equation*}
\|u_0\|_{\varepsilon}:=\sup_{x\in \RR} |u_0(x)|+\varepsilon \sup_{x\in\RR}\left(\int_{\RR} |u_0(x+h)-u_0(x)|^2 |h|^{2H-2}dh \right)^{\frac{1}{2}}\,,
\end{equation*}
In the formulae above, $C_0$ is as defined in \eref{est.normXX}, and we have $\theta_0=(6C_0)^{\frac{2}{H}} \kappa^{1-\frac{1}{H}}\|\sigma\|_{\rm Lip}^{\frac2H}$, and $\varepsilon_0 = (6C_0)^{1-\frac{1}{2H}} \kappa^{\frac{1}{4H}-\frac{1}{2}}p^{\frac{1}{2}-\frac{1}{4H}}\|\sigma\|_{\rm Lip}^{1-\frac1{2H}}$. 
In addition, from Proposition \ref{prop.embed}, we see that the initial condition $u_0$ is H\"older continuous with order $\beta_0$, then by Proposition \ref{prop:holder.est} we have
\begin{equation}
\|u(t,x)-u(s,y)\|_{L^p(\Omega)} \leq C (|t-s|^{\frac{H}{2}\wedge \frac {\beta_0}{2}}+ |x-y|^{H \wedge \beta_0})
\end{equation}
for all $s, t \in [0,T]$ and $x,y \in \RR$. 
\end{theorem}

\begin{proof}
% We will hinge our considerations on the spaces $\mathfrak{X}^{p}_ {\theta }= \XX^{\frac 12-H, p}_\theta$ defined by \eref{norm.x}. Along the same lines as in the proof of Lemma \ref{lemma: u in Z} we can show that $u \in \mathfrak{X}^{p}_ {\theta }$. 
We will apply Proposition \ref{prop.young} by taking $f$ to be the solution  $u$ to equation \eref{spde with sigma}, and combine it with the mild formulation of the solution. For every fixed $\varepsilon>0$, by noticing that $\|p_tu_0\|_{\mathfrak{X}^p_{T,\theta, \varepsilon}}\leq \|u_0\|_{\varepsilon}$, we get the following bound
\begin{equation*}
\|u\|_{\mathfrak{X}^{p}_ {T,\theta, \varepsilon}}\leq \|u_0\|_{\varepsilon} + C_0\|\sigma\|_{\rm Lip} \sqrt{p} \|u\|_{\mathfrak{X}^{p}_ {T,\theta, \varepsilon}} \left( \kappa^{\frac H2-\frac12}\theta^{-\frac H2} + \varepsilon^{-1}\kappa^{-\frac 14}\theta^{-\frac 14} +\varepsilon \kappa^{H-\frac34}\theta^{\frac 14-H} \right)\,.
\end{equation*}
We optimize the formula above by choosing $\varepsilon = \kappa^{\frac{1}{4}-\frac{H}{2}}\theta^{-\frac{1}{4}+\frac{H}{2}}$, in order to obtain
\begin{equation*}
\|u\|_{\mathfrak{X}^{p}_ {T,\theta, \varepsilon}} \leq \|u_0\|_{\varepsilon} + 3C_0\|\sigma\|_{\rm Lip} \sqrt{p} \|u\|_{\mathfrak{X}^{p}_ {T,\theta, \varepsilon}} \kappa^{\frac{H}{2}-\frac{1}{2}}\theta^{-\frac{H}{2}}\,,
\end{equation*}
then choose $\theta=\theta_0$ so that $3C_0\|\sigma\|_{\rm Lip}\sqrt{p}\kappa^{\frac{H}{2}-\frac{1}{2}}\theta^{-\frac{H}{2}}=\frac{1}{2}$, that is
\begin{equation*}
\theta_0=(6C_0)^{\frac{2}{H}} p^{\frac{1}{H}}\kappa^{1-\frac{1}{H}}\|\sigma\|_{\rm Lip}^{\frac2H}\,, \quad \text {and take}\quad
\varepsilon =\varepsilon_0 := (6C_0)^{1-\frac{1}{2H}} \kappa^{\frac{1}{4H}-\frac{1}{2}}p^{\frac{1}{2}-\frac{1}{4H}}\|\sigma\|_{\rm Lip}^{1-\frac1{2H}}\,.
\end{equation*}
Plugging this choice into the above inequality gives the bound
\begin{equation*}
\|u\|_{\mathfrak{X}^{p}_ {T,\theta_0, \varepsilon_0}}  \leq 2\|u_0\|_{\varepsilon_0} \,.
\end{equation*}
from which our claims are easily deduced by noticing that the constant $C_0$ does not depend on $T$. 
\end{proof}

	We now show the matching lower bound in term of $\kappa$ and $t$ for the second moment.
	\begin{proposition}
		Under the conditions of Theorem \ref{thm:exist with sigma}, let $u$ be a solution to the equation
		\begin{equation}
			u(t,x)=p_tu_0(x)+\int_0^t\int_\RR p_{t-s}(x-y)\sigma(u(s,y))W(ds,dy).
		\end{equation}
		Suppose that $u_0$ is a bounded nontrivial function and there is a positive constant $\sigma_*$ such that $|\sigma(z)|\ge \sigma_* |z|$ for all $z\in\RR$. Then there exist some universal constants $C$ and $L$ such that
		\begin{equation}
			\be |u(t,x)|^2\ge C\frac{|p_tu_0(x)|^3}{\|u_0\|_{L^\infty}} \exp\{L \sigma_*^{\frac2H} \kappa^{1-\frac1H} t\} \,.
		\end{equation}
	\end{proposition}
	\begin{proof}
		Applying It\^o isometry to equation \eref{eq:mild-formulation sigma}, we see that
		\begin{equation}
			\be |u(t,x)|^2=|p_tu_0(x)|^2+c_{1,H}\be\int_0^t\|p_{t-s}(x-y)\sigma(u(s,y))\|^2_{\dot{H}^{\frac12-H}}ds\,.
		\end{equation}
		Let us recall the well-known Sobolev embedding inequality
		\begin{equation*}
			\|g\|_{\dot{H}^{\frac12-H}}\ge c \|g\|_{L^{\frac1H}}\,,\quad\forall g\in\dot{H}^{\frac12-H}(\RR)\,.
		\end{equation*}
		Hence, together with our assumption on $\sigma$, it follows that there exists some positive constant $b$ such that
		\begin{equation*}
			\be |u(t,x)|^2\ge |p_tu_0(x)|^2 + b \sigma_*^2\be\int_0^t \|p_{t-s}(x-\cdot)u(s,\cdot)\|^2_{L^{\frac1H} (\RR)}ds\,.
		\end{equation*}
		Since $2H<1$, applying Jensen inequality we see that
		\begin{align*}
			\|p_{t-s}(x-\cdot)u(s,\cdot)\|^2_{L^{\frac1H}(\RR)}&=\lp\int_\RR p_{t-s}^{\frac1H-1}(x-y)|u(s,y)|^{\frac1H}p_{t-s}(x-y) dy\rp^{2H}
			\\&\ge\int_\RR p_{t-s}^{3-2H}(x-y)|u(s,y)|^{2}dy\,.
		\end{align*}
		It follows that
		\begin{equation*}
			\be |u(t,x)|^2\ge |p_tu_0(x)|^2+b\sigma_*^2\int_0^t\int_\RR p^{3-2H}_{t-s}(x-y)\be|u(s,y)|^{2}dyds\,.
		\end{equation*}
		Iterating the previous inequality yields
		\begin{equation}\label{tmp.lowuI}
			\be |u(t,x)|^2\ge |p_tu_0(x)|^2+\sum_{n=1}^\infty (b\sigma_*^2)^nI_n(t,x)\,.
		\end{equation}
		In the above, we have adopted the notation
		\begin{equation*}
			I_n(t,x)=\int_{T_n(t)}\int_{\RR^n}p_{t-s_n}^{3-2H}(x-y_n)\cdots p_{s_2-s_1}^{3-2H}(y_2-y_1)|p_{s_1}u_0(y_1)|^2d\bar yd\bar s
		\end{equation*}
		where $T_n(t)=\{(s_1,\dots,s_n)\in[0,t]^n:0<s_1<\cdots<s_n<t\}$ and $d\bar y=dy_1\cdots dy_n$, $d\bar s=ds_1\cdots ds_n$. Note that for every $x,z\in\RR$ and $a,b>0$, the following identity holds
		\begin{equation*}
			\int_\RR p_{a}^{3-2H}(x-y)p_b^{3-2H}(y-z)dy=(3-2H)^{-\frac12}\lp\frac{2 \pi \kappa ab}{a+b}\rp^{H-1} p_{a+b}^{3-2H}(x-z)\,.
		\end{equation*}
		We thus can compute $I_n(t,x)$ by integrating $y_j$'s in descending order starting from $y_n$. This procedure yields
		\begin{multline}\label{tmp.lowIn}
			I_n(t,x)=(3-2H)^{-\frac{n-1}2} \times
			\\\int_{T_n(t)}\lp\frac{t-s_n}{t-s_1}\prod_{j=2}^{n}2 \pi \kappa(s_j-s_{j-1})\rp^{H-1} \int_\RR p_{t-s_1}^{3-2H}(x-y_1)|p_{s_1}u_0(y_1)|^2dy_1 d\bar s\,.
		\end{multline}
		On the other hand, for every fixed $R>0$, applying Jensen inequality, we see that
		\begin{align}
			\int_\RR p_{t-s_1}^{3-2H}(x-y_1)&|p_{s_1}u_0(y_1)|^2dy_1
			\ge p_{t-s_1}^{1-2H}(R) \int_{|x-y_1|<R} p_{t-s_1}^{2}(x-y_1)|p_{s_1}u_0(y_1)|^2 dy_1
			\nonumber\\&\ge p_{t-s_1}^{1-2H}(R) R^{-1} \lp\int_{|x-y_1|<R} p_{t-s_1}(x-y_1)p_{s_1}*u_0(y_1)dy_1\rp^{2}\,.
			\label{tmp.low1}
		\end{align}
		The integral on the right side can be rewritten as
		\begin{equation*}
			p_tu_0(x)-\int_{|x-y_1|\ge R}p_{t-s_1}(x-y_1)p_{s_1}*u_0(y_1)dy_1\,.
		\end{equation*}
		Since $u_0$ is bounded, we see that $|p_{s_1}*u_0(y_1)|\le \|u_0\|_{L^\infty}$ and hence
		\begin{align*}
			|\int_{|x-y_1|\ge R}p_{t-s_1}(x-y_1)p_{s_1}*u_0(y_1)dy_1|
			&\le \|u_0\|_{L^\infty} \int_{|y|>R}p_{t-s_1}(y)dy
			\\&=\|u_0\|_{L^\infty} \pi^{-\frac12}  \int_{|z|>\frac{R}{\sqrt{2\kappa(t-s_1)}} }e^{-z^2}dz\,.
		\end{align*}
		For every fixed $\epsilon$ in $(0,1)$, we now choose $R=M\sqrt{2 \kappa (t-s_1)}$ where $M$ is such that $e^{-(1-2H)M^2}M^{-1}= \epsilon$. It follows that
		\begin{equation*}
			p_{t-s_1}^{1-2H}(R)R^{-1}=\pi^{H-\frac12}(2 \kappa(t-s_1))^{H-1}e^{-(1-2H)M^2}M^{-1}
		\end{equation*}
		and
		\begin{align*}
			|\int_{|x-y_1|<R}p_{t-s_1}(x-y_1)p_{s_1}*u_0(y_1)dy_1|\ge |p_tu_0 (x)|-\|u\|_{\infty}e^{-M^2}M^{-1}\,.
		\end{align*}
		Together with \eqref{tmp.low1}, we see that
		\begin{equation*}
			\int_{\RR}p_{t-s_1}^{3-2H}(x-y_1)|p_{s_1}u_0(y_1)|^2dy_1\ge c e^{-M^2}M^{-1} (\kappa(t-s_1))^{H-1}\lp |p_tu_0(x)|-e^{-M^2}M^{-1}\|u_0\|_{L^\infty}\rp^2
		\end{equation*}
		for some universal constant $c$.
		Hence, upon combining the previous estimate and \eqref{tmp.lowIn}, we arrive at
		\begin{align*}
			I_n(t,x)\ge \epsilon c^{n} \kappa^{(H-1)n} \int_{T_n(t)}\prod_{j=2}^{n+1}(s_j-s_{j-1})^{H-1} d\bar s\lp|p_tu_0(x)|-\epsilon\|u_0\|_{L^\infty}\rp^2
		\end{align*}
		where $s_{n+1}=t$ and $c$ is some universal constant. It is elementary (see Lemma \ref{lem:intg-simplex} below) to compute
		\begin{equation*}
			\int_{T_n(t)}\prod_{j=2}^{n+1}(s_j-s_{j-1})^{H-1} d\bar s=\frac{\Gamma(H)^nt^{nH}}{\Gamma(nH+1)}\,.
		\end{equation*}
		Therefore, together with \eqref{tmp.lowuI}, we obtain
		\begin{align*}
			\EE |u(t,x)|^2\ge \epsilon\lp|p_tu_0(x)|-\epsilon\|u_0\|_{L^\infty}\rp^2\sum_{n=0}^\infty (cb \Gamma(H))^n\frac{ (\sigma_*^{\frac2H}\kappa^{1-\frac1H} t)^{nH}}{\Gamma(nH+1)}.
		\end{align*}
{We now recall the elementary bound $\sum_{n\ge 0}x^{n}/(n!)^{a}\le 2 \exp(c x^{1/a})$, which can be found e.g in \cite[Lemma A.1]{BC}. Together with the previous inequality, this yields:}
		\begin{equation}\label{est.lowu2}
			\be |u(t,x)|^2\ge C \epsilon\lp p_tu_0(x)-\epsilon \|u_0\|_{L^\infty}\rp^2  e^{L \sigma_*^{\frac2H} \kappa^{1-\frac1H} t}\,.
		\end{equation}
		By choosing $\epsilon=\frac{|p_tu_0(x)|}{3\|u_0\|_{L^\infty}}$, we conclude the proof.
	\end{proof}
\begin{remark}
(i) We can add a drift $b(u(t,x))$ in equation \eqref{spde with sigma}, and if the function $b$ is Lipschitz continuous with $b(0)=0$, the results we have obtained on the existence and uniqueness of a solution can be extended to equations with drift.

\noindent
(ii) If we only assume that the initial condition $u_0$ is bounded and 
\begin{equation}
\sup_{x\in\RR}\int_{\RR}|u_0(x)-u_0(x+h)|^2 |h|^{2H-2}dh < \infty\,,
\end{equation}
and we only assume that $\sigma$ is Lipschitz. Then from the proof of Theorem \ref{thm:exist with sigma} it is easy to see that we have the weak existence of a solution to equation \eref{spde with sigma}. The assumption (1) in Theorem \ref{thm:exist with sigma} and the condition that the derivative of $\sigma$ is Lipschitz and $\sigma(0)=0$ are only used to show the uniqueness. 
 \end{remark}

\section{The Anderson model}
In this section we will study the special case of equation
\eref{spde with sigma} when the function $\sigma$ is the identity function.
This is  a continuous version of the so-called  parabolic Anderson
model. In this case equation \eref{spde with sigma} is reduced to
\begin{equation}\label{spde}
\frac{\partial u}{\partial t}=\frac{\kappa}{2}\frac{\partial ^2 u}{\partial x
^2}+u\,\dot W
\end{equation}
with deterministic initial condition $u(0,x)=u_0(x)$. This reduced form allows for some simplified versions of the existence-uniqueness theorems, and also some Feynman-Kac representation which is useful for intermittency estimates.

\subsection{Existence and uniqueness}\label{sec:anderson-exist-uniq}
 With some restrictions on the initial condition $u_0(x)$, the existence
and uniqueness of the solution to this linear equation stems
directly from Theorems \ref{thm:uniqueness} and \ref{thm:exist with sigma}. However, we shall
prove this result again by means of two different methods: one is
via Fourier transform and the other is via chaos expansion. We
include these methods here for two reasons: first,  they lead to
proofs which are shorter and more elegant than in the case of a
general coefficient $\si$;  {secondly, the assumptions on initial
conditions are different. }

\subsubsection{Existence and uniqueness via Fourier transform}\label{sec:picard}

In this subsection we discuss the existence and uniqueness of equation (\ref{spde}) using techniques of Fourier analysis.

We recall that
$\dot{H}^{\frac 12-H}$ is the class of functions  $f:    \RR \rightarrow \RR $ such that  there exists $  g\in L^2(   \RR)  $ such that $ f =I_-^{1/2-H}g$. Let    $\dot{H}^{\frac 12-H}_0$  be the set of functions $f\in L^2(\RR)$ such that $\int_\RR | \cf f(\xi)| ^2 |\xi|^{1-2H} d\xi <\infty$.
   These spaces are the time independent analogues to the spaces $\HH$ and $\HH_0$ introduced in Proposition
\ref{prop: H}.   We know that the inclusion  $\dot{H}^{\frac 12-H}_0 \subset  \dot{H}^{\frac 12-H}$ is strict and
$\dot{H}^{\frac 12-H}_0$ is not complete with the seminorm  $ \left[ \int_\RR | \cf f(\xi)| ^2 |\xi|^{1-2H} d\xi \right] ^\frac 12$
(see \cite{PT}). However,  it is not difficult to check that the space   $\dot{H}^{\frac 12-H}_0$ is complete for the seminorm
\[
\|f\|_{\cv(H)} ^2:=\int_\RR | \cf f(\xi)| ^2  (1+|\xi|^{1-2H} )d\xi.
\]

 In the  next theorem  we show the existence and uniqueness result  assuming that
 the initial condition belongs to $\dot{H}^{\frac 12-H}_0$ and
 using estimates based on the Fourier transform in the space variable. To this purpose, we   introduce
 the space $\cv_T(H)$  as the completion of the set  of elementary   $\dot{H}^{\frac 12-H}_0$-valued stochastic processes $\{u(t,\cdot), t\in [0,T]\}$ with respect to the seminorm
\begin{equation}  \label{nuH}
\|u\|_{\cv_{T}(H)}^{2}:=\sup_{t\in [0,T]}   \be \| u(t,\cdot)\|_{\cv(H)}^{2}.
\end{equation}

We now state a convolution lemma.

\begin{proposition}\label{prop:convolution-fourier}
Consider a function $u_{0}\in \dot{H}^{\frac 12-H}_0$ and
$\frac{1}{4}<H<\frac{1}{2}$. For any  {$v\in\cv_{T}(H)$} we set
$\gga(v)=V$ in the following way:
\begin{equation*}
\gga(v):=V(t,x)=p_t u_0(x) + \int_0^t \int_{\mathbb{R}}p_{t-s}(x-y) v(s,y) W(ds,dy),
\quad t\in[0,T], \, x\in\R.
\end{equation*}
Then $\gga$ is well-defined as a map from $\cv_{T}(H)$ to $\cv_{T}(H)$. Furthermore, there exist two positive constants $c_{1},c_{2}$ such that the following estimate holds true on $[0,T]$:
\begin{equation}\label{eq:intg-bnd-V-lin}
{ \|V(t,\cdot)\|_{\cv(H)}^{2} \le c_{1} \, \|u_0\|_{\cv(H)}^{2}
+c_{2}\int_0^t  (t-s)^{2H-3/2}
 \|v(s,\cdot)\|_{\cv(H)}^{2} \,
ds\,.}
\end{equation}
\end{proposition}

\begin{proof}
Let $v$ be a process in $\cv_{T}(H)$ and set $V=\gga(v)$.  We
focus on the bound \eqref{eq:intg-bnd-V-lin} for $V$. 

Notice that the Fourier transform of $V$ can be computed easily.
 {Indeed, setting $v_0(t,x)=p_tu_0(x)$ and }invoking a stochastic
version of Fubini's theorem, we get
\begin{equation*}
\mathcal{F}V(t,\xi)=\mathcal{F}v_0(t,\xi)
+\int_0^t\int_{\mathbb{R}} \lp \int_{\R} e^{i x \xi} \, p_{t-s}(x-y) \, dx \rp
v(s,y)W(ds,dy)\,.
\end{equation*}
According to the expression of $\cf p_{t}$, we obtain
\begin{eqnarray*}
\mathcal{F}V(t,\xi)=\mathcal{F}v_0(t,\xi)+\int_0^t\int_{\mathbb{R}}e^{-i\xi
y} e^{-\frac{\kappa}{2}(t-s)\xi^2}v(s,y)W(ds,dy)\,.
\end{eqnarray*}
We now evaluate the quantity
$\be[\int_{\mathbb{R}}|\mathcal{F}V(t,\xi)|^2|\xi|^{1-2H}d\xi ]$ in
the definition of $\|u_{n}\|_{\cv_{T}(H)}$ given by \eqref{nuH}.  We
thus write
\begin{multline*}
\be\lc \int_{\mathbb{R}}|\mathcal{F}V(t,\xi)|^2|\xi|^{1-2H}d\xi \rc
\leq 2 \, \int_{\mathbb{R}}|\mathcal{F}v_0(t,\xi)|^2|\xi|^{1-2H}d\xi \\
+2 \,  \int_{\mathbb{R}}\be\lc\Big|\int_0^t\int_{\mathbb{R}}e^{-i\xi
y}e^{-\frac{\kappa}{2}(t-s)\xi^2}v(s,y)W(ds,dy)\Big|^2 \rc |\xi|^{1-2H}d\xi
:= 2\lp I_{1} + I_{2} \rp \, ,
\end{multline*}
and we handle the terms $I_{1}$ and $I_{2}$ separately.

The term $I_1$ can be easily bounded by using that $u_0 \in\dot{H}^{\frac 12-H}_0$ and recalling $v_{0}=p_{t}u_{0}$. That is,
\[
I_1 = \int_\RR| \mathcal{F}u_0(\xi) | ^2e^{-\kappa t|\xi|^2} |\xi|^{1-2H}
d\xi \le C \, \|u_{0}\|_{\cv(H)}^{2}.
\]
 We thus focus on the estimation of $I_{2}$, and we set $f_{\xi}(s,\eta)=e^{-i\xi
\eta}e^{-\frac{\kappa}{2}(t-s)\xi^2}v(s,\eta)$. Applying the isometry
property \eqref{int isometry} together with the Fourier transform
expression for $\|h\|_{\dot H^{\frac 12-H}}$ in~\eref{eq: H_0 element H prod}, we
have:
\begin{equation*}
\be\lc\Big|\int_0^t\int_{\mathbb{R}}
e^{-i\xi y}e^{-\frac{\kappa}{2}(t-s)\xi^2}v(s,y)W(ds,dy)\Big|^2 \rc
=c_{1,H} \int_0^t \int_{\mathbb{R}}
\be\lc |\cf_{\eta}f_{\xi}(s,\eta) |^{2}\rc
|\eta|^{1-2H} \, ds d\eta,
\end{equation*}
where $\cf_{\eta}$ is  the Fourier transform with respect  to
$\eta$.  It is obvious %from the definition of Fourier transform
 that
the Fourier transform of $e^{-i\xi y} V(y)$ is $\mathcal{F}
V(\eta+\xi)$. Thus we have
\begin{align*}
I_{2}&= C\int_0^t\int_{\mathbb{R}}\int_{\mathbb{R}}e^{-\kappa(t-s)\xi^2}
\, \be\lc|\mathcal{F}v(s,\eta+\xi)|^2 \rc |\eta|^{1-2H}|\xi|^{1-2H}
\, d\eta d\xi ds\\
&= C\int_0^t\int_{\mathbb{R}}\int_{\mathbb{R}}e^{-\kappa(t-s)\xi^2} \,
\be\lc|\mathcal{F}v(s,\eta )|^2 \rc |\eta-\xi|^{1-2H}|\xi|^{1-2H} \,
d\eta d\xi ds\, .
\end{align*}
We now bound $|\eta-\xi |^{1-2H}$ by $|\eta|^{1-2H}+|\xi|^{1-2H}$,
which yields $I_{2}\le I_{21}+I_{22}$ with:
\begin{align*}
I_{21}&=C
\int_0^t
\int_{\mathbb{R}}\int_{\mathbb{R}} e^{-\kappa(t-s)\xi^2} \,
\be\lc|\mathcal{F}v(s,\eta)|^2 \rc
|\eta|^{1-2H}|\xi|^{1-2H} \, d\eta d\xi ds \\
I_{22}&=C\int_0^t\int_{\mathbb{R}}\int_{\mathbb{R}}e^{-\kappa(t-s)\xi^2}
\, \be\lc|\mathcal{F}v(s,\eta)|^2 \rc |\xi|^{2-4H} \, d\eta d\xi
ds\,.
\end{align*}
Performing the change of variable  {$\xi \rightarrow
(t-s)^{1/2}\xi$} and then trivially bounding the integrals of the
form $\int_{\R}|\xi|^{\beta} e^{-\kappa\xi^{2}} d\xi$ by constants,  we
end up with
\begin{align*}
I_{21}&\leq C
\int_0^t  (t-s)^{H-1}
\int_{\mathbb{R}}
\be\lc|\mathcal{F}v(s,\eta)|^2 \rc
|\eta|^{1-2H} \, d\eta  \, ds \\
I_{22}&\leq C
\int_0^t  (t-s)^{2H-3/2}
\int_{\mathbb{R}}
\be\lc|\mathcal{F}v(s,\eta)|^2 \rc
 \, d\eta  \, ds .
\end{align*}
Observe that for $H\in(\frac 14, \frac 12)$ the term $(t-s)^{2H-3/2}$ is more
singular than  $(t-s)^{H-1}$, but we still have $2H-\frac 32>-1$.
Summarizing our consideration  up to now, we have thus obtained
\begin{multline}\label{eq:bnd-picard-1}
\int_{\mathbb{R}}\be\lc |\mathcal{F}V(t,\xi)|^2 \rc |\xi|^{1-2H}d\xi \\
\le C _{1,T} \, { \|u_{0}\|_{\cv(H)}^{2}} + C_{2,T} \int_{0}^{t}
(t-s)^{2H-3/2} \int_{\mathbb{R}} \be\lc|\mathcal{F}v(s,\xi)|^2 \rc
\, (1+ |\xi|^{1-2H})
 \, d\xi  \, ds ,
\end{multline}
for two strictly positive constants $C_{1,T},C_{2,T}$.

The term $\be[\int_{\mathbb{R}}|\mathcal{F}V(t,\xi)|^2 d\xi ]$ in
the definition of $\|V\|_{\cv_{T}(H)}$ can be bounded with the same computations as above, and we find
\begin{multline}\label{eq:bnd-picard-2}
\int_{\mathbb{R}}\be\lc |\mathcal{F}V(t,\xi)|^2 \rc \, d\xi \\
\le C_{1,T} \, {  \|u_{0}\|_{\cv(H)}^{2}} + C_{2,T}  \int_{0}^{t}
(t-s)^{H-1} \int_{\mathbb{R}} \be\lc|\mathcal{F}v(s,\xi)|^2 \rc \,
(1+ |\xi|^{1-2H})
 \, d\eta  \, ds ,
\end{multline}
Hence, gathering our estimates \eqref{eq:bnd-picard-1} and \eqref{eq:bnd-picard-2}, our bound \eqref{eq:intg-bnd-V-lin} is easily obtained, which finishes the proof.
\end{proof}

As in the general case, Proposition \ref{prop:convolution-fourier} is the key to the existence and uniqueness result for equation \eqref{spde}.

\begin{theorem}\label{thm:exist-uniq-picard}
Suppose that $u_{0}$ is an element of $\dot{H}^{\frac 12-H}_0$ and
$\frac{1}{4}<H<\frac{1}{2}$. Fix $T>0$. Then there is a  unique
process $u$ in the space $\cv_{T}(H)$ such that for all $t\in
[0,T]$,
\begin{equation}
u(t,\cdot)=p_t u_0  + \int_0^t \int_{\mathbb{R}}p_{t-s}(\cdot-y)u(s,y) W(ds,dy).
\end{equation}
\end{theorem}

\begin{proof}
The proof follows from the standard Picard iteration scheme, where we just set $u_{n+1}=\gga(u_{n})$. Details are left to the reader for sake of conciseness.
\end{proof}

\subsubsection{Existence and uniqueness via chaos expansions}

Next, we provide another way to prove the existence and uniqueness of the solution to equation \eref{spde}, by means of chaos expansions. This will enable us to  obtain moment estimates.
 Before stating our main theorem in this direction, let us label an elementary lemma borrowed from \cite{HHNT} for further use. 

\begin{lemma}\label{lem:intg-simplex}
For $m\ge 1$ let $\alpha \in (-1+\ep, 1)^m$  with $\ep>0$ and  set $|\alpha |= \sum_{i=1}^m
\alpha_i  $. For $t\in\ott$, the $m$-th  dimensional simplex over $\ot$ is denoted by
$T_m(t)=\{(r_1,r_2,\dots,r_m) \in \R^m: 0<r_1  <\cdots < r_m < t\}$.
Then there is a constant $c>0$ such that
\[
J_m(t, \alpha):=\int_{T_m(t)}\prod_{i=1}^m (r_i-r_{i-1})^{\alpha_i}
dr \le \frac { c^m t^{|\alpha|+m } }{ \Gamma(|\alpha|+m +1)},
\]
where by convention, $r_0 =0$.
\end{lemma}

Let us now state a new existence and uniqueness theorem for our equation of interest.

\begin{theorem}\label{thm:exist-uniq-chaos}
Suppose that $\frac 14 <H<\frac 12$ and that the initial condition $u_0$ satisfies
\begin{equation}\label{cond:fu0}
\int_{\RR}(1+|\xi|^{\frac{1}{2}-H})|\mathcal{F}u_0(\xi)|d\xi < \infty\,.
\end{equation}
Then there exists a unique    solution to equation \eqref{spde},
that is a process $u\in \laa_H$ (remember that $\laa_H$ is defined in Proposition \ref{prop:intg-wrt-W}) such that for any
$(t,x)\in\ott\times\R$, relation  \eqref{eq:mild-formulation sigma}
holds true.
\end{theorem}

\begin{remark}
(i) The formulation of Theorem \ref{thm:exist-uniq-chaos} yields the definition of our solution $u$ for all $(t,x)\in\ott\times\R$. This is in contrast with Theorem \ref{thm:exist-uniq-picard} which gives a solution sitting in $\dot{H}^{\frac12-H}_0$ for every value of $t$, and thus defined a.e. in $x$ only.

(ii) Condition \eqref{cond:fu0} is satisfied by constant functions.
\end{remark}

\begin{proof}[Proof of Theorem \ref{thm:exist-uniq-chaos}]

Suppose that $u=\{u(t,x), \, t\geq 0, x \in \R^d\}$ is a solution to equation~\eqref{eq:mild-formulation sigma} in $\laa_{H} $. Then according to \eref{eq:chaos-dcp}, for any fixed $(t,x)$ the random variable $u(t,x)$ admits the following Wiener chaos expansion
\begin{equation}\label{eq:chaos-expansion-u(tx)}
u(t,x)=\sum_{n=0}^{\infty}I_n(f_n(\cdot,t,x))\,,
\end{equation}
where for each $(t,x)$, $f_n(\cdot,t,x)$ is a symmetric element in
$\HH^{\otimes n}$. Furthermore, we have seen that It\^o and Skorohod's integral coincide for processes in $\laa_{H}$. Hence, thanks to~\eqref{eq:delta-u-chaos} and using an iteration procedure, one can find an
explicit formula for the kernels $f_n$ for $n \geq 1$. Indeed, we have:
\begin{multline}\label{eq:expression-fn}
f_n(s_1,x_1,\dots,s_n,x_n,t,x)\\
=\frac{1}{n!}p_{t-s_{\si(n)}}(x-x_{\si(n)})\cdots p_{s_{\si(2)}-s_{\si(1)}}(x_{\si(2)}-x_{\si(1)})
p_{s_{\si(1)}}u_0(x_{\si(1)})\,,
\end{multline}
where $\si$ denotes the permutation of $\{1,2,\dots,n\}$ such that $0<s_{\si(1)}<\cdots<s_{\si(n)}<t$
(see, for instance,  formula (4.4) in \cite{HN} or  formula (3.3) in \cite{HHNT}).
Then, to show the existence and uniqueness of the solution it suffices to prove that for all $(t,x)$ we have
\begin{equation}\label{chaos}
\sum_{n=0}^{\infty}n!\|f_n(\cdot,t,x)\|^2_{\HH^{\otimes n}}< \infty\,.
\end{equation}
The remainder of the proof is devoted to prove  relation \eqref{chaos}.

Starting from relation \eqref{eq:expression-fn}, some elementary Fourier computations show that
\begin{eqnarray*}
\cf f_n(s_1,\xi_1,\dots,s_n,\xi_n,t,x)&=&
\frac{c_{H}^n}{n!}  \int_\RR
\prod_{i=1}^n e^{-\frac{\kappa}{2}(s_{\si(i+1)}-s_{\si(i)})|\xi_{\si(i)}+\cdots +
\xi_{\si(1)} -\zeta|^2} \\
&&\times { e^{-ix (\xi_{\sigma(n)}+ \cdots + \xi_{\sigma(1)}-\zeta)}} \mathcal{F}u_0(\zeta) e^{-\frac {\ka s_{\sigma(1)}|\zeta|^2} 2} d\zeta,
\end{eqnarray*}
where we have set $s_{\si(n+1)}=t$.
Hence, owing to formula \eref{eq: H_0 element H prod} for the norm in $\HH$ (in its Fourier mode version), we have
\begin{multline}\label{eq:expression-norm-fn}
n!\| f_n(\cdot,t,x)\|_{\HH^{\otimes n}}^2 =\frac{c_H^{2n}
}{n!} \, \int_{[0,t]^n}\int_{\RR^n}\bigg|   \int_\RR \prod_{i=1}^n
e^{-\frac {\kappa}{2} (s_{\si(i+1)}-s_{\si(i)})|\xi_i+\cdots +\xi_1-\zeta |^2} { e^{-ix (\xi_{\sigma(n)}+ \cdots + \xi_{\sigma(1)}-\zeta)}} \\ 
\mathcal{F}u_0(\zeta) e^{-\frac {\kappa s_{\sigma(1)}|\zeta|^2} 2} d\zeta \bigg|^2 
 \times  \prod_{i=1}^n  |\xi_i |^{1-2H} d\xi ds\,,
\end{multline}
where $d\xi$ denotes $d\xi_1 \cdots d\xi_n$ and similarly for $ds$.
Then using the change of variable $\xi_{i}+\cdots + \xi_{1}=\eta
_{i}$, for all $i=1,2,\dots, n$ and a linearization of the above expression, we obtain
\begin{multline*}
n!\| f_n(\cdot,t,x)\|_{\HH^{\otimes n}}^2 = \frac{c_H^{2n}
}{n!}\int_{[0,t]^n}\int_{\RR^n}
 \int_{\RR^2}\prod_{i=1}^n
e^{-\frac{\kappa}{2}(s_{\si(i+1)}-s_{\si(i)})(|\eta_{i}-\zeta|^2+|\eta_i-\zeta^{\prime}|^2)} \mathcal{F}u_0(\zeta) \overline{\mathcal{F}u_0(\zeta^{\prime})} \\
\times { e^{ix(\zeta -\zeta')}}e^{-\frac{\kappa s_{\sigma(1)}(|\zeta|^2+|\zeta^{\prime}|^2)}{2}} \prod_{i=1}^n|\eta_{i}-\eta_{i-1}|^{1-2H} d\zeta d\zeta^{\prime} d\eta ds\,,
\end{multline*}
where we have set $\eta_{0}=0$. Then we  use Cauchy-Schwarz inequality and bound the term $\exp(-\kappa s_{\sigma(1)}(|\zeta|^2+|\zeta^{\prime}|^2)/2)$
by $1$ to get
\begin{multline*}
n!\| f_n(\cdot,t,x)\|_{\HH^{\otimes n}}^2 \le
\frac{c_H^{2n}}{n!}
 \int_{\RR^2} \left ( \int_{[0,t]^n} \int_{\RR^n} \prod_{i=1}^n
e^{- \kappa (s_{\si(i+1)}-s_{\si(i)})|\eta_{i}-\zeta|^2}\prod_{i=1}^n|\eta_{i}-\eta_{i-1}|^{1-2H}d\eta ds \right)^{\frac{1}{2}} \\
\times \left ( \int_{[0,t]^n} \int_{\RR^n} \prod_{i=1}^n
e^{- \kappa (s_{\si(i+1)}-s_{\si(i)})|\eta_{i}-\zeta^{\prime}|^2}\prod_{i=1}^n|\eta_{i}-\eta_{i-1}|^{1-2H}d\eta ds \right)^{\frac{1}{2}}
\left|\mathcal{F}u_0(\zeta)\right| \left|\mathcal{F}u_0(\zeta^{\prime})\right| d\zeta d\zeta^{\prime}.
\end{multline*}
Arranging the integrals again, performing the change of variables $\eta_{i}:=\eta_{i}-\zeta$ and invoking the trivial bound $|\eta_{i}-\eta_{i-1}|^{1-2H}\le |\eta_{i-1}|^{1-2H}+|\eta_{i}|^{1-2H}$, this yields
\begin{equation}\label{eq:bnd-fn-L2-1}
n!\| f_n(\cdot,t,x)\|_{\HH^{\otimes n}}^2 \le \frac{c_H^{2n}
}{n!} \Bigg(\int_{\RR} L_{n,t}^{\frac{1}{2}}(\zeta) \left
|\mathcal{F}u_0(\zeta)\right|d\zeta\Bigg)^2 ,
\end{equation}
where
\begin{multline*}
L_{n,t}(\zeta) \\
=
\int_{[0,t]^n} \int_{\RR^n} \prod_{i=1}^n
e^{-\kappa (s_{\si(i+1)}-s_{\si(i)})|\eta_{i}|^2} (|\zeta|^{1-2H}+|\eta_1|^{1-2H})
\times \prod_{i=2}^n(|\eta_{i}|^{1-2H}+|\eta_{i-1}|^{1-2H})d\eta ds.
\end{multline*}
Let us expand the product $\prod_{i=2}^{n} (|\eta_{i}|^{1-2H}+|\eta_{i-1}|^{1-2H})$ in the integral defining $L_{n,t}(\zeta)$. We obtain an expression of the form $\sum_{\al\in D_{n}} \prod_{i=1}^{n} |\eta_{i}|^{\al_{i}}$, where $D_{n}$ is a subset of multi-indices of length $n-1$. The complete description of $D_{n}$ is omitted for sake of conciseness, and we will just use the following facts: $\text{Card}(D_{n})=2^{n-1}$ and for any $\al\in D_{n}$ we have
\begin{equation*}
|\al|\equiv \sum_{i=1}^{n} \alpha_i = (n-1)(1-2H),
\quad\text{and}\quad
\al_{i} \in \{0, 1-2H, 2(1-2H)\}, \quad i=1,\ldots, n.
\end{equation*}
This simple expansion yields the following bound
\begin{multline*}
L_{n,t}(\zeta)
\leq|\zeta|^{1-2H}\sum_{\alpha \in D_{n}} \int_{[0,t]^n} \int_{\RR^n} \prod_{i=1}^n
e^{-\kappa (s_{\si(i+1)}-s_{\si(i)})|\eta_{i}|^2} \prod_{i=1}^n |\eta_i|^{\alpha_i}d\eta ds\\
+\sum_{\alpha \in D_{n}} \int_{[0,t]^n}\int_{\RR^n}\prod_{i=1}^n e^{-\kappa (s_{\si(i+1)}-s_{\si(i)})|\eta_{i}|^2} |\eta_1|^{1-2H} \prod_{i=1}^n |\eta_i|^{\alpha_i}d\eta ds\,.
\end{multline*}
Perform the change of variable $\xi_{i}= (\kappa (s_{\sigma(i+1)}-s_{\sigma(i)}))^{1/2} \eta_{i}$ in the above integral, and notice that $\int_{\R} e^{- \xi^{2}} |\xi|^{\alpha_i}d\xi$ is bounded by a constant
for $\alpha_i>-1$. Changing the integral over $\ot^{n}$ into an integral over the simplex, we get
\begin{eqnarray*}
L_{n,t}(\zeta)&\leq& C |\zeta|^{1-2H} n! c_H^n \sum_{\alpha \in D
_{n}} {
\int_{T_n(t)}\prod_{i=1}^n
(\kappa(s_{i+1}-s_{i}))^{-\frac{1}{2}(1+\alpha_i)}ds.}\\
&&+C n! c_H^n \sum_{\alpha \in D
_{n}} {
\int_{T_n(t)}(\kappa(s_{2}-s_{1}))^{-\frac{2-2H+\alpha_1}{2}}\prod_{i=2}^n
(\kappa(s_{i+1}-s_{i}))^{-\frac{1}{2}(1+\alpha_i)}ds.}
\end{eqnarray*}
We observe that whenever $\frac{1}{4}< H < \frac{1}{2}$, we have $\frac12(1+\al_{i})<1$ for all $i=2,\ldots n$, and it is easy to see that $\alpha_1$ is at most $1-2H$ so $\frac{1}{2}(2-2H+\alpha_1)<1$. Thanks to Lemma \ref{lem:intg-simplex} and recalling that $\sum_{i=1}^n\alpha_i = (n-1)(1-2H)$ for all $\al\in D_{n}$, we thus conclude that
\begin{equation*}
L_{n,t}(\zeta)
\leq\frac{C (t^{H-\frac{1}{2}}\kappa^{H-\frac{1}{2}}+|\zeta|^{1-2H})n! c_H^n t^{nH}\kappa^{nH-n}}{\Gamma(nH+1)}\,.
\end{equation*}
Plugging this expression into \eqref{eq:bnd-fn-L2-1}, we end up with
\begin{equation}\label{eq:bnd-H-norm-fn}
n!\| f_n(\cdot,t,x)\|_{\HH^{\otimes n}}^2
\leq
\frac{C c_H^n t^{nH}\kappa^{nH-n}}{\Gamma(nH+1)}\left(\int_{\RR}(t^{H-\frac{1}{2}}\kappa^{H-\frac{1}{2}}+|\zeta|^{\frac{1}{2}-H})\left| \mathcal{F}u_0(\zeta)\right| d\zeta\right)^2.
\end{equation}
The proof of \eqref{chaos} is now easily completed thanks to the asymptotic behavior of the Gamma
function and our assumption of $u_0$, and this finishes the existence and uniqueness proof.
\end{proof}

\subsection{Moment bounds}
\label{sec:Anderson.momentbounds}

In this section we derive the upper and lower bounds for the moments of the solution to equation \eref{spde} which allow us to conclude on the intermittency of the solution. We proceed by first getting an approximation result for $u$, and then deriving the upper and lower bounds for the approximation.
\subsubsection{Approximation of the solution}
The approximation of the solution we consider is based on an approximation of the noise $W$, which is defined in \eref{eq:cov-W-epsilon}.
%\begin{equation}\label{eq:cov-W-epsilon}
%W^{\varepsilon}(\varphi)
%= \int_0^t \int_{\mathbb{R}} [p_{\ep}*\varphi](s,x)W(ds,dy)
%=\int_0^t \int_{\mathbb{R}}\int_{\mathbb{R}}\varphi(s,x)p_{\varepsilon}(x-y)W(ds,dy)dx\,,
%\end{equation}
%where we recall that $p_{t}$ stands for the heat semi-group in $\R$. Notice that relation \eqref{eq:cov-W-epsilon} can be also read (either in Fourier or direct coordinates) as:
%\begin{eqnarray*}
%\be\left[W^{\varepsilon}(\varphi) W^{\varepsilon}(\psi) \right]
%&=&
%c_H \int_0^t \int_{\mathbb{R}}
%\mathcal{F}\varphi(s,\xi)\, \overline{\mathcal{F}\psi(s,\xi)} \, e^{-\varepsilon |\xi|^2} |\xi|^{1-2H} d\xi ds  \\
%&=& c_H \int_0^t \int_{\mathbb{R}}\int_{\mathbb{R}}\varphi(s,x)f_{\varepsilon}(x-y)\psi(s,y) \, dx   dy   ds,
%\end{eqnarray*}
%where $f_{\ep}$ is given by: $f_{\varepsilon}(x)=\mathcal{F}^{-1}(e^{-\varepsilon |\xi|^2} |\xi|^{1-2H})$. In other words, our noise is still a white noise in time but its space covariance is now dictated by $f_{\ep}$. Note that $f_{\varepsilon}$ is a  {real} positive definite function, but is not necessarily positive.
The noise $W_{\ep}$ induces an approximation to the mild formulation of equation \eref{spde}, namely
\begin{equation}\label{appr eq}
u_{\ep}(t,x)=p_t u_0(x) + \int_0^t \int_{\mathbb{R}}p_{t-s}(x-y)u_{\ep}(s,y) \, W_{\ep}(ds,dy) ,
\end{equation}
where the integral is understood (as in Section \ref{sec:picard}) in the It\^o sense. We will start by a formula for the moments of $u_{\ep}$.

\begin{proposition}\label{prop:appro-moments}
Let $W_{\ep}$ be the noise defined by \eqref{eq:cov-W-epsilon}, and
assume $\frac{1}{4}<H<\frac{1}{2}$.
Assume $u_0$ is such that
$\int_{\RR}(1+|\xi|^{\frac{1}{2}-H})|\mathcal{F}u_0(\xi)|d\xi<
\infty$.  Then

\noindent
\emph{(i)} Equation \eqref{appr eq} admits a unique solution.

\noindent
\emph{(ii)} For any integer $n \geq 2$ and $(t,x)\in\ott\times\R$, we have
\begin{equation}\label{appro moment}
\be \left[ u^n_{\varepsilon}(t,x)\right] = { \be_B
\left[\prod_{j=1}^n u_0(x+B_{\kappa t}^j) \exp \left( c_{1,H} \sum_{1\leq j\neq k
\leq n} V_{t,x}^{\ep,j,k}\right)\right],}
\end{equation}
with
\begin{equation}\label{eq:def-V-tx-epsilon}
V_{t,x}^{\ep,j,k}
=
\int_0^t f_{\varepsilon}(B_{ \kappa r}^j-B_{\kappa r}^k)dr
=
\int_0^t \int_{\mathbb{R}}e^{-\varepsilon |\xi|^2} |\xi|^{1-2H} e^{i\xi (B_{\kappa r}^j-B_{\kappa r}^k)} \, d\xi dr .
\end{equation}
In formula \eqref{eq:def-V-tx-epsilon}, $\{ B^j; j=1,\dots,n\}$ is a family of $n$ independent standard Brownian motions  which are also independent of $W$ and $\be_{B}$ denotes  the expected value with respect to the randomness in $B$ only.

\noindent
\emph{(iii)} The quantity  $\be [( u_{\varepsilon}(t,x))^n]$ is uniformly bounded in $\ep$. More generally, for any $a>0$ we have
\begin{equation*}
\sup_{\ep>0}
\be_B \left[ \exp \left( a \sum_{1\leq j\neq k \leq n} V_{t,x}^{\ep,j,k}\right)\right] .
\equiv c_{a}<\infty
\end{equation*}
\end{proposition}
\begin{proof}
The proof of item (i) is almost identical to the proof of Theorem~\ref{thm:exist-uniq-chaos}, and is omitted for sake of conciseness. Moreover, in the proof of (ii) and (iii), we may take $u_0(x)\equiv 1$ for simplicity.

In order to check item (ii), set
\begin{equation}\label{eq:def-Atx}
A_{t,x}^{\varepsilon}(r,y)=
\rho_{\varepsilon}(B_{\kappa (t-r)}^x-y),
\quad\text{and}\quad
\alpha^{\varepsilon}_{t,x}=\|A^{\varepsilon}_{t,x}\|^2_{\HH}.
\end{equation}
Then one can prove, similarly to Proposition 5.2 in \cite{HN}, that $u_{\ep}$ admits a Feynman-Kac representation of the form
\begin{equation}\label{eq:feynman-u-ep}
u_{\varepsilon}(t,x)=\be_B \lc \exp \lp  W (
A_{t,x}^{\varepsilon})-\frac{1}{2}\alpha^{\varepsilon}_{t,x}\rp
\rc\,.
\end{equation}
Now fix an integer $n \geq 2$. According to \eqref{eq:feynman-u-ep} we have
\begin{equation*}
\be \lc  u^n_{\varepsilon}(t,x)\rc=\be_W \lc\prod_{j=1}^n
\be_B\lc \exp \lp   W(A^{\varepsilon,
B^j}_{t,x})-
\frac{1}{2}\alpha_{t,x}^{\varepsilon,B^j}\rp \rc \rc\,,
\end{equation*}
where for any $j=1,\dots,n$,  $A_{t,x}^{\varepsilon,B^j}$ and $\alpha_{t,x}^{\varepsilon,B^j}$ are evaluations of  \eqref{eq:def-Atx} using the Brownian motion $B^j$. Therefore,  since $W(A^{\varepsilon, B^j}_{t,x})$ is a Gaussian random variable conditionally on $B$, we obtain
\begin{eqnarray*}%\label{eq:exp-moments-utx-ep-delta}
\be \lc  u^n_{\varepsilon}(t,x)\rc&=&
\be_B \lc
\exp \lp\frac{1}{2}\|\sum_{j=1}^n A_{t,x}^{\varepsilon,B^j}\|^2_{\HH}
-\frac{1}{2}\sum_{j=1}^n \alpha_{t,x}^{\varepsilon,B^j}\rp\rc \notag\\
&=& \be_B \lc
\exp \lp\frac{1}{2}\|\sum_{j=1}^n A_{t,x}^{\varepsilon,B^j}\|^2_{\HH}
-\frac{1}{2}\sum_{j=1}^n \| A_{t,x}^{\varepsilon,B^j}\|^2_{\HH}\rp\rc   \notag\\
&=&\be_B \lc \exp \lp\sum_{1\leq i < j \leq n}\langle
A_{t,x}^{\varepsilon,B^i},
A_{t,x}^{\varepsilon,B^j}\rangle _{\HH}\rp\rc\,.
\end{eqnarray*}
The evaluation of $\langle A_{t,x}^{\varepsilon,B^i}, A_{t,x}^{\varepsilon,B^j}\rangle _{\HH}$ easily yields our claim \eqref{appro moment}, the last details being left to the patient reader.

Let us now prove item (iii), namely
\begin{equation}\label{appro moment finite}
\sup_{\varepsilon > 0} \sup_{t \in [0,T], x \in \mathbb{R}}
\be \lc  u^n_{\varepsilon}(t,x)\rc < \infty\,.
\end{equation}
To this aim, observe first that we have obtained an expression \eqref{appro moment} which does not depend on $x\in\R$, so that the $\sup_{t \in [0,T], x \in \mathbb{R}}$ in \eqref{appro moment finite} can be reduced to a $\sup$ in $t$ only. Next, still resorting to formula \eqref{appro moment}, it is readily seen that it suffices to show that for two independent Brownian motions $B$ and $\tilde{B}$, we have
\begin{equation}\label{eq:bnd-exp-F-t-epsilon}
\sup_{\varepsilon > 0, t\in [0,T]} \be_{B} \lc \exp \left (c \, F_t^{\varepsilon} \right)\rc <\infty,
\quad\text{with}\quad
F_t^{\varepsilon} \equiv
\int_0^t \int_{\mathbb{R}} e^{-\varepsilon |\xi|^2} |\xi|^{1-2H} e^{i \xi (B_{\kappa r}-\tilde{B}_{\kappa r})}d\xi dr,
\end{equation}
for any positive constant $c$.  In order to prove \eqref{eq:bnd-exp-F-t-epsilon}, we expand the exponential and write:
\begin{equation}\label{eq:moments-F-t-epsilon}
\be_{B} \lc \exp (c \, F_t^{\varepsilon})\rc
=\sum_{l=0}^{\infty}\frac{\be_{B} \lc (c \, F_t^{\varepsilon})^l\rc}{l!}\,.
\end{equation}
Next, we have
\begin{align*}
\be_{B} \lc\left( F_t^{\varepsilon}\right)^l\rc&=
\be_{B} \lc \int_{[0,t]^l} \int_{\RR^l}
\prod_{j=1}^l  e^{-i  \xi_j (B_{\kappa r_j}-\tilde{B}_{\kappa r_j})-\varepsilon  |\xi_j|^2} |\xi_j|^{1-2H} d\xi dr \rc \\
&\leq
\int_{[0,t]^l} \int_{\RR^l}
\prod_{j=1}^{l} e^{-\kappa (t-r_{\sigma(l)})|\xi_l+\dots+\xi_1|^2} \, |\xi_j|^{1-2H} \, d\xi dr\,,
\end{align*}
where $\sigma$ is the permutation on $\{1,2,\dots, l\}$ such that $t \geq r_{\sigma(l)} \geq \cdots \geq r_{\sigma(1)}$. We have thus gone back to an expression which is very similar to \eqref{eq:expression-norm-fn}. We now proceed as in the proof of Theorem \ref{thm:exist-uniq-chaos} to show that \eref{appro moment finite} holds true from equation \eqref{eq:moments-F-t-epsilon}.
\end{proof}

Starting from Proposition \ref{prop:appro-moments}, let us take limits in order to get the moment formula for the solution $u$ to equation~\eqref{spde}.

\begin{theorem}\label{THM moment}
Assume $\frac{1}{4}<H<\frac{1}{2}$ and  consider $n\ge 1$, $j,k\in\{1,\ldots,n\}$ with $j\ne k$.
For $(t,x)\in\ott\times\R$,  denote by   $V_{t,x}^{j,k}$  the limit  in $L^2(\Omega)$  as $\ep\rightarrow 0$  of
\begin{equation*}%
%V_{t,x}^{j,k}
%=
%{\hbox{\tiny $ L^{2}(\oom)-$}}
%\lim_{\ep\to 0} V_{t,x}^{\ep,j,k} (\hbox{the limit is taken in $L^2(\Omega)$}),
%\quad\text{with}\quad
V_{t,x}^{\ep,j,k}
=
\int_0^t \int_{\mathbb{R}}e^{-\varepsilon |\xi|^2} |\xi|^{1-2H} e^{i\xi (B_{ \kappa r}^j-B_{\kappa r}^k)}d\xi dr.
\end{equation*}
Then  $\be \lc  u^n_{\varepsilon}(t,x)\rc$ converges as $\varepsilon \to 0$ to $\be [u^n(t,x)]$, which is given by
\begin{equation}\label{moment}
\be[u^n(t,x)] ={ \be_{B}\left[ \prod_{j=1}^n u_0(B^j_{\kappa t}+x)\exp \left(
c_{1,H} \sum_{1\leq j \neq k \leq n} V_{t,x}^{j,k} \right)\right]\, .}
\end{equation}
\end{theorem}

\begin{proof}

As in Proposition \ref{prop:appro-moments}, we will prove the theorem for $u_0 \equiv 1$ for simplicity. For
any $p\ge 1$ and $1\le j < k \le n$, we can easily prove that
$V_{t,x}^{\ep,j,k}$ converges in $L^{p}(\oom)$ to $V_{t,x}^{j,k}$
defined by
\begin{equation}\label{eq:def-Vtx-jk}
V_{t,x}^{j,k}
=
\int_0^t \int_{\mathbb{R}} |\xi|^{1-2H} e^{i\xi (B_{\kappa r}^j-B_{\kappa r}^k)}d\xi dr.
\end{equation}
Indeed, this is due to the fact that $e^{-\varepsilon |\xi|^2} |\xi|^{1-2H} e^{i\xi (B_{\kappa r}^j-B_{\kappa r}^k)}$ converges to $|\xi|^{1-2H} e^{i\xi (B_{\kappa r}^j-B_{\kappa r}^k)}$ in the $d\xi\otimes dr\otimes d\bp$ sense, plus standard uniform integrability arguments. Now, taking into account relation \eqref{appro moment}, Proposition \ref{prop:appro-moments} and the fact that $
%L^{2}(\oom)-\lim_{\ep\to 0}
V_{t,x}^{\ep,j,k}$ converges to $V_{t,x}^{j,k}$ in $L^{2}(\oom)$ as $\ep\to 0$, we obtain
\begin{eqnarray}\label{eq:lim-moments-u-epsilon}
\lim_{\ep\to 0} \be \lc  u^n_{\varepsilon}(t,x)\rc
&=&
\lim_{\ep\to 0} \be_B \left[ \exp \left( c_{1,H} \sum_{1\leq j\neq k \leq n} V_{t,x}^{\ep,j,k}\right)\right]  \notag \\
&=&
\be_B \left[ \exp \left( c_{1,H} \sum_{1\leq j\neq k \leq n} V_{t,x}^{j,k}\right)\right].
\end{eqnarray}

To end the proof, let us now identify the right hand side of \eqref{eq:lim-moments-u-epsilon} with $\be [u^n(t,x)]$, where $u$ is the solution to equation \eqref{spde}. For $\ep,\ep'>0$ we write
\[
\be \lc  u_{\varepsilon}(t,x) \, u_{\varepsilon'}(t,x)  \rc=
\be_B \lc  \exp \lp\ \langle A^{\varepsilon,B^1}_{t,x} , A^{\varepsilon',
B^2}_{t,x}  \rangle_{\HH}\rp\rc\, ,
\]
where we recall that $A^{\varepsilon,B}_{t,x}$ is defined by
relation \eqref{eq:def-Atx}. As before we can show that this
converges as $\varepsilon, \varepsilon'$ tend to zero. So,
$u_{\varepsilon}(t,x)$ converges in $L^2$ to some limit $v(t,x)$,
and the limit is actually in  $L^p$ , for all $p \geq 1$. Moreover,
$\be [v^k(t,x)]$ is equal  to the right hand side
of~\eref{eq:lim-moments-u-epsilon}.   {Finally,  for any smooth random
variable $F$ which is a linear combination of $W({\bf
1}_{[a,b]}(s)\varphi(x))$, where $\varphi$ is a $C^{\infty}$
function with compact support,  using }the fact that It\^o's and
Skorohod's integrals coincide on the set $\laa_{H}$, plus the
duality relation \eqref{dual}, we have
\begin{equation}\label{eq:duality-u-varepsilon}
{  \be \lc F u_{\varepsilon}(t,x)\rc
 =\be \lc F\rc+\be \lc \langle  Y^{\ep} ,DF\rangle _{\HH}\rc,}
\end{equation}
where
\begin{equation*}
Y^{t,x}({s,z})= \lp \int_{\mathbb{R}}
p_{t-s}(x-y) \, p_{\varepsilon}(y-z) u_{\varepsilon} (s,y)\, dy \rp \1_{\ot}(s) .
\end{equation*}
Letting $\varepsilon$ tend to zero in equation
\eref{eq:duality-u-varepsilon}, after some easy calculation we
get
\begin{equation*}
\be [F v_{t,x}]= \be[ F]  +\be \lc \langle DF, v
p_{t-\cdot}(x-\cdot)\rangle_{\HH}\rc\,.
\end{equation*}
This equation is valid for any $F \in \mathbb{D}^{1,2}$ by
approximation. So the above equation implies that the process $v$
is the solution of equation \eqref{spde}, and by the uniqueness of
the solution we have $v=u$.
\end{proof}

\subsubsection{Intermittency estimates}
In this section we prove some upper and lower bounds on the moments of the solution which entail the intermittency phenomenon.

\begin{theorem}\label{thm:intermittency-estimates}
Let $\frac{1}{4}<H<\frac{1}{2}$, and consider the solution $u$ to
equation \eqref{spde}. For simplicity we assume that the initial condition is $u_0(x)\equiv 1$.  Let $n \geq 2$ be an integer, $x\in\R$ and
$t\ge 0$. Then there  exist
some positive  constants $c_{1},c_{2},c_{3}$ independent of
$n$, $t$ and $\kappa$ with
$0<c_{1}<c_{2}$ satisfying
\begin{equation}\label{eq:intermittency-bounds}
\exp (c_{1} n^{1+\frac{1}{H}}\kappa^{1-\frac{1}{H}}t)
\leq \be\lc u^n(t,x) \rc
\leq c_{3} \exp \big(c_{2} n^{1+\frac{1}{H}}\kappa^{1-\frac{1}{H}} t\big)\,.
\end{equation}
\end{theorem}

\begin{remark}
{Observe that the upper bound in \eqref{eq:intermittency-bounds} has already been proven for a general coefficient $\si$ (see \eqref{eq:upp-bnd-Lp-u-sigma-general}). }
\end{remark}

\begin{proof}[Proof of Theorem \ref{thm:intermittency-estimates}]
We divide this proof into upper and lower bound estimates.

\noindent \textit{Step 1: Upper bound.} Recall from equation
\eqref{eq:chaos-expansion-u(tx)} that for $(t,x)\in\R_{+}\times\R$,
$u(t,x)$ can be written as:
$u(t,x)=\sum_{m=0}^{\infty}I_m(f_m(\cdot,t,x))$. Moreover, as a
consequence of the hypercontractivity property on a fixed chaos we
have (see \cite[p. 62]{Nua})
\begin{equation*}
\|I_m(f_m(\cdot,t,x))\|_{L^{n}(\oom)}\leq
(n-1)^{\frac{m}{2}}\|I_m(f_m(\cdot,t,x))\|_{L^{2}(\oom)} \,,
\end{equation*}
and substituting  the above right hand side by the bound \eqref{eq:bnd-H-norm-fn}, we end up with
\begin{eqnarray*}
\|I_m(f_m(\cdot,t,x))\|_{L^{n}(\oom)}
\leq
n^{\frac{m}{2}}\|I_m(f_m(\cdot,t,x))\|_{L^{2}(\oom)}
\leq
\frac{c^{\frac{n}{2}}n^{\frac{m}{2}}t^{\frac{mH}{2}}\kappa^{\frac{Hm-m}{2}}}
{\big[\Gamma(m
H+1)\big]^{\frac{1}{2}}}\,.
\end{eqnarray*}
Therefore recalling again the elementary bound $\sum_{n\ge 0}x^{n}/(n!)^{a}\le 2 \exp(c x^{1/a})$,  we get:
\begin{eqnarray*}
\|u(t,x)\|_{L^{n}(\oom)}
\leq
\sum_{m=0}^{\infty} \|J_m(t,x)\|_{L^{n}(\oom)}
\leq
\sum_{m=0}^{\infty}\frac{c^{\frac{m}{2}}n^{\frac{m}{2}}t^{\frac{mH}{2}}\kappa^{\frac{Hm-m}{2}}}{\big(\Gamma(m
H+1)\big)^{\frac{1}{2}}}\leq c_{1}\exp {\big(c_{2} t n^{\frac{1}{H}} \kappa^{\frac{H-1}{H}}\big)}\,,
\end{eqnarray*}
from which the upper bound in our theorem is easily deduced.

\noindent
\textit{Step 2: Lower bound for $u_{\ep}$.}
For the lower bound, we start from the moment formula \eref{appro moment} for the approximate solution, and write
\begin{multline*}
\be \lc  u^n_{\varepsilon}(t,x)\rc \\
=
\be_{B} \lc \exp \left(c_{1,H}\left[ \int_0^t \int_{\mathbb{R}} e^{-\varepsilon |\xi|^2}
\left| \sum_{j=1}^n e^{-i B_{\kappa r}^j \xi}\right|^2 |\xi|^{1-2H} d\xi dr
-nt \int_{\mathbb{R}} e^{-\varepsilon |\xi|^2} |\xi|^{1-2H} d\xi\right] \right)\rc.
\end{multline*}
In order to estimate the expression above, notice first that the
obvious change of  variable $\la= \ep^{1/2}\xi$ yields
$\int_{\mathbb{R}} e^{-\varepsilon |\xi|^2} |\xi|^{1-2H} d\xi=C
\ep^{-(1-H)}$ for some constant $C$.  Now for an additional arbitrary parameter $\eta>0$,
consider the set
\begin{equation*}
A_\eta=\left\{\om; \, \sup_{1\leq j\leq n}\sup_{0\leq r \leq t}|B_{\kappa r}^{j}(\om)|\leq
\frac{\pi}{3\eta}\right\}.
\end{equation*}
Observe that classical small balls inequalities for a Brownian motion (see (1.3) in \cite{LS}) yield $\bp(A_{\eta})\geq c_{1} e^{-c_{2}  \eta^2 n \kappa t}$ for a large enough $\eta$. In addition, if we assume that $A_{\eta}$ is realized and $|\xi|\le\eta$, some elementary trigonometric identities show that the following deterministic bound hold true: $| \sum_{j=1}^n e^{-i B_{\kappa r}^j \xi}| \ge \frac{n}{2}$.
Gathering those considerations, we thus get
\begin{align*}
\be \lc  u^n_{\varepsilon}(t,x)\rc
&\geq
\exp \left( c_1 n^2 \int_0^t \int_0^{\eta} e^{-\varepsilon |\xi|^2} |\xi|^{1-2H} d\xi dr - c_2 nt \varepsilon^{H-1} \right)
\bp\lp A_\eta \rp \\
&\geq C
\exp \left( c_1 n^2 t  \varepsilon^{-(1-H)} \int_0^{\ep^{1/2}\eta} e^{- |\xi|^2} |\xi|^{1-2H} d\xi  - c_2 nt \varepsilon^{-(1-H)} - c_{3} n \kappa t \eta^{2} \right).
\end{align*}
We now choose the parameter $\eta$ such that $\kappa \eta^2=\varepsilon^{-(1-H)}$, which means in particular that $\eta \to \infty$ as $\varepsilon \to 0$. It is then easily seen that $\int_0^{\ep^{1/2}\eta} e^{- |\xi|^2} |\xi|^{1-2H} d\xi$ is of order $\ep^{H(1-H)}$ in this regime, and some elementary algebraic manipulations entail
\begin{equation*}
\be \lc  u^n_{\varepsilon}(t,x)\rc
\geq C
\exp \left( c_1 n^2 t \kappa^{H-1}\varepsilon^{-(1-H)^2} -c_2 nt\varepsilon^{-(1-H)}\right)
\geq C \exp \left(c_{3} t \kappa^{1-\frac{1}{H}}n^{1+\frac{1}{H}}\right),
\end{equation*}
where the last inequality is obtained by choosing $\varepsilon^{-(1-H)}=c \, \kappa ^{\frac{H-1}{H}}n^{\frac{1}{H}}$ in order to optimize the second expression. We have thus reached the desired lower bound in \eqref{eq:intermittency-bounds} for the approximation $u^{\ep}$ in the regime $\varepsilon=c \, \kappa ^{\frac{1}{H}}n^{-\frac{1}{H(1-H)}}$.

\noindent
\textit{Step 3: Lower bound for $u$.}
To complete the proof, we need to show that for all sufficiently small $\varepsilon$, $\be \lc  u^n_{\varepsilon}(t,x)\rc\leq \be[u^n(t,x)]$. We thus start from equation \eref{appro moment} and use the series expansion of the exponential function as in \eqref{eq:moments-F-t-epsilon}. We get
\begin{equation}\label{eq:expansion-moment-u-epsilon}
\be \lc  u^n_{\varepsilon}(t,x)\rc=
\sum_{m=0}^{\infty} \frac{c_H^m}{m!}  \,
\be _{B} \!\lc \left( \sum_{1\leq j \neq k \leq n} V_{t,x}^{\ep,j,k} \right)^m \rc,
\end{equation}
where we recall that $V_{t,x}^{\ep,j,k}$ is defined by \eqref{eq:def-V-tx-epsilon}. Furthermore, expanding the $m$th power above, we have
\begin{equation*}
\be_{B} \!\lc \left( \sum_{1\leq j \neq k \leq n} V_{t,x}^{\ep,j,k} \right)^m \rc
=
\sum_{\al\in K_{n,m}}  \int_{[0,t]^m} \int_{\mathbb{R}^m}
e^{-\varepsilon \sum_{l=1}^m |\xi_l|^2} \be_{B} \lc e^{i B^{\al}(\xi)} \rc
\prod_{l=1}^m |\xi_l|^{1-2H} \, d\xi dr\,,
\end{equation*}
where $K_{n,m}$ is a set of multi-indices defined by
\begin{equation*}
K_{n,m}=
\lcl
\al=(j_{1},\ldots,j_{m},k_{1},\ldots,k_{m}) \in \{1,\ldots,n\}^{2m} ; \,
j_{l}<k_{l} \text{ for all } l=1,\ldots,n
\rcl,
\end{equation*}
and $B^{\al}(\xi)$ is a shorthand for the linear combination $\sum_{l=1}^m \xi_{l}(B_{\kappa r_{l}}^{j_{l}}-B_{\kappa r_{l}}^{k_{l}})$. The important point here is that $E _{B} e^{iB^{\al}(\xi)}$ is positive for any $\al\in K_{n,m}$. We thus get the following inequality, valid for all $m\ge 1$
\begin{eqnarray*}
\be _{B} \!\lc \left( \sum_{1\leq j \neq k \leq n} V_{t,x}^{\ep,j,k} \right)^m \rc
&\le&
\sum_{\al\in K_{n,m}}  \int_{[0,t]^m} \int_{\mathbb{R}^m}
\be _{B} \lc e^{i B^{\al}(\xi)} \rc
\prod_{l=1}^m |\xi_l|^{1-2H} \, d\xi dr \\
&=&
\be _{B} \!\lc \left( \sum_{1\leq j \neq k \leq n} V_{t,x}^{j,k} \right)^m \rc,
\end{eqnarray*}
where $V_{t,x}^{j,k}$ is defined by \eqref{eq:def-Vtx-jk}. Plugging this inequality back into \eqref{eq:expansion-moment-u-epsilon} and recalling expression \eqref{moment} for $\be [u^n(t,x)]$, we easily deduce that $\be [(u^n_{\ep}(t,x)] \le \be[u^n(t,x)]$, which finishes  the proof.
\end{proof}

\end{document}